\documentclass[a4paper,10pt]{article}

\usepackage[english]{babel}
\usepackage[utf8]{inputenc}

\usepackage{amsmath}
\usepackage{amssymb}
\usepackage{amsthm}
\usepackage{mathtools}
\usepackage{mathrsfs}
\usepackage{abraces}
\usepackage{bbm}

\usepackage{tikz}
\usetikzlibrary{cd,patterns,positioning,decorations.markings,decorations.pathmorphing}
\usepackage{xstring}
\usepackage{ifthen}
\usepackage{pdflscape}
\usepackage{diagbox}

\usepackage[autostyle]{csquotes}
\usepackage[backend=biber,style=trad-abbrv,doi=false,isbn=false,url=false]{biblatex}
\usepackage{titlesec}
\titleformat{\section}{\normalfont\fillast\scshape}{\thesection.}{.5em}{}
\titleformat{\subsection}[runin]{\normalfont\bfseries}{\thesubsection.}{.5em}{}[.]

\usepackage{enumitem}

\setlist[description]{%
	font={\normalfont\itshape},
	}

\newtheorem{thm}{Theorem}[subsection]
\newtheorem{lem}[thm]{Lemma}
\newtheorem{pro}[thm]{Proposition}
\newtheorem{cor}[thm]{Corollary}
\theoremstyle{definition}
\newtheorem{dfn}[thm]{Definition}
\newtheorem{rmk}[thm]{Remark}
\newtheorem{exa}[thm]{Example}

\DeclareMathOperator{\Hom}{Hom}

\DeclareMathOperator{\id}{id}
\DeclareMathOperator{\JW}{JW}

\newcommand{\cohdeg}[1]{\scriptstyle\textcolor{green}{#1}}

\newcommand{\Cdg}{\mathcal{C}_{\mathrm{dg}}}

\newcommand{\hh}{\mathfrak{h}}
\newcommand{\HH}{\mathbf{H}}
\newcommand{\ZZ}{\mathbb{Z}}
\newcommand{\roo}{\ZZ\Phi}
\newcommand{\D}{\mathscr{H}}
\newcommand{\DBS}{\D_{\mathrm{BS}}}

\newcommand{\Da}{{}^I\! \D}

\newcommand{\Gammaa}{{}^I\Gamma}

\newcommand{\Rshr}{\mathcal{L}_s}
\newcommand{\Bshr}{\mathcal{L}_t}
\newcommand{\Kb}{\mathcal{K}^{\mathrm{b}}}

\newcommand{\Homb}{\Hom^\bullet}
\newcommand{\un}{\mathbbm{1}}
\newcommand{\boe}{\mathbf{e}}

\newcommand{\uw}{\underline{w}}

\newcommand{\uu}{\underline{u}}
\newcommand{\uv}{\underline{v}}
\newcommand{\ux}{\underline{x}}

\newcommand{\uom}{\underline{\omega}}

\newcommand{\kk}{\Bbbk}
\newcommand{\us}{\underline{\mathbf{s}}}
\newcommand{\ut}{\underline{\mathbf{t}}}

\newcommand{\qn}[2]{[#1]_{#2}}
\newcommand{\qbc}[3]{\begin{bsmallmatrix}
						#1 \\ #2
						\end{bsmallmatrix}_{#3}}
\newcommand{\Wak}[1]{\Theta_{#1}}
\newcommand{\Wakr}[1]{\tilde{\Theta}_{#1}}
\newcommand{\Roudg}[1]{F_{#1}^\bullet}
\newcommand{\ud}[1]{z(#1)}

\newcommand{\Ra}{\kk}
\newcommand{\Casph}{{}^I\! C_n}
\newcommand{\Cn}{\tilde{C}_n}
\newcommand{\ropen}{\textcolor{red}{(}}
\newcommand{\rclosed}{\textcolor{red}{)}}
\newcommand{\bopen}{\textcolor{blue}{(}}
\newcommand{\bclosed}{\textcolor{blue}{)}}
\newcommand{\rcomma}{\textcolor{red}{|}}
\newcommand{\bcomma}{\textcolor{blue}{|}}

\newcommand{\rdot}{\textcolor{red}{\bullet}}
\newcommand{\bdot}{\textcolor{blue}{\bullet}}

\newcommand{\Laur}{\ZZ[v,v^{-1}]}

\newcommand{\alldJW}{\omega}
\newcommand{\cword}{S^*}
\newcommand{\bword}{\Sigma^*}

\newcommand{\Wf}{W_{\mathrm{f}}}
\newcommand{\chr}{\mathrm{char}}

\newcommand{\re}[1]{\textcolor{red}{#1}}
\newcommand{\bl}[1]{\textcolor{blue}{#1}}

\newcommand{\Tate}[1]{\langle #1 \rangle}

\newcommand{\tibe}{\tilde{\beta}}

\newcommand{\JWbox}[5]{ 
					 \draw (#4,#5)--({#4+0.36*(#1+2*#2)},#5)--({#4+0.36*(#1+2*#2)},#5+0.6)--(#4,#5+0.6)--cycle;
					 \node at ({#4+0.18*(#1+2*#2)},#5+0.3) (JW) {$JW_{#3}$};
					} 

\newcommand{\longstrand}[2]{ 
					   \draw[#2] ({0.18+(#1-1)*0.36},-0.6)-- ++(0,1.8);
					  }
\newcommand{\upstrand}[4]{ 
					   \draw[#4] ({#1+0.18+(#3-1)*0.36},{#2+0.6})-- ++(0,0.3);
					  }
\newcommand{\downstrand}[2]{ 
					   \draw[#2] ({0.18+(#1-1)*0.36},0)-- ++(0,-0.3);
					  }
\newcommand{\uplongstrand}[2]{ 
					   \draw[#2] ({0.18+(#1-1)*0.36},0.6)-- ++(0,0.6);
					  }
\newcommand{\downlongstrand}[2]{ 
					   \draw[#2] ({0.18+(#1-1)*0.36},0)-- ++(0,-0.6);
					  }
\newcommand{\updots}[1]{ 
					   \foreach \i in {0,1,2}{
	 				    \fill ({0.18+(#1-1)*0.36+\i*0.18},0.75) circle (.5pt);
	 				   }
					  }
\newcommand{\downdots}[1]{ 
					   \foreach \i in {0,1,2}{
	 				    \fill ({0.18+(#1-1)*0.36+\i*0.18},-0.15) circle (.5pt);
	 				   }
					  }
\newcommand{\updot}[2]{ 
					   \fill[#2] ({0.18+(#1-1)*0.36},0.9)circle (1pt);
					  }
\newcommand{\downdot}[2]{ 
					   \fill[#2] ({0.18+(#1-1)*0.36},-0.3)circle (1pt);
					  }
\newcommand{\doubleupdot}[2]{ 
					   \fill[#2] ({0.18+(#1-1)*0.36},1.2)circle (1pt);
					  }
\newcommand{\doubledowndot}[2]{ 
					   \fill[#2] ({0.18+(#1-1)*0.36},-0.6)circle (1pt);
					  }
\newcommand{\uppitchfork}[3]{ 
					   \draw[#2] ({0.18+(#1-1)*0.36},0.9)-- ++(0,-0.3);					   
					   \draw[#2] ({0.18+(#1-1)*0.36},0.9)-- ++(0.18,0)-- ++(0,0.3);
					   \draw[#2] ({0.18+(#1-1)*0.36},0.9)-- ++(-0.18,0)-- ++(0,0.3);
					   \draw[#3] ({0.18+(#1-1)*0.36},1)-- ++(0,0.2);
					   \fill[#3] ({0.18+(#1-1)*0.36},1) circle (1pt);
					   
					  }
\newcommand{\upleftpitchfork}[2]{ 
 								 \draw[#1] (0.18,0.9)-- ++(0,0.3);
 								 \draw[#1] (0.18,0.9)-- ++(0,-0.3);
 								 \draw[#1] (0.18,0.9)-- ++(-0.36,0)-- ++(0,0.3);															 \draw[#1] ({0.18-0.36},0.9)-- ++(0,-1.5);
 								 \draw[#2] (0,1)-- ++(0,0.2);\fill[#2] (0,1) circle (1pt);
								}
\newcommand{\uprightpitchfork}[3]{ 
 								 \draw[#1] ({0.18+(#3-2)*0.36},0.6)-- ++(0,0.3)
 								 -- ++(0.72,0)-- ++(0,-1.5);
 								 \draw[#1] ({0.9+(#3-2)*0.36},0.9)-- ++(0,0.3);
 								 \draw[#2] ({0.18+(#3-1)*0.36},0.6)-- ++(0,0.2);
 								 \fill[#2] ({0.18+(#3-1)*0.36},0.8) circle (1pt);
								}					  
\newcommand{\uprightdoublepitchfork}[3]{ 
 								 \draw[#1] ({0.18+(#3-2)*0.36},0.6)-- ++(0,0.3)
 								 -- ++(0.72,0)-- ++(0,-1.5);
 								 \draw[#1] ({0.9+(#3-2)*0.36},0.9)-- ++(0,0.3);
 								 \draw[#1] ({0.18+(#3-2)*0.36},0.9)-- ++(0,0.3);
 								 \draw[#2] ({0.18+(#3-1)*0.36},0.6)-- ++(0,0.2);
 								 \fill[#2] ({0.18+(#3-1)*0.36},0.8) circle (1pt);
 								 \draw[#2] ({0.18+(#3-1)*0.36},1)-- ++(0,0.2);
								 \fill[#2] ({0.18+(#3-1)*0.36},1) circle (1pt);
								}					  
\newcommand{\uppernode}[2]{
						   \node[anchor=south west, rotate=45] at ({0.18+(#1-1)*0.36},1.2) (color) {\scriptsize #2};
						  }

\renewcommand\labelenumi{(\textit{\roman{enumi}})}
\renewcommand\theenumi\labelenumi

\definecolor{green}{rgb}{0,.5,0}

\tikzset{%
	virtual/.style = {font=\scriptsize,draw=black,circle,dashed,inner sep=2pt},
	real/.style    = {circle,inner sep=2pt,font=\scriptsize},
	patch/.style   = {gray,pattern=north west lines, pattern color=gray},
	negative/.style = {line width=.25pt,double distance=1.5pt},
	positive/.style = {line width=2pt},
	braid/.style = {thick,double distance=2pt,rounded corners}
}

\tikzset{%
	patch/.pic={
			\clip (-11pt,0) rectangle (11pt,11pt);
			\draw[patch] (0,0) circle (7pt); 		
	}
}

\tikzset{%
	pics/dot/.style n args={1}{%
		code={%
			\draw (0,0) -- (0,#1);
			\fill (0,#1) circle (1.5pt); 		
			}
	}
}

\tikzset{%
	pics/arch/.style n args={2}{%
		code={%
			\draw (0,0) ..controls (0,#2) and (#1,#2).. (#1,0);		
		}
	}
}

\tikzset{%
	pics/bridge/.style n args={3}{%
		code={%
			\draw (0,0) ..controls (0,#3) and (#2,#3).. (#2,0);		
			\begin{scope}	
				\clip (0,0) ..controls (0,#3) and (#2,#3).. (#2,0);
				\draw (#1,0)--(#1,#3);
			\end{scope}
		}
	}
}

\tikzset{%
	pics/aqueduct/.style n args={4}{%
		code={%
			\draw (0,0) ..controls (0,#4) and (#3,#4).. (#3,0);		
			\begin{scope}	
				\clip (0,0) ..controls (0,#4) and (#3,#4).. (#3,0);
				\draw (#1,0)--(#1,#4);
				\draw (#2,0)--(#2,#4);
			\end{scope}
		}
	}
}

\def\d{.5cm}\def\h{1cm}\def\hdot{.4cm}

\title{\Large \textbf{Infinite dihedral Wakimoto sheaves}}
\author{Leonardo Maltoni}

\begin{document}
	\maketitle
	\begin{abstract}
		We study the extension groups between (modular) Wakimoto sheaves in type 
		$\tilde{A_1}$. Firstly we determine them completely over 
		characteristic zero fields. Secondly we describe a dg model which allows us to compute 
		these groups in the antispherical category for arbitrary coefficients.
	\end{abstract}
	\section{Introduction}
		The affine Hecke algebra is a fundamental object in the representation theory of reductive algebraic
		groups in positive characteristic.
		
		Recall that the affine Hecke algebra $\HH$ is a deformation of the group algebra of the 
		affine Weyl group $W$ attached to a reductive algebraic group $G$ over an algebraically closed 
		field, with a Borel subgroup $B$ and a maximal torus $T$.
 		The group $W$ contains the coroot lattice acting by translations
		$t_{\lambda}$ and one can find, also in $\HH$, elements
		$\theta_\lambda$ with the property that 
		$\theta_{\lambda_1}\theta_{\lambda_2}=\theta_{\lambda_1+\lambda_2}$. In other words we have a 
		commutative subalgebra of $\HH$ corresponding to the lattice in $W$. Bernstein gave a
		presentation of the affine Hecke algebra which highlights the properties of this lattice, and
		is very useful when trying to address representation theoretic questions about $\HH$.
		
		The affine Hecke algebra can be seen as the Grothendieck ring of some graded monoidal additive 
		category $\D$ which is called \emph{Hecke category}. There are actually several versions of the latter,
		which are equivalent under certain assumptions. In the modular setting two versions 
		are of particular interest: the category of equivariant parity sheaves over the affine flag variety, and 
		its diagrammatic presentation. It is then natural to ask what the higher level counterpart of the $\theta_\lambda$'s is.
		To find an appropriate answer one should consider the \emph{mixed} setting, 
		which in this case consists of the bounded homotopy category of $\D$. Here we find complexes 
		$\Wak{\lambda}$	which correspond to the $\theta_\lambda$'s, called \emph{Wakimoto sheaves}.
		Hence, in order to somehow ``lift'' the Bernstein presentation to a categorical level, one should 
		understand the subcategory that these objects form, and, to begin with, study the morphisms 
		between them.

		The homotopy category of $\D$ was already considered by Rouquier, and the Wakimoto sheaves above
		are special cases of \emph{Rouquier complexes}. 
		The latter form an interesting subcategory of
		$\Kb(\D)$ which is known to categorify (the actions of) a quotient of the braid group associated with $W$. 

		In this paper we study these morphism spaces in type $\tilde{A_1}$ with arbitrary coefficients.
		We use 
		the reduction of Rouquier complexes 
		from \cite{Mal_red}.

			\subsection{Diagrammatic category}
			We will use the presentation by generators and relations of the Hecke category
			introduced by Elias and Williamson \cite{EW}.
			This was initially a presentation for the category of \emph{Soergel bimodules},
			introduced in \cite{Soe} as an algebraic model of the Hecke category.
			 
			The Hecke category is obtained by Karoubian completion from 
			(the additive hull of) a certain \emph{Bott-Samelson} category, that can be thought of 
			as its monoidal skeleton.
			Hence one can give a presentation to this smaller category and then recover the whole 
			Hecke category formally. The generating objects, denoted $B_s$, correspond to simple 
			reflections, and are represented as 
			colored points (one color for each $s\in S$). The morphisms between 
			tensor products of the $B_s$'s are described by certain \emph{diagrams} inside the strip
			$\mathbb{R}\times [0,1]$. Namely, these are
			planar graphs obtained from some generating vertices, that connect the sequences of 
			colored points corresponding to the source and the targets,
			and they are identified under some relations.

			This \emph{diagrammatic Hecke category}, denoted by $\D$, 
			only depends on the Coxeter system and on its
			\emph{realization} $\mathfrak{h}$ (i.e.\ a 
			finite rank representation over the coefficient ring satisfying certain properties).
			
			Furthermore we have explicit bases for the morphism spaces, given in terms of 
			\emph{light leaves} maps. These were first introduced by Libedinsky \cite{Lib}, and then
			described diagrammatically in \cite{EW}.
		\subsection{Rouquier complexes}
			As we said, we can rephrase the problem of computing extension groups between Wakimoto sheaves 
			in terms of the homotopy category $\Kb(\D)$
			of the (diagrammatic) Hecke category. For a simple reflection $s$ the standard and 
			costandard sheaves correspond to certain complexes denoted by
			$F_s$ and $F_{s}^{-1}$, and the Wakimoto sheaves
			correspond to certain tensor products between these. 
			For a general Coxeter group $W$, the objects obtained as products 
			between the $F_{s}^{\pm 1}$ in all possible ways are called \emph{Rouquier complexes}, and
			they form a very interesting subcategory of $\Kb(\D)$.
			
			A key ingredient of our computation is a general homotopy reduction on these
			complexes, see \cite{Mal_red}. Let us recall it briefly.
			Consider a \emph{positive} Rouquier complex, of the form:
			\[
				F_{\uw}^\bullet=F_{s_1}F_{s_2}\dots F_{s_n}			
			\]
			for $s_i\in S$. As a graded object, this is the direct sum of Bott-Samelson 
			objects indexed by the $2^n$ subexpressions of $\uw=s_1s_2\dots s_n$.
			We find a summand $F_{\uw}$ which, 
			as a graded object, is a direct sum indexed over \emph{subwords} (each of which can correspond
			to several subexpressions). Then we have that the inclusion of the complex 
			$F_{\uw}$ in $\Roudg{\uw}$ is a homotopy equivalence.

			When the category $\D$ is Krull-Schmidt, any complex 
			admits a \emph{minimal subcomplex}: a homotopy equivalent summand with no
			contractible direct summand. Furthermore, one can show that this is unique up to isomorphism.
			The minimal subcomplex is very hard to find in general, but one can see this result as a first step of such a 
			reduction, available with no restriction on the coefficients.
		\subsection{Results}
			We only address the problem of extensions between Wakimoto sheaves problem in type $\tilde{A}_1$. Here the Wakimoto sheaves are 
			indexed by the integers.
			 
			First we assume that $\kk$ is a characteristic zero field.
			In this case one can actually compute the minimal subcomplexes of the Wakimoto sheaves.
			In \S\,\ref{subs_extchar0}, we use this to compute the extension groups. The result is described
			in Table \ref{tab_Homwak}, where the cell $(n,j)$ shows the space $\Hom(\Wak{-n},\un[j])$ (for positive values of the index, 
			one gets zero by the Rouquier formula).
			\begin{table}[h]
			  	\[
				   \begin{array}{|c|c|c|c|c|c|c|c|c|c|c|c}
				   \hline
				    \hbox{\diagbox{$n$}{$j$}}							 	& 0 	& 1 																	& 2 								& 3 								& 4 								& 5 								& 6 								& 7									& 8									& 9						& \dots			\\
				    \hline   																																																																																																		
				    0 							& R		&																		& 									&									&									&									&									&									&									&								&	\\
				    \hline   																																																																																
				    1								&		&																		& \kk(2)	&									&									&									&									&									&									&								&	\\
				    \hline   																																																																																																		
				    2								&		&																		& \kk(0)		&									& \kk(4)	&									&									&									&									&							&		\\
				    \hline   																																																																																																		
				    3								&		&																		& \kk(-2)	&									& \kk(2)	&									& \kk(6)  &									&									&							&		\\
				    \hline																																																																																																			
					4								&		& 																		& \kk(-4)	&									& \kk(0)		&									& \kk(4)	&									& \kk(8)	&							&		\\
					\hline																																																																																																			
				    \dots 							& \dots & \dots 																& \dots 							& \dots 							& \dots 							& \dots 							& \dots								& \dots								& \dots 							& \dots					& \dots			\\
				   \end{array}		  
				  \]
				  \caption{The morphism space $\Hom(\Wak{-n},\un[j])$.}\label{tab_Homwak}
			  \end{table}
			  
			  For a more general coefficient ring it is much harder to find the minimal subcomplex, 
			  and, if the category is not Krull-Schmidt, this is not even well defined. 
			  Nevertheless we can use the above general reduction to describe morphism spaces in the dg category
			  of complexes.
			  We consider the dg module of morphisms
			  \begin{equation}\label{eq_hombwaktoun}
			  	\Homb(\Wak{-n},\un).
			  \end{equation}
			  The problem of finding the extension groups is equivalent to computing the cohomology of this complex.
			 The dg module \eqref{eq_hombwaktoun} is free as a left dg $R$-module with a basis given by a certain version of the light leaves maps.
			 The differential is now described by matrices with entries in $R$. A crucial observation
			 is that many of these entries are just $\pm 1$ so we can use Gaussian elimination to simplify the complex
			 of morphisms.
			 In this way, in \S\,\ref{subs_weed}, we find a much simpler model for \eqref{eq_hombwaktoun}, homotopy equivalent to it.
			 It is still hard to compute its cohomology but the complexity is remarkably decreased. Furthermore this allows us to
			 compute the cohomology, with arbitrary coefficients, in the \emph{antispherical category}, considered in \cite{RW} or \cite{LibWil_anti}.

			Let ${}^I\!\Wak{n}$ denote the images
			of the Wakimoto sheaves in the antispherical category.
			The result is expressed in terms of \emph{two-color cyclotomic polynomials} $\phi_{n}$, 
			a certain data depending on the realization of the Coxeter system. For the standard Cartan matrix,
			$\phi_n$ is the exponential of the Von Mangoldt function:
			\[
				\phi_n=	\begin{cases}
							p	&\text{if $n=p^r$ for a prime $p$,}\\
							1	&\text{otherwise.}
						\end{cases}
			\]
			Let $\hat{P}(k)$ denote the set of partitions of a positive integer $k$ such that 
			each part divides the next. For $\lambda=(\lambda_1,\dots,\lambda_i)\in \hat{P}(k)$, 
			consider the cube of the following form	(the picture is for $i=3$).
			
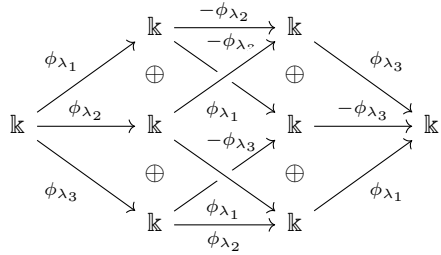
\[
				\begin{tikzcd}[row sep=tiny,column sep=large]
																												& \Ra\ar[r,"-\phi_{\lambda_2}"]\ar[ddr,"-\phi_{\lambda_3}", near start]		& \Ra \ar[ddr,"\phi_{\lambda_3}"]														&		\\
																												& \oplus												& \oplus																	&		\\
					\Ra \ar[uur,"\phi_{\lambda_1}"] \ar[r,"\phi_{\lambda_2}"] \ar[ddr,swap,"\phi_{\lambda_3}"]	& \Ra \ar[uur,crossing over,"\phi_{\lambda_1}", near start,swap]		& \Ra \ar[r,"-\phi_{\lambda_3}"]														& \Ra		\\
																												& \oplus												& \oplus																		&		\\
																												& \Ra \ar[uur,"\phi_{\lambda_1}", near start,swap] \ar[r,swap,"\phi_{\lambda_2}"]	& \Ra \ar[from=uul,crossing over,"-\phi_{\lambda_3}", near start]\ar[uur,swap,"\phi_{\lambda_1}"]	&		
				\end{tikzcd}				
			\]
			One can show that the complex $\Homb({}^I\!\Wak{-n},\un)$ reduces to many small cubes 
			as above, each with a certain shift. So we can compute the cohomology of each of them
			and then patch the contributions together.
			
			For $\lambda\in\hat{P}(k)$, let $I_{\lambda}$ be the ideal of $\kk$ generated by the 
			corresponding cyclotomic polynomials:
			\[
				I_{\lambda}=(\phi_{\lambda_1},\phi_{\lambda_2},\dots,\phi_{\lambda_{k-1}},\phi_{\lambda_k}).
			\]
			Then it is easy to see that the cohomology of a cube as above is
			$\kk/I_{\lambda}$ in the top degree and zero elsewhere.
			
			One finally obtains the following pattern.
			For a given partition $\lambda$ with $k$ parts and with $d$ distinct numbers appearing,
			the \emph{weight} $|\lambda|$ is defined as $2k-d$. 	
			Then we define the graded $\kk$ modules $H_k$ as follows: 
			\begin{align*}
				& H_1=\kk, &
				& H_k:=\bigoplus_{\lambda\in\hat{P}(k)} \kk/I_\lambda[1-|\lambda|],\quad\text{for $k>2$.}			
			\end{align*}
			Then we have (see Theorem \ref{thm_cohomologyZ} below):
			\begin{thm}
				The cohomology of $\Homb({}^I\!\Wak{-n},\un)$ is
				\[
					\bigoplus_{i=2}^{2n}H_{\lfloor i/2 \rfloor} [i-2-2n],
				\]
				where $\lfloor\cdot\rfloor$ denotes the floor function.
			\end{thm}
			We display the cohomology for the first values of $n$ in Table \ref{tab_cohZ} at page \pageref{tab_cohZ}. The entry 
			$i$ stands for $\kk/(\phi_i)$, and $(i,j)$ for $\kk/(\phi_i,\phi_j)$. Stacked entries represent
			direct sums. So, for example the entry
			\[
				\substack{6\\2,4\\2}
			\]
			means $\kk/(\phi_6)\oplus\kk/(\phi_2,\phi_4)\oplus\kk/(\phi_2)$.
			Notice that, in characteristic zero, the only part that would survive is the upper ``stair''
			of $\kk$'s.
			\begin{landscape}
			\begin{table}
				\[
				   \begin{array}{|c|c|c|c|c|c|c|c|c|c|c|c|c|c|c|c|c|c|c|c|c|c|c|c|c|c|c}
				   \hline
				    \hbox{\diagbox{$n$}{$j$}}	 	& 0 	& 1 		& 2 			& 3 							& 4 					& 5 							& 6 							& 7								& 8									&  9								& 10						& 11							& 12 						& 13						& 14					& 15						& 16				& 17				&18					& 19				& 20				& 21		& 22		& 23		&  24	&\dots\\
				    \hline   																																																																																																																																							
				    0 								& \kk	&			& 				&								&						&								&								&								&									&									&							&								& 							&							&						&							&					&					&					&					&					&			&			&			&		&\dots\\
				    \hline   																																																																																																																																											
				    1								&		& \kk		& \kk			&								&						&								&								&								&									&									&							& 								& 							&							&						&							&					&					&					&					&					&			&			&			&		&\dots\\
				    \hline   																																																																																																																																								
				    2								&		&			& 2				& \kk							& \kk					&								&								&								&									&									&							&								& 							&							&						&							&					&					&					&					&					&			&			&			&		&\dots\\
				    \hline   																																																																																																																																																
				    3								&		&			& 3				& 2								& 2						& \kk							& \kk  							&								&									&									&							&								&							&							&						&							&					&					&					&					&					&			&			&			&		&\dots\\
				    \hline																																																																																																																																									
					4								&		& 			& 4				& 3								& \substack{3\\2}		& 2								& 2								& \kk							& \kk								&									&							&								&							&							&						&							&					&					&					&					&					&			&			&			&		&\dots\\
					\hline																																																																																																																																																
					5								&		& 			& 5				& 4								& 4						& \substack{3\\2}				& \substack{3\\2}				& 2								& 2									& \kk								& \kk						&								&							&							&						&							&					&					&					&					& 					&			&			&			&		&\dots\\
					\hline																																																																																																																																														
					6								&		& 			& 6				& \substack{5\\2,4}				& \substack{5\\3}		& 4								& \substack{4\\2}				& \substack{3\\2}				& \substack{3\\2}					& 2									& 2							& \kk							& \kk						&							&						&							&					&					&					&					&					&			&			&			&		&\dots\\
					\hline																																																																																																																																																	
					7								&		& 			& 7				& 6								& \substack{6\\2,4}		& \substack{5\\2,4\\3}			& \substack{5\\3}				& 4								& \substack{4\\2} 					& \substack{3\\2}					& \substack{3\\2}			& 2								& 2							& \kk						& \kk					&							&					&					&					&					&					&			&			&			&		&\dots\\
					\hline																																																																																																																																																							
					8								&		& 			& 8				& \substack{7\\2,6}				& \substack{7\\4}		& \substack{6\\2,4}				& \substack{6\\2,4}				& \substack{5\\2,4\\3}			& \substack{5\\3\\2} 				& 4									& \substack{4\\2} 			& \substack{3\\2}				& \substack{3\\2}			& 2							& 2						& \kk						& \kk				&					&					&					&					&			&			&			&		&\dots\\
					\hline																																																																																																																																																					
					9								&		& 			& 9				& \substack{8\\3,6}				& \substack{8\\2,6}		& \substack{7\\2,6\\4}			& \substack{7\\4\\2,4\\3}		& \substack{6\\2,4}				& \substack{6\\2,4}					& \substack{5\\2,4\\3\\2}			& \substack{5\\3\\2}		& 4								& \substack{4\\2}			& \substack{3\\2}			& \substack{3\\2}		& 2							& 2					& \kk				& \kk				&					&					&			&			&			&		&\dots\\
					\hline																																																																																																																																																			
					10								&		& 			& 10			& \substack{9\\2,8}				& \substack{9\\3,6\\5}	& \substack{8\\3,6\\2,6\\2,4}	& \substack{8\\2,6}				& \substack{7\\2,6\\4\\3\\2,4}	& \substack{7\\4\\2,4\\3}			& \substack{6\\2,4}					& \substack{6\\2,4\\2}		& \substack{5\\2,4\\3\\2}		& \substack{5\\3\\2}		& 4							& \substack{4\\2}		& \substack{3\\2}			& \substack{3\\2}	& 2					& 2					& \kk				& \kk				&			&			&			&		&\dots\\
					\hline																																																																																																																																													
					11								&		& 			& 11			& 10							& \substack{10\\2,8}	& \substack{9\\2,8\\5}			& \substack{9\\3,6\\5\\2,6\\2,4}& \substack{8\\3,6\\2,6\\2,4}	& \substack{8\\2,6\\2,4}			& \substack{7\\2,6\\4\\3\\2,4}		& \substack{7\\4\\2,4\\3}	& \substack{6\\2,4\\2}			& \substack{6\\2,4\\2}		& \substack{5\\2,4\\3\\2}	& \substack{5\\3\\2}	& 4							& \substack{4\\2}	& \substack{3\\2}	& \substack{3\\2}	& 2					& 2					& \kk		& \kk		&			&		&\dots\\
					\hline																																																																															
					12								&		& 			& 12			& \substack{11\\2,10\\3,9\\4,8}	& \substack{11\\6}		& \substack{10\\2,8\\3,6}		& \substack{10\\2,8\\4}			& \substack{9\\2,8\\5\\2,6\\2,4}& \substack{9\\3,6\\5\\2,6\\2,4\\3}	& \substack{8\\3,6\\2,6\\2,4\\2,4}	& \substack{8\\2,6\\2,4}	& \substack{7\\2,6\\4\\3\\2,4}	& \substack{7\\4\\2,4\\3\\2}& \substack{6\\2,4\\2}		& \substack{6\\2,4\\2}	& \substack{5\\2,4\\3\\2}	& \substack{5\\3\\2}& 4					& \substack{4\\2}	& \substack{3\\2}	& \substack{3\\2}	& 2			& 2			& \kk		& \kk	&\dots\\
					\hline																																																																															
				    \dots 							& 	 	&		 	& \dots 		& \dots 						& \dots 				& \dots 						& \dots							& \dots							& \dots								& \dots								& \dots						& \dots							& \dots 					& \dots						& \dots					& \dots						& \dots				& \dots				& \dots				&\dots				&	\dots			& \dots		& \dots		& \dots		& \dots	&\dots\\
				   \end{array}		  
				\]
				\caption{The cohomology of $\Homb({}^I\!\Wak{-n},\un)$.}\label{tab_cohZ}
			\end{table}
			\end{landscape}
	\section{Setup}
	In this section we describe the machinery of Soergel calculus and 
	Rouquier complexes that we will use. As we will only apply this in type $\tilde{A_1}$, 
	we will give full detail only for that case.
	\subsection{Notation for Coxeter systems}
	Let $(W,S)$ be a Coxeter system.
	We call \textit{Coxeter word}
	an element of the free monoid $\cword$ generated by $S$. Let $\Sigma^+=\{\sigma_s\mid s\in S\}$ be the set of generators of the corresponding braid group, and
	$\Sigma^-=\{\sigma_s^{-1}\mid s\in S\}$ and $\Sigma=\Sigma^+\sqcup \Sigma^-$.
	We call \textit{braid word} an element of the free monoid $\bword$ generated by $\Sigma$. 
	Words will be denoted by underlined letters:
	\[
		\uw= s_1s_2\cdots s_k \quad\text{or}\quad \uom=\sigma_{s_1}^{\pm 1}\cdots \sigma_{s_k}^{\pm 1}.
	\]		
	Let $\ell(\uw)=k$ denote the length of the word $\uw$. We still use
	$\ell$ for the length function in $W$ (so $\ell(w)$ is the minimal length of words for $w$).
	Let $\le$ denote the Bruhat order.
	Recall that a \textit{reflection} is a conjugate of a simple reflection.
	\begin{exa}
		Let $W$ be the affine Weyl group of type $\tilde{A_1}$.
		Here there are two simple reflections $S=\{s,t\}$, where
		$S_{\mathrm{f}}=\{s\}$ is the finite reflection and $\{t\}=S\setminus S_{\mathrm{f}}$ 
		is the affine reflection. 
		Notice that, for $n>0$, there are exactly
		two elements in $W$ with length $n$, and each of them has a unique reduced (i.e.\ of minimal length) word. Let $\sigma_n$ 
		(or $\tau_n$) be the
		only element of length $n$ whose unique reduced word starts with $s$ (resp.\ $t$).
		And let also $\us_n$ and $\ut_n$ denote their respective reduced words:
		\begin{align*}
			\us_n &=\underbrace{stst\dots}_{n \text{ letters}}\\
			\ut_n &=\underbrace{tsts\dots}_{n \text{ letters}}
		\end{align*}	
 		Notice that here the reflections are precisely the elements of odd length.
	\end{exa}			
		\subsection{Two-colored quantum numbers}
		Consider the ring $\ZZ[x,y]$ and define $[1]_x:=[1]_y:=1$ and $[2]_x:=x$ and $[2]_y:=y$. Then, recursively
		\begin{align*}
			&[n+1]_x:=[2]_x[n]_y - [n-1]_x, \\
			&[n+1]_y:=[2]_y[n]_x - [n-1]_y.
		\end{align*}
		These are the \emph{two-colored quantum numbers}.
		It is easy to see that if $n$ is odd then $[n]_x=[n]_y$, so we will sometimes omit
		the index in this case.
		\begin{rmk}\label{rmk_specquantumnumb}
			If we consider the morphism $\ZZ[x,y]\rightarrow \ZZ[v,v^{-1}]$ sending both $x$ and $y$ to 
			$v+v^{-1}$, the images of the two-colored quantum numbers are (the symmetrized version of) 
			the usual quantum numbers:
			\[
				\qn{n}{x},\qn{n}{y}\mapsto \frac{v^{n}-v^{-n}}{v-v^{-1}}=v^{-n+1}+v^{-n+3}+\dots +v^{n-3}+v^{n-1}.
			\]
			The two-colored version, as we will see, shares many properties with the usual one.
		\end{rmk}			
	\subsection{Realization of a Coxeter system}
	Let $\kk$ be a commutative ring and $\mathfrak{h}$ a 
	realization of $W$ in the sense of \cite[\S 3.1]{EW}. This consists in a free,
	finite rank $\kk$-module $\mathfrak{h}$ over which $W$ acts linearly via 
	\[
		s(v)=v-\langle \alpha_s, v\rangle \alpha_s^{\vee}, 
															\quad\forall s \in S
	\]
	for certain distinguished elements $\alpha_s^\vee \in \mathfrak{h}$ and 
	$\alpha_s \in \mathfrak{h}^*=\Hom(\mathfrak{h},k)$ (that we call respectively \textit{simple coroots} and \textit{simple roots})
	such that $\langle \alpha_s,\alpha_s^\vee\rangle =2$, for each $s\in S$,
	where $\langle -,-\rangle:\mathfrak{h}^*\times \mathfrak{h}\rightarrow \kk$ is the natural pairing.
	Then $W$ acts on $\mathfrak{h}^*$ by the contragredient representation, given by similar formulas.
	The minimal information of the realization, namely its action over simple roots and coroots, 
	is then encoded in the \textit{Cartan matrix} $(\langle \alpha_s , \alpha_t^\vee \rangle)_{s,t\in S}$.	

	We assume that the realization satisfies \textit{Demazure surjectivity} (see \cite[Ass.\ 3.9]{EW}), namely we suppose that, for each $s\in S$,
	the maps $\langle \alpha_s,- \rangle: \mathfrak{h}\rightarrow \kk$ and 
	$\langle -,\alpha_s^\vee \rangle: \mathfrak{h}^*\rightarrow \kk$ 
	are surjective. In this case we choose some $\delta_s\in\mathfrak{h}^*$ such that 
	$\langle \delta_s,\alpha_s^\vee\rangle=1$.

	Let $R=S(\mathfrak{h}^*)$, with $\mathfrak{h}^*$ in degree 2. The action of $W$ on $\mathfrak{h}^*$ extends naturally to $R$. 
	We define, for each $s\in S$, the \textit{Demazure operator} $\partial_s : R \rightarrow R$, via
	\[
					f \mapsto \frac{f-s(f)}{\alpha_s}
	\]
	We set $\qn{n}{s}$ and $\qn{n}{t}$ to be the images of $\qn{n}{x}$ and $\qn{n}{y}$ via the map
	the map $\ZZ[x,y]\rightarrow \kk$ sending $x$ to $-\partial_s(\alpha_t)$ and $y$ to $-\partial_t(\alpha_s)$.
	\begin{exa}
		In type $\tilde{A_1}$ the Cartan matrix is a $2\times 2$ matrix of the form
		\[
			\begin{pmatrix}
				2	& \langle \alpha_s , \alpha_t^\vee \rangle \\
				\langle \alpha_t , \alpha_s^\vee \rangle	& 2
			\end{pmatrix}
			=
			\begin{pmatrix}
				2	& \partial_t(\alpha_s) \\
				\partial_s(\alpha_t) 	& 2
			\end{pmatrix}
			=
			\begin{pmatrix}
				2	& -\qn{2}{t} \\
				-\qn{2}{s} & 2
			\end{pmatrix}			
		\]
		and it is determined by the two values $\qn{2}{s}=-\partial_s(\alpha_t)$ and $\qn{2}{t}=-\partial_t(\alpha_s)$.
		We call \textit{standard} Cartan matrix for $\tilde{A_1}$ the one for which $\qn{2}{s}=\qn{2}{t}=2$.
		Here are three interesting realizations corresponding to it (cf.\ \cite[Ex.\ 3.3]{EW}).
		\begin{enumerate}
			\item The \textit{natural} realization coming from the root datum of $\mathrm{SL}_2$. We take 
				$\mathfrak{h}^*=\kk \otimes X$ and $\mathfrak{h}=\kk \otimes X^\vee$. Then we take $\alpha_s$ to be the only simple root
				and $\alpha_s^\vee$ to be the only simple coroot. Then we set $\alpha_t:=-\alpha_s$ and $\alpha_t^\vee:=-\alpha_s^\vee$
				(so neither $\{\alpha_s,\alpha_t\}$, nor $\{\alpha_s^\vee,\alpha_t^\vee\}$ are linearly independent).
				One can also take $\mathfrak{h}$ to be $\kk\otimes \roo^\vee$ (then a natural 
				assumption to make is that $\kk\otimes \roo \times \kk\otimes \roo^\vee\rightarrow \kk$ 
				is a perfect pairing so as to identify $\mathfrak{h}^*$ with the base-changed root lattice). 
				This is essentially the realization considered in
				\cite{RW}, up to a different choice of notation (exchanging roots and coroots). 
				Notice that the action of $W$ here factors through that of $W_{\mathrm{f}}$.
			\item The \textit{geometric representation} of $(W,S)$, where $\hh:=\mathbb{R}\alpha_s^\vee \oplus \mathbb{R}\alpha_t^\vee$ and 
				$\alpha_s,\alpha_t\in \hh^*$ are defined via the above matrix. In this case, in fact, the Cartan matrix coincides with the 
				Coxeter matrix $\big(-2\cos(\pi/m_{st})\big)$, so we can actually define this realization over any $\kk$.
				Notice that in particular $\alpha_t=-\alpha_s$.				
			\item The \textit{Kac-Moody} realization, as described in \cite{Kac}. If we consider the 
				Kac-Moody Lie algebra $\mathfrak{g}$, defined by our Cartan matrix, we can take $\hh$ 
				to be any $\ZZ$-lattice of the Cartan subalgebra of $\mathfrak{g}$, 
				containing the coroot lattice such that the dual lattice $\hh^*$ contains the root 
				lattice. Here both the simple roots and the simple coroots
				and simple roots are linearly independent.
				Over the root lattice (and similarly for the coroot lattice), we can picture this realization as follows
				
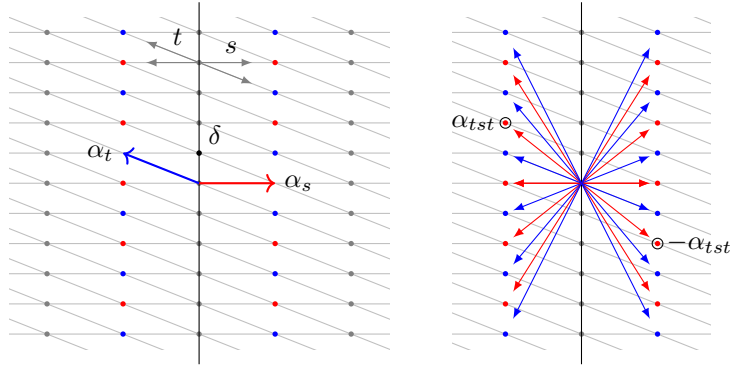
\[
					\begin{tikzpicture}
						\def\al{1cm}\def\del{.4cm}
						\begin{scope}
							\clip (-2.5*\al,-5.5*\del) rectangle (2.5*\al,5.5*\del);
							\foreach \i in {7,6,5,4,3,2,1,0,-1,-2,-3,-4,-5,-6,-7}{%
								\draw[lightgray] (-3*\al,\i*\del+3*\del)--(3*\al,\i*\del-3*\del);
								\draw[lightgray] (-3*\al,\i*\del)--(3*\al,\i*\del);
								}
						\end{scope}
						\foreach \i in {-5,-4,-3,-2,-1,0,1,2,3,4,5}{%
							\foreach \j in {-2,0,2}{%
								\fill[gray] (\j*\al,\i*\del) circle (1pt);
							}
						}
							\draw[gray,latex-latex] (-.7*\al,4*\del+.7*\del)--node[above,pos=.3,black] {$t$} (.7*\al,4*\del-.7*\del);
							\draw[gray,latex-latex] (-.7*\al,4*\del)--node[above,pos=.8,black] {$s$} (.7*\al,4*\del);
						\draw (0,-6*\del)--(0,6*\del);
						\draw[red,thick,->] (0,0)--(\al,0); \node[anchor=west] at (\al,0) {$\alpha_s$}; 
						\draw[blue,thick,->] (0,0)--(-\al,\del); \node[anchor=east] at (-\al,\del) {$\alpha_t$};
						\fill (0,\del) circle (1pt); \node[anchor=south west] at (0,\del) {$\delta$};
						\foreach \i in {1,2,-1,-2}{%
							\fill[red] (\al,2*\i*\del) circle (1pt);
							\fill[red] (-\al,2*\i*\del) circle (1pt);
							}
						\fill[red] (-\al,0) circle (1pt);
						\foreach \i in {-1,-3,-5,3,5}{%
							\fill[blue] (\al,\i*\del) circle (1pt);
							\fill[blue] (-\al,\i*\del) circle (1pt);
							}
						\fill[blue] (\al,\del) circle (1pt);
					\end{tikzpicture}
					\qquad
					\begin{tikzpicture}
						\def\al{1cm}\def\del{.4cm}\def\fac{0.9}
						\begin{scope}
							\clip (-1.7*\al,-5.5*\del) rectangle (1.7*\al,5.5*\del);
							\foreach \i in {7,6,5,4,3,2,1,0,-1,-2,-3,-4,-5,-6,-7}{%
								\draw[lightgray] (-3*\al,\i*\del+3*\del)--(3*\al,\i*\del-3*\del);
								\draw[lightgray] (-3*\al,\i*\del)--(3*\al,\i*\del);
								}
							\foreach \i in {-5,-4,-3,-2,-1,0,1,2,3,4,5}{%
								\foreach \j in {-2,0,2}{%
									\fill[gray] (\j*\al,\i*\del) circle (1pt);
								}
							}
						\end{scope}
						\foreach \i in {0,1,2,-1,-2}{%
							\fill[red] (\al,2*\i*\del) circle (1pt);
							\fill[red] (-\al,2*\i*\del) circle (1pt);
							}
						\foreach \i in {-1,-3,-5,1,3,5}{%
							\fill[blue] (\al,\i*\del) circle (1pt);
							\fill[blue] (-\al,\i*\del) circle (1pt);
							}
						\draw (0,-6*\del)--(0,6*\del);
						\foreach \i in {0,2,4}{%
							\draw[latex-latex,red,scale=\fac] (-\al,-\i*\del) -- (\al,\i*\del);							
						}
						\foreach \i in {1,3,5}{%
							\draw[latex-latex,blue,scale=\fac] (-\al,-\i*\del) -- (\al,\i*\del);							
						}
						\foreach \i in {0,2,4}{%
							\draw[latex-latex,red,scale=\fac] (-\al,\i*\del) -- (\al,-\i*\del);							
						}
						\foreach \i in {1,3,5}{%
							\draw[latex-latex,blue,scale=\fac] (-\al,\i*\del) -- (\al,-\i*\del);							
						}
						\draw (-\al,2*\del) circle (2pt);
						\draw (\al,-2*\del) circle (2pt);
						\node[anchor=east] at (-\al,2*\del) {$\alpha_{tst}$};
						\node[anchor=west] at (\al,-2*\del) {$-\alpha_{tst}$};
					\end{tikzpicture}
				\]		
				where the actions of $s$ and $t$ are indicated by the gray arrows, and the red (and blue) points are the 
				orbit of $\alpha_s$ (and $\alpha_t$) via $W$. The points in the vertical black line are fixed by $W$.
				
				In the picture on the right, one can see the action of (not only simple) reflections (the red ones are conjugates of
				$s$ and the blue ones of $t$).
			 	Each reflection $x$ is associated with a unique pair  of \textit{roots} $\pm \alpha_x\in\mathfrak{h}^*$ and 
			 	of \textit{coroots} $\pm\alpha_x^\vee\in \mathfrak{h}$, such that 
			 	\[
		 			x(v)=v-\langle \alpha_x, v\rangle \alpha_x^\vee,\quad \forall v\in\mathfrak{h}
			 	\]
			 	If $x=wsw^{-1}$ one sees that $\alpha_x^\vee=\pm w(\alpha_s^\vee)$ and $\alpha_x=\pm w(\alpha_s)$. 	
			 	We put $\alpha_x:=w(\alpha_s)$ where $w$ is of minimal length (and similarly with $t$ instead of $s$).
			 	In the picture, we have indicated for example $\alpha_{tst}=t(\alpha_s)$.
			 	In the context of \cite{Kac}, these correspond to the positive roots and coroots of the corresponding
			 	Kac-Moody algebra, but we extend the definition to all realizations of $\tilde{A_1}$.
			\end{enumerate}
		Notice that in order to satisfy Demazure surjectivity for such realizations we may have to assume that
		2 is invertible in $\kk$ (we definitely need this for (\textit{i}) and (\textit{ii}) and we could avoid this for 
		(\textit{iii})).
	\end{exa}
	\subsection{The diagrammatic Hecke category (in type $\tilde{A_1}$)}
	We now define a category $\D=\D(\mathfrak{h},\kk)$ which is a
	$\kk$-linear monoidal category enriched in graded $R$-bimodules, depending on the above data.
	First one defines the Bott-Samelson category $\D_{\mathrm{BS}}$ by generators and relations,
	then one gets $\D$ as the Karoubi envelope of the closure of $\D_{\mathrm{BS}}$ by direct sums and shifts.

	\begin{enumerate}
		\item The objects of $\DBS$ are generated by tensor product from objects $B_s$ for $s\in S$. So a general object
			corresponds to a Coxeter word: if
			$\uw=\underline{s_1\dots s_n}$, let $B_{\uw}$ denote the object
			$B_{s_1}\otimes \cdots \otimes B_{s_n}$. Let also $\un$ denote the monoidal unit. 
		\item Morphisms in $\Hom_{\DBS}(B_{\uw_1},B_{\uw_2})$
			are $\kk$-linear combinations of \textit{Soergel graphs}, which are defined as follows.
			\begin{itemize}
				\item We associate a color to each simple reflection;
				\item A Soergel graph is then a colored, \textit{decorated} planar graph 
					contained in the planar strip $\mathbb{R}\times [0,1]$, with boundary in 
					$\mathbb{R}\times\{0,1\}$;
				\item The bottom (and top) boundary are arrangements
					of boundary points colored according to the letters of the source word $\uw_1$ 
					(and target word $\uw_2$, respectively).
				\item The edges of the graph are colored in such a way that those connected with the boundary have consistent colors.
				\item The other vertices of the graph are then either: 
					\begin{enumerate}
						\item univalent (called \textit{dots}) which are declared of degree 1, or;
						\item trivalent	with three edges of the same color, degree $-1$, or;
						\item $2m_{st}$-valent with edges of alternating colors corresponding to $s$ and $t$, if $m_{st}$ is the order of $st$ in $W$, of degree 0.
					\end{enumerate}
				\item Decorations are boxes labeled by homogeneous elements in 
					$R$ that can appear in
					any region (i.e.\ connected component of the complementary of the graph): 
					we will usually omit the boxes and just write the polynomials.
			\end{itemize}
		\item These diagrams are identified via some relations:
			\begin{description}
				\item[Polynomial relations.] The first such relation says that whenever two polynomials are in the same region, they multiply
					(this makes morphism spaces $R$-bimodules, by acting on the leftmost or the rightmost region). Then we have also
					\begin{gather}
						\begin{tikzpicture}[baseline=0,scale=.7, transform shape] \label{barbell}
							\draw[gray,dashed] (0,0) circle (1cm); \draw[red] (0,.5)--(0,-.5); \fill[red] (0,.5) circle (2pt);\fill[red] (0,-.5) circle (2pt);
						\end{tikzpicture}
						=
						\begin{tikzpicture}[baseline=0,scale=.7, transform shape]
							\draw[gray,dashed] (0,0) circle (1cm); \node[draw] at (0,0) {$\alpha_{\textcolor{red}{s}}$};
						\end{tikzpicture}			\\
						\begin{tikzpicture}[baseline=0,scale=.7, transform shape]\label{sliding}
							\draw[gray,dashed] (0,0) circle (1cm); \draw[red] (0,1)--(0,-1); 
							\node[draw] at (.5,0) {$f$};
						\end{tikzpicture}			
						=
						\begin{tikzpicture}[baseline=0,scale=.7, transform shape]
							\draw[gray,dashed] (0,0) circle (1cm); \draw[red] (0,1)--(0,-1); 
							\node[draw, inner sep =.08cm] at (-.5,0) {$\textcolor{red}{s}(f)$};
						\end{tikzpicture}			
						+
						\begin{tikzpicture}[baseline=0,scale=.7, transform shape]
							\draw[gray,dashed] (0,0) circle (1cm); \draw[red] (0,1)--(0,.5); \fill[red] (0,.5) circle (2pt);
							\draw[red] (0,-1)--(0,-.5); \fill[red] (0,-.5) circle (2pt); 
							\node[draw] at (0,0) {$\partial_{\textcolor{red}{s}}(f)$};
						\end{tikzpicture}
						\end{gather}
				\item[One color relations.] These are the following
					\begin{gather}
						\begin{tikzpicture}[baseline=0,color=red,scale=sqrt(2)/3,scale=.7, transform shape]\label{associativity}
							\draw (-1.5,-1.5)--(-0.6,0)--(0.6,0)--(1.5,-1.5);
							\draw (-1.5,1.5)--(-0.6,0); \draw(0.6,0)--(1.5,1.5);
							\draw[dashed,gray] (0,0) circle ({3/sqrt(2)});
						\end{tikzpicture}
						=
						\begin{tikzpicture}[baseline=0,color=red,scale=sqrt(2)/3,rotate=90,scale=.7, transform shape]
							\draw (-1.5,-1.5)--(-0.6,0)--(0.6,0)--(1.5,-1.5);
							\draw (-1.5,1.5)--(-0.6,0); \draw(0.6,0)--(1.5,1.5);
							\draw[dashed,gray] (0,0) circle ({3/sqrt(2)});
						\end{tikzpicture}\\
						\begin{tikzpicture}[color=red,baseline=0,scale=.7, transform shape]\label{dotline}
							\draw[gray,dashed] (0,0) circle (1cm);
							\draw (0,-1) -- (0,1); \draw (0,0)--(.3,0);\fill (0.3,0) circle (2pt);
						\end{tikzpicture}
						=
						\begin{tikzpicture}[baseline=0,scale=.7, transform shape]
							\draw[gray,dashed] (0,0) circle (1cm);						
							\draw[red] (0,-1) -- (0,1);
						\end{tikzpicture}
						\\
						\begin{tikzpicture}[baseline=0,scale=.7, transform shape]
							\draw[gray,dashed] (0,0) circle (1cm);
							\draw[red] (0,-1)--(0,-0.3) arc (-90:270:0.3cm);
						\end{tikzpicture}
						= 0
					\end{gather}
				\item[Two color relations.] These allow to move dots,
					or trivalent vertices, past $2m_{st}$-valent vertices. They involve the 
					\textit{Jones-Wenzl morphisms} that will be used later, but we will not 
					need these relations;
				\item[Three color relations.] For each finite parabolic subgroup
					of rank 3, these relations ensure compatibilities between the three corresponding 
					$2m$-valent vertices. We will not need them either.
			\end{description}
			For the last two relations, more details can be found in \cite[\S 5.1]{EW}.
			
			Notice that all the relations are homogeneous, so the morphism spaces are \textit{graded}
			$R$-bimodules.
	\end{enumerate}
	This completes the definition of $\DBS$.
	\begin{exa}
		In the $\tilde{A_1}$ setting, we associate the color red to $s$ and blue to $t$.
		Here we only have polynomial and one color relations (as there is no braid relation in 
		the Coxeter group).
		The following diagram represents a morphism from $B_{\underline{tsttts}}$ to $B_{\underline{ts}}$
		\begin{center}
			\begin{tikzpicture}
				\draw (.5*\d,0)--(7.5*\d,0);
				\draw (.5*\d,4*\h)--(7.5*\d,4*\h);				
				\draw[blue] (\d,0) ..controls (\d,1.5*\h).. (2*\d,2*\h);
				\draw[blue] (2*\d,2*\h) ..controls (1.7*\d,2.3*\h).. (\d,2.3*\h); \fill[blue] (\d,2.3*\h) circle (1.5pt);
				\draw[blue] (2*\d,2*\h) ..controls (3*\d,2.5*\h).. (3*\d,3*\h);
				\draw[blue] (3*\d,3*\h) -- (4*\d,3*\h); \fill[blue] (4*\d,3*\h) circle (1.5pt);
				\draw[blue] (3*\d,3*\h) -- (3*\d,4*\h);
				\draw[red] (2*\d,0) ..controls (2.3*\d,1.4*\h).. (4.5*\d,2*\h);
				\draw[red] (4.5*\d,2*\h) ..controls (6.7*\d,1.4*\h).. (7*\d,0);
				\draw[red] (4.5*\d,2*\h) ..controls (5*\d,2.5*\h).. (5*\d,3.5*\h);
				\draw[red] (5*\d,3.5*\h) ..controls (6.5*\d,3.5*\h) and (5*\d,2.5*\h).. (6.5*\d,2.5*\h); \fill[red] (6.5*\d,2.5*\h) circle (1.5pt);
				\draw[red] (5*\d,3.5*\h) -- (5*\d,4*\h);
				\pic[blue] at (3*\d,0) {bridge={\d}{3*\d}{1.7*\h}};
				\begin{scope}
					\clip (3*\d,0) ..controls (3*\d,1.7*\h) and (6*\d,1.7*\h).. (6*\d,0);
					\draw[blue] (5*\d,2*\h)--(5*\d,.7*\h);\fill[blue] (5*\d,.7*\h) circle (1.5pt);
				\end{scope}
				\node at (5*\d,.25*\h) {$\alpha_s$};
				\node at (3.6*\d,2.3*\h) {$\alpha_s^2\alpha_t$};
				\node at (6.5*\d,2*\h) {$\alpha_s^3$};
				\node at (1.8*\d,3.3*\h) {$sts(\alpha_t)$};
			\end{tikzpicture}
		\end{center}
		We leave as an exercise to the reader to use the relations to simplify this diagram and \
		reduce it to a linear combination of diagrams with all polynomial on the left.
	\end{exa}
	
	By adjunction arguments we will actually only need to deal with morphism spaces of the 
	form $\Hom_{\D}(B_{\uw},\un)$. Let us introduce an operation on morphisms of this form that will 
	come useful later.	
	For a given positive integer $k$, consider the morphism
	\[
		\psi_{k}^s:=
		\underbrace{%
		\begin{tikzpicture}[baseline=.1*\h]
			\draw (-.5*\d,0)--(7.5*\d,0);
			\pic[red] at (0,0) {aqueduct={2*\d}{5*\d}{7*\d}{\h}};
			\node at (3.5*\d,.3*\h) {\dots};
			\foreach \i in {0,2,5,7}{%
				\node at (\i*\d,-.3*\h) {$B_s$};
				}
			\node at (3.5*\d,-.2*\h) {\dots};
		\end{tikzpicture}
		}_{k+1}
		\in \Hom_{\D}(B_{\underline{ss\dots ss}},\un)
	\]	
	For $i=1,\dots, k$, let $\uw_i$ be a Coxeter word and
	$\gamma_i\in \Hom_{\D}(B_{\uw_i},\un)$. Then
	we define
	\[
		\ropen \gamma_1 \rcomma \dots \rcomma \gamma_k \rclosed :=
		\psi_{k}^s
		\circ
		(\id_{B_s}\otimes\gamma_1\otimes\id_{B_s}\otimes\dots\otimes\id_{B_s}\otimes\gamma_k\otimes\id_{B_s})
	\]
	In other words, this is the morphism obtained by covering the 
	original ones with an $s$-arch and separating them by vertical strands.
	\subsection{Rouquier complexes}
	We consider the homotopy category $\Kb(\D)$. We have the standard and costandard complexes
		\[
			\begin{tikzcd}[row sep=small,%
							execute at end picture={%
								\def\hdott{.32cm}
								\pic[red,yshift=.3em] at (A) {dot={\hdott}};
								\pic[red,yshift=1.5em,yscale=-1] at (B) {dot={\hdott}};
								}]
				F_{\textcolor{red}{s}}=\,\cdots\ar[r]		& 0 \ar[r]		& 0	\ar[r]								& B_{\textcolor{red}{s}}\ar[r,""{coordinate,name=A}]	& \un(1)\ar[r]	& 0\ar[r]		& \cdots \\
															& \cohdeg{-2}	& \cohdeg{-1}							& \cohdeg{0}											& \cohdeg{1}	& \cohdeg{2}	& \\
				F_{\textcolor{red}{s}}^{-1}=\,\cdots\ar[r]	& 0\ar[r]		& \un(-1)\ar[r,""{coordinate,name=B}]	& B_{\textcolor{red}{s}}\ar[r]							& 0	\ar[r]		& 0\ar[r]		& \cdots \\
			\end{tikzcd}
		\]
		where $(-)$ is the shift in the grading of $\D$, and the numbers in the middle denote the cohomological degrees.
	Then, for any braid word $\uom=\sigma_{s_1}^{\pm 1}\sigma_{s_2}^{\pm 1}\dots \sigma_{s_n}^{\pm 1}$, we put
	\[
		F_{\uom}^\bullet:=F_{s_1}^{\pm 1}\otimes F_{s_2}^{\pm 1}\otimes \dots \otimes F_{s_n}^{\pm 1}
	\]
	The objects of this form, or rather their isomorphism classes in $\Kb(\D)$, are called
	\textit{Rouquier complexes}. These objects satisfy the following properties (see \cite{Rou_cat}, or 
	\cite{Mal_red} for a diagrammatic proof).
		\begin{pro}\label{pro_braidRou}
			One has the following.
			\begin{enumerate}
			 	\item\label{item_propinvers} Let $s\in S$, then $F_s F_s^{-1}\cong F_s^{-1}F_s\cong \un$.
				\item\label{item_propbraid} Let $s,t\in S$ with $m_{st}<\infty$, then
					\[
						\underbrace{F_s F_tF_sF_t \cdots}_{m_{st} \text{ times}} \cong %
							\underbrace{F_t F_sF_tF_s \cdots}_{m_{st} \text{ times}}
					\]
			\end{enumerate}
			Hence, for each pair of braid words $\uom_1$, $\uom_2$ expressing
			the same element $\omega\in B_W$, there is an isomorphism $F_{\uom_1}^\bullet\cong F_{\uom_2}^\bullet$.
			Furthermore:
			\begin{enumerate}[resume]
				\item\label{item_propcanonic}(Rouquier Canonicity)
					for each $\uom_1$ and $\uom_2$ as above, we have 
					\[
						\Hom(F_{\uom_1}^\bullet,F_{\uom_2}^\bullet)\cong R,
					\]
					and one can chose
					$\gamma_{\uom_1}^{\uom_2}$ such that the system 
					$\{\gamma_{\uom_1}^{\uom_2}\}_{\uom_1,\uom_2}$ is transitive.
			\end{enumerate}
		\end{pro}
		Let us observe that, thanks to these properties,
		the Rouquier complex $F_{\omega}$ associated to $\omega\in B_W$ is well defined up to a canonical isomorphism.
		
		Another important property is the so-called Rouquier formula. This was
		conjectured in \cite{Rou_der}, and proved in \cite{LibWil}, and in \cite{Maki}.
		See again also \cite{Mal_red} for a diagrammatic proof.
		\begin{pro}\label{pro_rouf}
			Let $w,v\in W$ and let $\uw$ and $\uv$ be reduced words expressing them.
			Let $\uom$ be the positive word lift of $\uw$ and $\underline{\nu}$ 
			be the negative word lift for $\uv$. Then
			\[
				\Hom^\bullet(F_{\uom},F_{\underline{\nu}})\simeq	\begin{cases}
																	R[0] & \text{if $w=v$,} \\
																	0 & \text{otherwise.}
																\end{cases}
			\] 
		\end{pro}
		\subsection{Wakimoto sheaves}
		Let now $W$ be an affine Weyl group of the form $\Wf\ltimes\roo^\vee$, where $\ZZ\Phi^\vee$ is the coroot lattice.
		
		Consider $\lambda\in \roo^\vee$ and write $\lambda=\mu-\nu$ with $\mu$ and $\nu$ dominant.
		Then, with $\tau_\mu$ and $\tau_\nu$ positive lifts of $t_{\mu}$ and $t_\mu$, we define
		the \emph{(modular) Wakimoto sheaf} associated to $\lambda$ as 
		\[
			\Wak{\lambda}:= F_{\tau_\mu} F_{\tau_\nu^{-1}}
		\]
		It is not difficult to prove that
		this definition does not depend on the choice of $\mu$ and $\nu$. Furthermore, in the same way 
		we have
		\begin{equation}\label{eq_addWaki}
			\Wak{\lambda_1}\Wak{\lambda_2}\cong \Wak{\lambda_1+\lambda_2}
		\end{equation}
		These are the objects categorifying the lattice part of the Hecke algebra, and our goal is 
		to study the subcategory they form. More precisely we want to understand the morphism spaces:
		\begin{equation}\label{eq_morphWaki}
			\Hom_{\Kb(\D)}(\Wak{\lambda_1},\Wak{\lambda_2}[i]), \quad i\in \ZZ.
		\end{equation}
		Then notice that \eqref{eq_addWaki} implies that $\Wak{-\lambda}$ is dual to $\Wak{\lambda}$, hence
		one can reduce to the case $\lambda_2=0$,
		which means $\Wak{\lambda_2}=\un$.
%
%
%
%
	\begin{exa}
		The Wakimoto sheaves in type $\tilde{A_1}$ are, for $n$ a non-negative integer,
		\[
			\Wak{n}	= F_{\mathbf{t}_{2n}}\quad\text{and}\quad	\Wak{-n}= F_{\mathbf{t}_{2n}^{-1}}
		\]
		and by the above, we will only have to consider $\Hom(\Wak{n},\un[i])$. Furthermore, Rouquier formula
		implies that the result is zero for positive $n$.
	\end{exa}
	To compute this morphism space, we will choose appropriate representatives for the 
	Wakimoto sheaves in the category of complexes and study the cohomology of the dg module of morphisms in the dg category of complexes.
	We will use the notation $\Homb$ to indicate the complex of morphisms in the dg category of complexes.
%
%
		\section{The characteristic zero case}
		Under the assumption that $\Bbbk$ is of characteristic zero, the category $\D$ has 
		nice decomposition properties: for instance, in type $\tilde{A_1}$, one can compute the 
		minimal subcomplexes of the Wakimoto sheaves quite easily.

		More precisely, in this section we assume 
		\begin{align}\label{eq_asschr0}
			&\chr(\kk)=0,& &\qn{n}{s},\qn{n}{t}\in\kk^\times, \, \forall n\in \ZZ_{>0}.&
		\end{align}
		Notice that for the standard Cartan matrix, 
	 	the two-colored quantum numbers are just the integers:
 		\[
 			\qn{n}{s}=\qn{n}{t}=n.
	 	\]		
		So, in this case, the second assumption in \eqref{eq_asschr0} is automatic from the first.
		\subsection{More on two-colored quantum numbers}
		We will use the following properties, which can be easily proved by induction: 
		if $n$ is odd and $m\ge n$, then
		\begin{align*}
			&[n][m]_x = \sum_{k=1}^n	[m-n+2k-1]_x, &
			&[n][m]_y = \sum_{k=1}^n	[m-n+2k-1]_y.
		\end{align*}
		If instead $n$ is even then
		\begin{align*}
			&[n]_y[m]_x = \sum_{k=1}^n	[m-n+2k-1]_y, &
			&[n]_x[m]_y = \sum_{k=1}^n	[m-n+2k-1]_x. 
		\end{align*}
		In words, for an appropriate 
		choice of colors, the product of the two-colored quantum numbers corresponding to $n$ and $m$
		can be written as a sum of 
		an increasing sequence of $n$ two-colored quantum numbers, centered at $m$. Note that the recursive definition
		is the particular case $n=2$.

		We also define (two-colored quantum) factorials,
		\[
			[n]_x^!=[n]_x[n-1]_x\cdots [1]_x,
		\]
		and binomial coefficients,
		\[
			\qbc{n}{k}{x} = \frac{[n]_x^!}{[k]_x^![n-k]_x^!},
		\]
		and similarly for $y$. One can show that these are still elements of $\ZZ[x,y]$.
		\subsection{Minimal subcomplexes of the Wakimoto sheaves}
		As we mentioned above, under assumptions \eqref{eq_asschr0}, one can actually 
		compute the minimal subcomplex $\Wak{n}$
		of the Wakimoto sheaves in type $\tilde{A_1}$. Recall that this is a homotopy equivalent
		summand with no contractible summands: for complexes in $\D$, when
		$\kk$ is a field, or, more generally, a complete local ring, this property 
		identifies a unique	complex up to isomorphism. 
		The minimal subcomplexes $F_{\mathbf{t}_{n}}$ (and $F_{\mathbf{s}_n}$) of $F_{\ut_{n}}$ 
		(and $F_{\us_{n}}$ respectively) were computed in \cite{AMRW_free} for the (infinite) 
		dihedral group, so we will refer to this result.
		The case of $F_{\mathbf{t}_n^{-1}}$ (and $F_{\mathbf{s}_n^{-1}}$) is symmetric, but for 
		convenience of application of this result, we prefer to use the positive complexes. 
		So we compute the cohomology of 
		\begin{equation}\label{eq_homwakimin}
			\Homb(\un,\Wak{n})	
		\end{equation}
		with $n\ge 0$, which, by the remarks in the previous sections, is the same as $\Homb(\Wak{-n},\un)$.
		
		The assumptions \eqref{eq_asschr0} allow to define, for all $m$, the so-called 
		Jones-Wenzl morphism, that we represent in the following way:
		
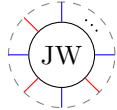
\begin{center}
			\begin{tikzpicture}[baseline=0]
				\def\radius{.7cm}
				\draw[dashed,gray] (0,0) circle (\radius);
				\node[draw,circle, inner sep=.1cm,fill=white] (a) at (0,0) {$\JW$};
				\foreach \i in {0,-90,-180,-270}{%
					\draw[blue] (\i:\radius)--(a);
					}
				\foreach \i in {-45,-135,-225}{%
					\draw[red] (\i:\radius)--(a);
					}
				\foreach \i in {35,45,55}{%
					\fill (\i:0.8*\radius) circle (.4pt);
					}
			\end{tikzpicture}
		\end{center}
		with $2m-2$ strands around the circle.
		These morphisms are related with the two colored Temperley-Lieb category and 
		can be described inductively:
		later, in the proof of a technical lemma, we will give one of the known recursive formulas for them.
		For other such formulas, see \cite{Elias_two} or \cite{EMTW}.
		
		From each Jones-Wenzl morphism, one can build an idempotent as follows
		\begin{align*}
			 &\JW_{\us_n}=%
			 \begin{tikzpicture}[scale=.8,baseline=-2,x=.9cm]
			  \foreach \n in {6}{
			   \node at (0,0) (JW) {$\JW_{\us_n}$};
			   \ifnum \n=1 \draw ({-0.5},-0.4)--({-0.5},0.4)--({0.5},0.4)--({0.5},-0.4)--cycle; 
			   	[\else \draw ({-0.25*\n},-0.4)--({-0.25*\n},0.4)--({0.25*\n},0.4)--({0.25*\n},-0.4)--cycle;]\fi
			   \foreach \i in {1,2,3}{
			    \ifthenelse{\isodd{\i}}
			    	{\draw[red] ({0.25-0.25*\n+(0.5*(\i-1))},0.4) -- ++(0,0.5);
			    	\draw[red] ({0.25-0.25*\n+(0.5*(\i-1))},-0.4) -- ++(0,-0.5);
					}
			    	{\draw[blue] ({0.25-0.25*\n+(0.5*(\i-1))},0.4) -- ++(0,0.5);
			    	\draw[blue] ({0.25-0.25*\n+(0.5*(\i-1))},-0.4) -- ++(0,-0.5);
					}
			   };
			   \foreach \i in {\n}{
			    \draw[violet] ({0.25-0.25*\n+(0.5*(\i-1))},0.4) -- ++(0,0.5);
			    \draw[violet] ({0.25-0.25*\n+(0.5*(\i-1))},-0.4) -- ++(0,-0.5);
			   };
			   \foreach \i in {1,2,3}{
			    \fill ({0.75-0.25*\n+(0.5*(1.5+0.5*\i))},0.65) circle (.5pt);
			    \fill ({0.75-0.25*\n+(0.5*(1.5+0.5*\i))},-0.65) circle (.5pt);
			   };
			  };
			 \end{tikzpicture}
			 =
			\begin{tikzpicture}[baseline=0,x=.7cm,y=.8cm]
				\node[draw,circle, inner sep=.1cm,fill=white] (a) at (0,0) {$\JW$};
				\draw[red] (a)--(-1,0);\draw[red] (-1,-1)--(-1,1);
				\draw[blue] (a) ..controls (-.5,.5).. (-.5,1);
				\draw[blue] (a) ..controls (-.5,-.5).. (-.5,-1);
				\draw[red] (a)--(0,1);\draw[red] (a)--(0,-1);
				\draw[violet] (a)--(1,0); \draw[violet] (1,-1)--(1,1);
				\node at (.5,.75) {\dots};\node at (.5,-.75) {\dots};
			\end{tikzpicture}\\
			 &\JW_{\ut_n}=%
			 \begin{tikzpicture}[scale=.8,baseline=-2,x=.9cm]
			  \foreach \n in {6}{
			   \node at (0,0) (JW) {$\JW_{\ut_n}$};
			   \ifnum \n=1 \draw ({-0.5},-0.4)--({-0.5},0.4)--({0.5},0.4)--({0.5},-0.4)--cycle; 
			   	[\else \draw ({-0.25*\n},-0.4)--({-0.25*\n},0.4)--({0.25*\n},0.4)--({0.25*\n},-0.4)--cycle;]\fi
			   \foreach \i in {1,2,3}{
			    \ifthenelse{\isodd{\i}}
			    	{\draw[blue] ({0.25-0.25*\n+(0.5*(\i-1))},0.4) -- ++(0,0.5);
			    	\draw[blue] ({0.25-0.25*\n+(0.5*(\i-1))},-0.4) -- ++(0,-0.5);
					}
			    	{\draw[red] ({0.25-0.25*\n+(0.5*(\i-1))},0.4) -- ++(0,0.5);
			    	\draw[red] ({0.25-0.25*\n+(0.5*(\i-1))},-0.4) -- ++(0,-0.5);
					}
			   };
			   \foreach \i in {\n}{
			    \draw[violet] ({0.25-0.25*\n+(0.5*(\i-1))},0.4) -- ++(0,0.5);
			    \draw[violet] ({0.25-0.25*\n+(0.5*(\i-1))},-0.4) -- ++(0,-0.5);
			   };
			   \foreach \i in {1,2,3}{
			    \fill ({0.75-0.25*\n+(0.5*(1.5+0.5*\i))},0.65) circle (.5pt);
			    \fill ({0.75-0.25*\n+(0.5*(1.5+0.5*\i))},-0.65) circle (.5pt);
			   };
			  };
			 \end{tikzpicture}
			 =
			\begin{tikzpicture}[baseline=0,x=.7cm,y=.8cm]
				\node[draw,circle, inner sep=.1cm,fill=white] (a) at (0,0) {$\JW$};
				\draw[blue] (a)--(-1,0);\draw[blue] (-1,-1)--(-1,1);
				\draw[red] (a) ..controls (-.5,.5).. (-.5,1);
				\draw[red] (a) ..controls (-.5,-.5).. (-.5,-1);
				\draw[blue] (a)--(0,1);\draw[blue] (a)--(0,-1);
				\draw[violet] (a)--(1,0); \draw[violet] (1,-1)--(1,1);
				\node at (.5,.75) {\dots};\node at (.5,-.75) {\dots};
			\end{tikzpicture}
		\end{align*}
		Here and below we will use violet to represent a non-specified color among red and blue: in this case
		the actual color depends on the parity of $n$.
		
		These idempotents give the maximal indecomposable summands inside the corresponding Bott-Samelson objects.
		More precisely, let $w\in W$ and let $\uw$ be the (unique) reduced word corresponding to it. 
		Then the idempotent $\JW_{\uw}$ defined above identifies the summand $B_{w}$ of $B_{\uw}$.
			
		Let us define the complex
		$F_{\mathbf{t}_n}$ to be 
		\begin{equation}\label{complex}
		%
		%
		\begin{tikzcd}[column sep=small,font=\small]
			 					& B_{\mathbf{t}_{n-1}}(1)\ar[r]\ar[rdd]		  	& B_{\mathbf{t}_{n-2}}(2) \ar[rdd]\ar[r]			& \dots\dots\dots \ar[r]\ar[rdd]	& B_{t}(n-1) \ar[rd]				&		\\
		 B_{\mathbf{t}_n}\ar[ru]\ar[rd]& \oplus		 					   	 	& \oplus			 						& \dots\dots 						&\oplus								& R(n) \\
			 					& B_{\mathbf{s}_{n-1}}(1)\ar[ruu]\ar[r]  	& B_{\mathbf{s}_{n-2}}(2) \ar[ruu]\ar[r]	& \dots\dots\dots \ar[r]\ar[ruu] 	& B_{s}(n-1) \ar[ru]	&
		\end{tikzcd} 
		\end{equation}
		where each arrow 
		$\phi_{w,u}=B_{w}(n-\ell(w))\rightarrow B_{u}(n-\ell(u))$ 
		is given, in terms of the 
		corresponding Bott-Samelson objects (via the Jones-Wenzl idempotents), 
		as follows:
		\begin{align}\label{diagphi}
			\phi_{\mathbf{t}_k,\mathbf{t}_{k-1}}= \quad &
			 (-1)^{k+1}\,
			 \begin{tikzpicture}[baseline=0.5cm]
			  \JWbox{4}{1}{\ut_{k}}{0}{0}\upstrand{0}{0}{1}{blue}\upstrand{0}{0}{2}{red}\updots{3}\upstrand{0}{0}{5}{violet}\upstrand{0}{0}{6}{violet}\updot{6}{violet}
			  \downstrand{1}{blue}\downstrand{2}{red}\downdots{3}\downstrand{5}{violet}\downstrand{6}{violet}
			  \JWbox{3}{1}{\ut_{k-1}}{0}{0.9}\upstrand{0}{0.9}{1}{blue}\upstrand{0}{0.9}{2}{red}\upstrand{0}{0.9}{5}{violet}
			 \end{tikzpicture} &
			 \phi_{\mathbf{t}_k,\mathbf{s}_{k-1}}= \quad &
			  \begin{tikzpicture}[baseline=0.5cm]
			  \JWbox{4}{1}{\ut_{k}}{0}{0}\upstrand{0}{0}{1}{blue}\upstrand{0}{0}{2}{red}\updots{3}\upstrand{0}{0}{5}{violet}\upstrand{0}{0}{6}{violet}\updot{1}{blue}
			  \downstrand{1}{blue}\downstrand{2}{red}\downdots{3}\downstrand{5}{violet}\downstrand{6}{violet}
			  \JWbox{3}{1}{\us_{k-1}}{0.36}{0.9}\upstrand{0.36}{0.9}{1}{red}\upstrand{0.36}{0.9}{4}{violet}\upstrand{0.36}{0.9}{5}{violet}
			 \end{tikzpicture}	\notag	 \\ \hfill \\ \notag
			\phi_{\mathbf{s}_k,\mathbf{s}_{k-1}}=\quad &	
			 (-1)^{k+1}\,
			 \begin{tikzpicture}[baseline=0.5cm]
			  \JWbox{4}{1}{\us_{k}}{0}{0}\upstrand{0}{0}{1}{red}\upstrand{0}{0}{2}{blue}\updots{3}\upstrand{0}{0}{5}{violet}\upstrand{0}{0}{6}{violet}\updot{6}{violet}
			  \downstrand{1}{red}\downstrand{2}{blue}\downdots{3}\downstrand{5}{violet}\downstrand{6}{violet}
			  \JWbox{3}{1}{\us_{k-1}}{0}{0.9}\upstrand{0}{0.9}{1}{red}\upstrand{0}{0.9}{2}{blue}\upstrand{0}{0.9}{5}{violet}
			 \end{tikzpicture} &
			\phi_{\mathbf{s}_k,\mathbf{t}_{k-1}}= \quad & 
			 \begin{tikzpicture}[baseline=0.5cm]
			  \JWbox{4}{1}{\us_{k}}{0}{0}\upstrand{0}{0}{1}{red}\upstrand{0}{0}{2}{blue}\updots{3}\upstrand{0}{0}{5}{violet}\upstrand{0}{0}{6}{violet}\updot{1}{red}
			  \downstrand{1}{red}\downstrand{2}{blue}\downdots{3}\downstrand{5}{violet}\downstrand{6}{violet}
			  \JWbox{3}{1}{\ut_{k-1}}{0.36}{0.9}\upstrand{0.36}{0.9}{1}{blue}\upstrand{0.36}{0.9}{4}{violet}\upstrand{0.36}{0.9}{5}{violet}
			 \end{tikzpicture}  
		\end{align}
		The fact that this is indeed a complex follows from
		the properties of the Jones-Wenzl idempotents (see for example \cite{AMRW_free} or \cite{EW}), 
		which imply in particular that a smaller Jones-Wenzl idempotent is \textit{swallowed} by a bigger one\footnote{%
		Notice that this implies that we could simplify the diagrams for the $\phi$'s in \eqref{diagphi}.
		We displayed them in this way to highlight the considered summands inside the Bott-Samelson object.}. 
		One can define $F_{\mathbf{s}_n}$ similarly (the formulas for the $\phi$'s being the same).
		
		Then we have (see \cite[\S 9]{AMRW_free}):
		\begin{thm}
			For $w\in W$ and $\uw$ its unique reduced word, the minimal subcomplex of $F_{\uw}$ is $F_w$.
		\end{thm}
		Hence, in particular, the minimal subcomplex of a dominant Wakimoto sheaf are $\Wak{n}=F_{\mathbf{t}_{2n}}$.
		\subsection{Consequences of Soergel's conjecture}
		Our next ingredient is the nice decomposition behavior of the Hecke category in this case, which 
		is encoded in Soergel's conjecture (see \cite{Soe_HC} and \cite{EW}) which was proved in \cite{EWHodge}.
		
		In type $\tilde{A}_1$, as Kazhdan-Lusztig polynomials are all trivial, this just says that 
		\begin{equation}\label{eq_Homfromun}
			\Hom(\un,B_{w})=R(-\ell(w)).					
		\end{equation}
		 Furthermore, a basis for $\Hom(\un,B_{w})$ is given,
		 on the Bott-Samelson level, by the
		 morphism
		 \[
		  \begin{tikzpicture}
		   \JWbox{3}{1}{\uw}{0}{0}\upstrand{0}{0}{1}{violet}\upstrand{0}{0}{2}{violet}\updots{3}
		   \upstrand{0}{0}{5}{violet}
		   \downstrand{1}{violet}\downstrand{2}{violet}\downdots{3}\downstrand{5}{violet}
		   \downdot{1}{violet}\downdot{2}{violet}\downdot{5}{violet}
		  \end{tikzpicture}
		 \]
	 	which has degree $\ell(w)$. In particular the basis of $\Hom(\un,\un)$ is the identity.
		\subsection{Extension groups}\label{subs_extchar0}
		We can now compute the extension groups under our assumptions.
		The complex $\Homb(\un,F_{\mathbf{t}_{n}})$ is, by \eqref{complex}, the following:
	 	\[
		  	\begin{tikzcd}[column sep=tiny, row sep=small,font=\footnotesize]
		 											& \Hom(\un,B_{\mathbf{t}_{n-1}})(1)\ar[r]\ar[rdd]	& 
		 																							\dots\dots\dots \ar[r]\ar[rdd]	& \Hom(\un,B_{t})(n-1) \ar[rd]					\\
		   \Hom(\un,B_{\mathbf{t}_n})\ar[ru]\ar[rd] 	& \oplus 				    				& 
		   																							 \dots\dots 						& \oplus						& \Hom(\un,\un)(n) 	\\
		 											& \Hom(\un,B_{\mathbf{s}_{n-1}})(1)\ar[ruu]\ar[r]  	& 
		 																							\dots\dots\dots \ar[r]\ar[ruu]	& \Hom(\un,B_{s})(n-1) \ar[ru]
		  \end{tikzcd}
		\]
		Hence, by \eqref{eq_Homfromun}, this becomes
	 	\begin{equation}\label{complexR}
		  \begin{tikzcd}
			 					& R(-n+2)\ar[r]\ar[rdd]	& 
			 												\dots\dots\dots \ar[r]\ar[rdd]	& R(n-2) \ar[rd]			\\
		   R(-n)\ar[ru]\ar[rd]	& \oplus 				& 
		   													\dots\dots 						& \oplus			& R(n) 	\\
			 					& R(-n+2)\ar[ruu]\ar[r]	& 
			 												\dots\dots\dots \ar[r]\ar[ruu] 	& R(n-2) \ar[ru]
		  \end{tikzcd} 
		\end{equation}
	 	Here, each arrow is the multiplication by a certain polynomial that we now want to determine. 
	
	 	We have the following recursive formula, which we
	 	shall prove later, for the elements of $R$ obtained by putting dots everywhere above and under a Jones-Wenzl
	 	morphism.
	 	\begin{lem}\label{lem_prodroots}
	 		Let $u<w$ be elements of $W$ with $\ell(w)=\ell(u)+1$. 
			Then
		  	\begin{equation}\label{recform2}
			   \begin{tikzpicture}[baseline=0.2cm]
			    \JWbox{4}{1}{\uw}{0}{0} 
			    \upstrand{0}{0}{1}{violet} \upstrand{0}{0}{2}{violet} \updots{3} \upstrand{0}{0}{5}{violet} 
			    \upstrand{0}{0}{6}{violet}
			    \downstrand{1}{violet} \downstrand{2}{violet} \downdots{3} \downstrand{5}{violet} \downstrand{6}{violet}
			    \updot{1}{violet}\updot{2}{violet}  
			    \updot{5}{violet}\updot{6}{violet}
			    \downdot{1}{violet}\downdot{2}{violet}  
			    \downdot{5}{violet}\downdot{6}{violet}
			   \end{tikzpicture}
			   =
			   q_{w,u}\,
			   \begin{tikzpicture}[baseline=0.2cm]
			    \JWbox{3}{1}{\uu}{0}{0}
			    \upstrand{0}{0}{1}{violet} \upstrand{0}{0}{2}{violet} \updots{3} 	\upstrand{0}{0}{5}{violet} 
			    \downstrand{1}{violet} 	\downstrand{2}{violet} 	\downdots{3} 	\downstrand{5}{violet}
			    \updot{1}{violet}			\updot{2}{violet}  						\updot{5}{violet}
			    \downdot{1}{violet}		\downdot{2}{violet}   					\downdot{5}{violet}
			   \end{tikzpicture}
		  	\end{equation}
		  	where the coefficient $q_{w,u}\in R$ is as follows:
		  	\begin{enumerate}
		  		\item if $\ell(u)=2r$ is even, then 
		  			\[
		  				q_{w,u}=\begin{cases}
		  							\frac{\qn{r}{s}}{\qn{2r}{s}}\alpha_{\mathbf{s}_n} & \text{if $w=\mathbf{s}_n$},\\[.5em]
		  							\frac{\qn{r}{t}}{\qn{2r}{t}}\alpha_{\mathbf{t}_n} & \text{if $w=\mathbf{t}_n$};\\	  							
		  						\end{cases}
		  			\]
		  		\item if $\ell(u)=2r+1$ is odd, then 
		  			\[
		  				q_{w,u}=\begin{cases}
		  							\frac{\qn{r+1}{t}}{\qn{2r+1}{t}}\alpha_{\mathbf{s}_n} & \text{if $u=\mathbf{t}_n$},\\[.5em]
		  							\frac{\qn{r+1}{s}}{\qn{2r+1}{s}}\alpha_{\mathbf{t}_n} & \text{if $u=\mathbf{s}_n$}.\\	  							
		  						\end{cases}
		  			\]
		  	\end{enumerate} 		
		  	In other words, up to a scalar, the coefficient is the root corresponding to the only 
		  	reflection which is a subword of $\uw$ but not of $\uu$.
		 \end{lem}
	 	Let us assume this lemma and go back to the complex \eqref{complexR}.
	 	The arrow $\Hom(\un,B_w)\rightarrow\Hom(\un,B_u)$ is given by a polynomial $p_{w,u}$ such that
	 	\[
			 \begin{tikzpicture}[baseline=0.5cm]
			  \JWbox{4}{1}{\uw}{0}{0}\upstrand{0}{0}{1}{violet}\upstrand{0}{0}{2}{violet}\updots{3}\upstrand{0}{0}{5}{violet}\upstrand{0}{0}{6}{violet}
			  \downstrand{1}{violet}\downstrand{2}{violet}\downdots{3}\downstrand{5}{violet}\downstrand{6}{violet}
			  \downdot{1}{violet}\downdot{2}{violet}\downdot{5}{violet}\downdot{6}{violet}
			 \draw (0,.9)--({0.36*(4+2*1)},.9)--({0.36*(4+2*1)},.9+0.6)--(0,.9+0.6)--cycle;
						 \node at ({0.18*(4+2*1)},.9+0.3) (phi) {$\phi_{w,u}$};
				\upstrand{0}{0.9}{1.5}{violet}\upstrand{0}{0.9}{2.5}{violet}\upstrand{0}{0.9}{5.5}{violet}
				\begin{scope}[yshift=.9cm]
					\updots{3.5}
				\end{scope}
			 \end{tikzpicture}
			 =
			 p_{w,u}\,
			  \begin{tikzpicture}[baseline=0.2cm]
			   \JWbox{3}{1}{\uu}{0}{0}\upstrand{0}{0}{1}{violet}\upstrand{0}{0}{2}{violet}\updots{3}
			   \upstrand{0}{0}{5}{violet}
			   \downstrand{1}{violet}\downstrand{2}{violet}\downdots{3}\downstrand{5}{violet}
			   \downdot{1}{violet}\downdot{2}{violet}\downdot{5}{violet}
			  \end{tikzpicture}
	 	\]
	 	Now it suffices to add dots on all top strands, use the expressions of $\phi_{w,u}$ in \eqref{diagphi}, and then compare the result with the formula of the lemma
	 	to get
	 	\begin{align*}
	 		&(-1)^{k+1} p_{\mathbf{t}_k,\mathbf{t}_{k-1}} = q_{\mathbf{t}_k,\mathbf{t}_{k-1}}& 				&p_{\mathbf{t}_k,\mathbf{s}_{k-1}} = q_{\mathbf{t}_k,\mathbf{s}_{k-1}} \\
	 		&(-1)^{k+1} p_{\mathbf{s}_k,\mathbf{s}_{k-1}} =  q_{\mathbf{s}_k,\mathbf{s}_{k-1}}& 	&p_{\mathbf{s}_k,\mathbf{t}_{k-1}} = q_{\mathbf{s}_k,\mathbf{t}_{k-1}} 
	 	\end{align*} 		
	 	We can now compute the cohomology of the complex \eqref{complexR}: 
	 	suppose first that $n$ is even (which is the case of Wakimoto sheaves). 
	 	Now, take $1<k<n-1$. If $k$ is even and $n-k=2r$, the complex at degree $k$ looks as follows:
	 	\[
	 		\begin{tikzcd}
	 			R(-n+2k-2)^{\oplus 2} \ar[r,"f"] & R(-n+2k)^{\oplus 2} \ar[r,"g"] & R(-n+2k+2)^{\oplus 2},
	 		\end{tikzcd} 	
	 	\]
	 	where
	 	\begin{align*}
	 		f=& \begin{pmatrix} \frac{\qn{r}{t}}{\qn{2r}{t}}\alpha_{\mathbf{t}_{n-k+1}} & \frac{\qn{r}{s}}{\qn{2r}{s}}\alpha_{\mathbf{s}_{n-k+1}} \\[.5em] \frac{\qn{r}{t}}{\qn{2r}{t}}\alpha_{\mathbf{t}_{n-k+1}} & \frac{\qn{r}{s}}{\qn{2r}{s}}\alpha_{\mathbf{s}_{n-k+1}} 	\end{pmatrix}, \\
	 		g=& \begin{pmatrix} -\frac{\qn{r}{t}}{\qn{2r-1}{t}}\alpha_{\mathbf{s}_{n-k-1}} & \frac{\qn{r}{t}}{\qn{2r-1}{t}}\alpha_{\mathbf{s}_{n-k-1}} \\[.5em] \frac{\qn{r}{s}}{\qn{2r-1}{s}}\alpha_{\mathbf{t}_{n-k-1}} & -\frac{\qn{r}{s}}{\qn{2r-1}{s}}\alpha_{\mathbf{t}_{n-k-1}} \end{pmatrix}.
	 	\end{align*}
	 	The cohomology at degree $k$ is then
	 	\[
	 		\begin{split}
		 		\frac{R}{\bigg(\frac{\qn{r}{s}}{\qn{2r}{s}}\alpha_{\mathbf{s}_{n-k+1}},%
		 			\frac{\qn{r}{t}}{\qn{2r}{t}}\alpha_{\mathbf{t}_{n-k+1}}\bigg)} (-n+2k)%
		 				& = \frac{R}{(\alpha_{\mathbf{s}_{n-k+1}}, \alpha_{\mathbf{t}_{n-k+1}})} (-n+2k) \\ 
		 				& = \frac{R}{(\alpha_s,\alpha_t)}(-n+2k). 	
		 	\end{split}
	 	\]
	 	If $\alpha_s$ and $\alpha_t$ generate $\mathfrak{h}^*$ then this is $\kk(-n+2k)$.
	 	
		If instead $k$ is odd and $n-k=2r+1$, the complex at degree $k$ has the form
	 	\[
	 		\begin{tikzcd}
	 			R(-n+2k-2)^{\oplus 2} \ar[r,"f"] & R(-n+2k)^{\oplus 2} \ar[r,"g"] & R(-n+2k+2)^{\oplus 2}
	 		\end{tikzcd} 	
	 	\]
	 	where
	 	\begin{align*}
	 		f=& \begin{pmatrix} -\frac{\qn{r+1}{t}}{\qn{2r+1}{t}}\alpha_{\mathbf{s}_{n-k}} & \frac{\qn{r+1}{t}}{\qn{2r+1}{t}}\alpha_{\mathbf{s}_{n-k}} \\[.5em] \frac{\qn{r+1}{s}}{\qn{2r+1}{s}}\alpha_{\mathbf{t}_{n-k}} & -\frac{\qn{r+1}{s}}{\qn{2r+1}{s}}\alpha_{\mathbf{t}_{n-k}} \end{pmatrix} \\
	 		g=& \begin{pmatrix} \frac{\qn{r}{t}}{\qn{2r}{t}}\alpha_{\mathbf{t}_{n-k}} & \frac{\qn{r}{s}}{\qn{2r}{s}}\alpha_{\mathbf{s}_{n-k}} \\[.5em] \frac{\qn{r}{t}}{\qn{2r}{t}}\alpha_{\mathbf{t}_{n-k}} & \frac{\qn{r}{s}}{\qn{2r}{s}}\alpha_{\mathbf{s}_{n-k}} 	\end{pmatrix} 
	 	\end{align*}
	 	and hence the cohomology at degree $k$ is 0.
	 	It is easy to adapt the above computation to the cases $k=0,1$ and $k=n-1,n$.
	 	The results are summarized in the following table (where empty cells are zeroes).
	  \begin{equation*}
	   \begin{array}{|c|c|c|c|c|c|c|c|c|c|c|c}
	   \hline
	    \hbox{\diagbox{$n$}{$j$}}							 	& 0 	& 1 																	& 2 								& 3 								& 4 								& 5 								& 6 								& 7									& 8									& 9						& \dots			\\
	    \hline   																																																																																																		
	    0 							& R		&																		& 									&									&									&									&									&									&									&								&	\\
	    \hline   																																																																																
	    2								&		&																		& \kk(2)	&									&									&									&									&									&									&								&	\\
	    \hline   																																																																																																		
	    4								&		&																		& \kk(0)		&									& \kk(4)	&									&									&									&									&							&		\\
	    \hline   																																																																																																		
	    6								&		&																		& \kk(-2)	&									& \kk(2)	&									& \kk(6)  &									&									&							&		\\
	    \hline																																																																																																			
		8								&		& 																		& \kk(-4)	&									& \kk(0)		&									& \kk(4)	&									& \kk(8)	&							&		\\
		\hline																																																																																																			
	    \dots 							& \dots & \dots 																& \dots 							& \dots 							& \dots 							& \dots 							& \dots								& \dots								& \dots 							& \dots					& \dots			\\
	   \end{array}
	  \end{equation*}
	  (for simplicity, the table corresponds to the case $R/(\alpha_s,\alpha_t)=\kk$).
	  
	  When $n$ is odd we can proceed similarly and obtain an analogous parity vanishing of the cohomology
	  (in this case in even degrees),
	  but we get something different for $k=1$. In fact, if
	 $n=2r+1$ is odd, the beginning of the complex has the form
	 	\begin{center}
	 		\begin{tikzpicture}
	 			\def \d{4cm}
	 			\node (k0) at (0,0) {$R(-n)$}; \node (k1) at (.8*\d,0) {$R(-n+2)^{\oplus 2}$}; \node (k2) at (2.5*\d,0) {$R(-n+4)^{\oplus 2}$};
	 			\draw[->] (k0) to node[above] {$\begin{pmatrix} \frac{\qn{r}{t}}{\qn{2r}{t}}\alpha_{\mathbf{t}_{n}} \\[.5em] \frac{\qn{r}{t}}{\qn{2r}{t}}\alpha_{\mathbf{t}_{n}}	\end{pmatrix}$} (k1); 
	 			\draw[->] (k1) to node[above] {$\begin{pmatrix} -\frac{\qn{r}{t}}{\qn{2r-1}{t}}\alpha_{\mathbf{s}_{n-2}} & \frac{\qn{r}{t}}{\qn{2r-1}{t}}\alpha_{\mathbf{s}_{n-2}} \\[.5em] \frac{\qn{r}{s}}{\qn{2r-1}{s}}\alpha_{\mathbf{t}_{n-2}} & -\frac{\qn{r}{s}}{\qn{2r-1}{s}}\alpha_{\mathbf{t}_{n-2}} \end{pmatrix}$} (k2); 
	 		\end{tikzpicture} 	
	 	\end{center}
	 	which gives zero cohomology in degree 0 and $R/(\alpha_{\mathbf{t}_n})(-n+2)$ in degree 1. Hence the table one gets is the following.
	  \begin{equation*}
	   \begin{array}{|c|c|c|c|c|c|c|c|c|c|c|c}
	   	\hline
	    \hbox{\diagbox{$n$}{$j$}}				 	& 0 	& 1 																	& 2 								& 3 								& 4 								& 5 								& 6 								& 7									& 8									& 9						& \dots			\\
	    \hline   																																																																																																		
	    1								&		& \frac{R}{(\alpha_{t})}(1)												& 									&									&									&									&									&									&									&								&	\\
	    \hline   																																																																																																		
	    3								&		& \frac{R}{(\alpha_{\mathbf{t}_{3}})}(-1)	& 									& \kk(3)	&									&									&									&									&									&								&	\\
	    \hline   																																																																																																		
	    5								&		& \frac{R}{(\alpha_{\mathbf{t}_{5}})}(-3)	& 									& \kk(1)	&									& \kk(5)	&									&									&									&									& \\
	    \hline   																																																																																																		
	    7								&		& \frac{R}{(\alpha_{\mathbf{t}_{7}})}(-5)	& 									& \kk(-1)	&									& \kk(3)	&									& \kk(7)	&									&								&	\\
		\hline																																																																																																			
		9								&		& \frac{R}{(\alpha_{\mathbf{t}_{9}})}(-7)	&									& \kk(-3)	&									& \kk(1)	&									& \kk(5)	&									& \kk(9) &	\\
		\hline																																																																																																			
	    \dots 							& \dots & \dots 																& \dots 							& \dots 							& \dots 							& \dots 							& \dots								& \dots								& \dots 							& \dots						& \dots		\\
	   \end{array}
	  \end{equation*}
	  	One has, of course, similar tables for 
	 	$\Homb(\un,F_{\mathbf{s}_n})$.
		 We now turn to the proof of the recursive formula.
	 	\begin{proof}[Proof of Lemma \ref{lem_prodroots}] 
	 	We will suppose $w=\mathbf{s}_{n+1}$ and $u=\mathbf{s}_n$. The other cases can then be obtained by swapping $s$ and $t$ and/or applying horizontal reflection.
	 	
	 	For convenience, let $\alldJW_{n}$ denote the polynomial obtained by putting dots everywhere 
	 	above and under $JW_{\us_n}$:
	 	\[
	 		\alldJW_n:=
			   		\begin{tikzpicture}[baseline=0.2cm]
			    		\JWbox{3}{1}{\us_n}{0}{0} 
			    		\upstrand{0}{0}{1}{red} \upstrand{0}{0}{2}{blue} \updots{3} \upstrand{0}{0}{5}{violet} 
			    		\downstrand{1}{red} \downstrand{2}{blue} \downdots{3} \downstrand{5}{violet}
			    		\updot{1}{red}\updot{2}{blue}  
			    		\updot{5}{violet}
			    		\downdot{1}{red}\downdot{2}{blue}  
			    		\downdot{5}{violet}
			   		\end{tikzpicture}
	 	\]
	 	So we want to show that
	 	\[
	 		\alldJW_{n+1}=q_{\mathbf{s}_{n+1},\mathbf{s}_n}\alldJW_n
	 	\]
	  	We shall use the symbol 
	  	\[
	   		\ud{k}:=
		   \begin{cases}
		    s & \text{if $k$ is odd}\\ 
		    t & \text{otherwise}
		   \end{cases}
		 \]
	 	The lemma is a direct consequence of a recursive formula for Jones-Wenzl morphisms, which can
	 	be found in \cite[\S 8.2]{AMRW_free} (for clarity we put a numbering over the strands):
	 	\begin{multline}\label{recform}
	   		\begin{tikzpicture}[baseline=0.2cm]
	    		\JWbox{4}{1}{\us_{n+1}}{0}{0} 
	   	 		\uplongstrand{1}{red} \uplongstrand{2}{blue} \updots{3} \uplongstrand{5}{violet} \uplongstrand{6}{violet}
	    		\downlongstrand{1}{red} \downlongstrand{2}{blue} \downdots{3} \downlongstrand{5}{violet} \downlongstrand{6}{violet}
	    		\uppernode{1}{$1$}\uppernode{2}{$2$}  
	    		\uppernode{5}{$n$}\uppernode{6}{$n+1$}
	   		\end{tikzpicture}
	   		=
	   		\begin{tikzpicture}[baseline=0.2cm]
	    		\JWbox{3}{1}{\us_n}{0}{0}
	    		\uplongstrand{1}{red} \uplongstrand{2}{blue} \updots{3} \uplongstrand{5}{violet} 
	    		\downlongstrand{1}{red} \downlongstrand{2}{blue} \downdots{3} \downlongstrand{5}{violet}
	    		\longstrand{6}{violet}
	    		\uppernode{1}{$1$}\uppernode{2}{$2$}  
	    		\uppernode{5}{$n$}\uppernode{6}{$n+1$}
	   		\end{tikzpicture}
	   		+\frac{\qn{1}{t}}{\qn{n}{\ud{n+1}}}
	   		\begin{tikzpicture}[baseline=0.2cm]  
	    		\JWbox{4}{1}{\us_n}{0}{0}
	    		\upleftpitchfork{red}{blue} \uplongstrand{2}{blue} \updots{3} \uprightpitchfork{violet}{violet}{6}
	    		\downstrand{1}{red}\downdot{1}{red} \downlongstrand{2}{blue} \downdots{3} \downlongstrand{5}{violet} \downlongstrand{6}{violet}
	    		\uppernode{0}{$1$}\uppernode{0.5}{$2$}
	    		\uppernode{7}{$n+1$}
	   		\end{tikzpicture}
	   		+\\+\sum_{a=2}^{n-2}\frac{\qn{a}{\ud{a+1}}}{\qn{n}{\ud{n+1}}}
	   		\begin{tikzpicture}[baseline=0.2cm]
	    		\JWbox{5}{2}{\us_n}{0}{0}
	    		\uplongstrand{1}{red} \uplongstrand{2}{blue} \updots{3} \uppitchfork{5}{violet}{violet} \updots{6}
	    		\uprightpitchfork{violet}{violet}{9}
	    		\downlongstrand{1}{red} \downlongstrand{2}{blue} \downdots{3}\downdots{4.5}\downdots{6}
	    		\downlongstrand{8}{violet} \downlongstrand{9}{violet}
	    		\uppernode{1}{$1$}\uppernode{2}{$2$}\uppernode{4}{$a$}\uppernode{5}{$a+1$}\uppernode{6}{$a+2$}
	    		\uppernode{10}{$n+1$}
	   		\end{tikzpicture}
	   		+\frac{\qn{n-1}{\ud{n}}}{\qn{n}{\ud{n+1}}}
	   		\begin{tikzpicture}[baseline=0.2cm]
	    		\JWbox{4}{1}{\us_n}{0}{0}
	    		\uplongstrand{1}{red} \uplongstrand{2}{blue} \updots{3} \uprightdoublepitchfork{violet}{violet}{6}
	    		\downlongstrand{1}{red} \downlongstrand{2}{blue} \downdots{3} \downlongstrand{5}{violet} \downlongstrand{6}{violet}
	    		\uppernode{1}{$1$}\uppernode{2}{$2$}\uppernode{5}{$n-1$}\uppernode{6}{$n$}\uppernode{7}{$n+1$}
	   		\end{tikzpicture}  
	  	\end{multline}
	  	If we put dots everywhere in \eqref{recform}, 
	  	we obtain
	  	\begin{multline*}
	   		\begin{tikzpicture}[baseline=0.2cm]
	    		\JWbox{4}{1}{\us_{n+1}}{0}{0} 
	   	 		\uplongstrand{1}{red} \uplongstrand{2}{blue} \updots{3} \uplongstrand{5}{violet} \uplongstrand{6}{violet}
	    		\downlongstrand{1}{red} \downlongstrand{2}{blue} \downdots{3} \downlongstrand{5}{violet} \downlongstrand{6}{violet}
	    		\doubleupdot{1}{red}\doubleupdot{2}{blue}  
	    		\doubleupdot{5}{violet}\doubleupdot{6}{violet}
	    		\doubledowndot{1}{red}\doubledowndot{2}{blue}  
	    		\doubledowndot{5}{violet}\doubledowndot{6}{violet}
	   		\end{tikzpicture}
	   		=
	   		\begin{tikzpicture}[baseline=0.2cm]
	    		\JWbox{3}{1}{\us_n}{0}{0} 
	    		\uplongstrand{1}{red} \uplongstrand{2}{blue} \updots{3} \uplongstrand{5}{violet} 
	    		\downlongstrand{1}{red} \downlongstrand{2}{blue} \downdots{3} \downlongstrand{5}{violet}
	    		\longstrand{6}{violet}
	    		\doubleupdot{1}{red}\doubleupdot{2}{blue}  
	    		\doubleupdot{5}{violet}\doubleupdot{6}{violet}
	    		\doubledowndot{1}{red}\doubledowndot{2}{blue}  
	    		\doubledowndot{5}{violet}\doubledowndot{6}{violet}
	    		\uppernode{6}{$n+1$}
	   		\end{tikzpicture}
	   		+\frac{[1]_t}{\qn{n}{\ud{n+1}}}
	   		\begin{tikzpicture}[baseline=0.2cm]  
	    		\JWbox{4}{1}{\us_n}{0}{0} 
	    		\upleftpitchfork{red}{blue} \uplongstrand{2}{blue} \updots{3} \uprightpitchfork{violet}{violet}{6}
	    		\downstrand{1}{red}\downdot{1}{red} \downlongstrand{2}{blue} \downdots{3} \downlongstrand{5}{violet} \downlongstrand{6}{violet}
	    		\doubleupdot{0}{red}\doubleupdot{0.5}{blue}\doubleupdot{1}{red}\doubleupdot{2}{blue}
	    		\doubleupdot{7}{violet}
	    		\doubledowndot{0}{red}\doubledowndot{2}{blue}\doubledowndot{5}{violet}\doubledowndot{6}{violet}\doubledowndot{7}{violet}
	    		\uppernode{0.5}{$2$}
	   		\end{tikzpicture}
	   		+\\+\sum_{a=2}^{n-2}\frac{\qn{a}{\ud{a+1}}}{\qn{n}{\ud{n+1}}}
	   		\begin{tikzpicture}[baseline=0.2cm]
	    		\JWbox{5}{2}{\us_n}{0}{0} 
	    		\uplongstrand{1}{red} \uplongstrand{2}{blue} \updots{3} \uppitchfork{5}{violet}{violet} \updots{6}
	    		\uprightpitchfork{violet}{violet}{9}
	    		\downlongstrand{1}{red} \downlongstrand{2}{blue} \downdots{3}\downdots{4.5}\downdots{6}
	    		\downlongstrand{8}{violet} \downlongstrand{9}{violet}
	    		\doubleupdot{1}{red}\doubleupdot{2}{blue}\doubleupdot{4.5}{violet}\doubleupdot{5}{violet}\doubleupdot{5.5}{violet}
	    		\doubleupdot{10}{violet}
	    		\doubledowndot{1}{red}\doubledowndot{2}{blue}\doubledowndot{8}{violet}\doubledowndot{9}{violet}\doubledowndot{10}{violet}
	    		\uppernode{5}{$a+1$}
	   		\end{tikzpicture}
	   		+\frac{\qn{n-1}{\ud{n}}}{\qn{n}{\ud{n+1}}}
	   		\begin{tikzpicture}[baseline=0.2cm]
	    		\JWbox{4}{1}{\us_n}{0}{0} 
	    		\uplongstrand{1}{red} \uplongstrand{2}{blue} \updots{3} \uprightdoublepitchfork{violet}{violet}{6}
	    		\downlongstrand{1}{red} \downlongstrand{2}{blue} \downdots{3} \downlongstrand{5}{violet} \downlongstrand{6}{violet}
	    		\doubleupdot{1}{red}\doubleupdot{2}{blue}\doubleupdot{5}{violet}\doubleupdot{6}{violet}\doubleupdot{7}{violet}
	    		\doubledowndot{1}{red}\doubledowndot{2}{blue}\doubledowndot{5}{violet}\doubledowndot{6}{violet}\doubledowndot{7}{violet}
	    		\uppernode{6}{$n$}
	   		\end{tikzpicture}  
	  	\end{multline*}
	  	After applying relations \eqref{barbell} and \eqref{dotline}, we see that all the terms on the right
	  	are multiples of $\alldJW_n$ and we get
	  	\[	
	  		\begin{split}
		   		\alldJW_{n+1}	
				&=
			   		\bigg(\alpha_{\ud{n+1}}+\frac{\qn{1}{t}}{\qn{n}{\ud{n+1}}}\alpha_{t}
			   		+\sum_{a=2}^{n-2}\frac{\qn{a}{\ud{a+1}}}{\qn{n}{\ud{n+1}}}\alpha_{\ud{a+1}}
			   		+\frac{\qn{n-1}{\ud{n}}}{\qn{n}{\ud{n+1}}}\alpha_{\ud{n}}\bigg)\alldJW_n
			   		\\
			   	& = \bigg( %
			   		\frac{\qn{1}{t}\alpha_t+\qn{2}{s}\alpha_s+\dots+ %
			   			\qn{n-1}{\ud{n}}\alpha_{\ud{n}}+\qn{n}{\ud{n+1}}\alpha_{\ud{n+1}}}{\qn{n}{\ud{n+1}}} %
			   			\bigg)\alldJW_n
		   	\end{split}
	  	\]
	  	Now suppose that $n=2r$ is even: if we split the sum according to the color and we use 
	  	the properties of the two-colored quantum numbers, the coefficient becomes
	  	\[
	  		\begin{split}
			   \bigg( \frac{\qn{2}{s}+\qn{4}{s}+\dots +\qn{2r}{s}}{\qn{2r}{s}} \bigg)
			    & \alpha_{s} +\bigg(\frac{\qn{1}{t}+\qn{3}{t}+\dots +\qn{2r-1}{t}}{\qn{2r}{s}}\bigg)\alpha_{t}	\\
			   	& = \bigg(\frac{\qn{r}{s}\qn{r+1}{s}}{\qn{2r}{s}}\bigg)\alpha_{s}
				   +\bigg(\frac{\qn{r}{s}\qn{r}{t}}{\qn{2r}{s}}\bigg)\alpha_{t} \\
			  	& = \frac{\qn{r}{s}}{\qn{2r}{s}}\big(\qn{r+1}{s}\alpha_{s}+\qn{r}{t}\alpha_{t}\big)
			   	=\frac{\qn{r}{s}}{\qn{2r}{s}}\alpha_{\mathbf{s}_{n+1}}
			\end{split}
	  	\]
	  	If instead $n=2r+1$ is odd then it is
	  	\[
	  		\begin{split}	
			   \bigg(\frac{\qn{2}{s}+\qn{4}{s}+\dots +\qn{2r}{s}}{\qn{2r+1}{t}}\bigg)
			   	& \alpha_{s} +\bigg(\frac{\qn{1}{t}+\qn{3}{t}+\dots +\qn{2r+1}{t}}{\qn{2r+1}{t}}\bigg)\alpha_{t}	\\
		   		& = \bigg(\frac{\qn{r}{s}\qn{r+1}{s}}{\qn{2r+1}{t}}\bigg)\alpha_{s}
				   +\bigg(\frac{\qn{r+1}{s}\qn{r+1}{t}}{\qn{2r+1}{t}}\bigg)\alpha_{t} 	\\
		  		& = \frac{\qn{r+1}{s}}{\qn{2r+1}{s}}\big(\qn{r}{s}\alpha_{s}+\qn{r+1}{t}\alpha_{t}\big)
				   =\frac{\qn{r+1}{s}}{\qn{2r+1}{s}} \alpha_{\mathbf{t}_n}  
			\end{split}
	  	\]
	  	This concludes the proof.
	 	\end{proof} 
	 	Let us observe that
	 	\[
		  \frac{\qn{1}{s}}{\qn{1}{s}}\frac{\qn{2}{s}}{\qn{1}{s}}\frac{\qn{3}{s}}{\qn{2}{s}}
		  \frac{\qn{4}{s}}{\qn{2}{s}}\dots \frac{\qn{2r}{s}}{\qn{r}{s}}=\qbc{2r}{r}{s}
		\]
		and that
	 	\[
		  \frac{\qn{1}{s}}{\qn{1}{s}}\frac{\qn{2}{s}}{\qn{1}{s}}\frac{\qn{3}{s}}{\qn{2}{s}}
		  \frac{\qn{4}{s}}{\qn{2}{s}}\dots \frac{\qn{2r+1}{s}}{\qn{r+1}{s}}=\qbc{2r+1}{r+1}{s} 
	 	\]
		and similarly for $t$. Hence, by induction from the lemma, one 
	 	obtains the following formulas for the everywhere-dotted Jones-Wenzl morphisms.
	 	\begin{cor}
	 	Let $T$ be the set of reflections in $W$. Then, for $n\ge 0$, one has:
	  	\[
	  	\qbc{n-1}{\lceil \frac{n-1}{2}\rceil}{s}\,\,
	  	\begin{tikzpicture}[scale=.7,baseline=-2]
		   \foreach \n in {5}{
		    \node at (0,0) (JW) {$JW_{\us_n}$};
		    \ifnum \n=1 \draw ({-0.5},-0.4)--({-0.5},0.4)--({0.5},0.4)--({0.5},-0.4)--cycle;
		    	[\else \draw ({-0.25*\n},-0.4)--({-0.25*\n},0.4)--({0.25*\n},0.4)--({0.25*\n},-0.4)--cycle;]\fi
		   \foreach \i in {1,2}{
		    \ifthenelse{ \isodd{\i}}
		    	{\draw[red] ({0.25-0.25*\n+(0.5*(\i-1))},0.4) -- ++(0,0.5);
		    	\draw[red] ({0.25-0.25*\n+(0.5*(\i-1))},-0.4) -- ++(0,-0.5);
				\fill[red] ({0.25-0.25*\n+(0.5*(\i-1))},0.9) circle (2pt);
				\fill[red] ({0.25-0.25*\n+(0.5*(\i-1))},-0.9) circle (2pt);
				}
		    	{\draw[blue] ({0.25-0.25*\n+(0.5*(\i-1))},0.4) -- ++(0,0.5);
		    	\draw[blue] ({0.25-0.25*\n+(0.5*(\i-1))},-0.4) -- ++(0,-0.5);
				\fill[blue] ({0.25-0.25*\n+(0.5*(\i-1))},0.9) circle (2pt);
				\fill[blue] ({0.25-0.25*\n+(0.5*(\i-1))},-0.9) circle (2pt);
				}
		   };
		   \foreach \i in {\n}{
		    \draw[violet] ({0.25-0.25*\n+(0.5*(\i-1))},0.4) -- ++(0,0.5);
		    \draw[violet] ({0.25-0.25*\n+(0.5*(\i-1))},-0.4) -- ++(0,-0.5);
				\fill[violet] ({0.25-0.25*\n+(0.5*(\i-1))},0.9) circle (2pt);
				\fill[violet] ({0.25-0.25*\n+(0.5*(\i-1))},-0.9) circle (2pt);
		   };
		   \foreach \i in {1,2,3}{
		    \fill ({0.25-0.25*\n+(0.5*(1.5+0.5*\i))},0.65) circle (1pt);
		    \fill ({0.25-0.25*\n+(0.5*(1.5+0.5*\i))},-0.65) circle (1pt);
		   };
		  };
	 	\end{tikzpicture}=\prod\limits_{\substack{v\in T \\ x\le \mathbf{s}_n}}\alpha_x
	 	\qquad\qquad
	  	\qbc{n-1}{\lceil\frac{n-1}{2}\rceil}{t}\,\,
	 	\begin{tikzpicture}[scale=.7,baseline=-2]
		  \foreach \n in {5}{
		   \node at (0,0) (JW) {$JW_{\ut_n}$};
		   \ifthenelse{\n=1}
		   	{\draw ({-0.5},-0.4)--({-0.5},0.4)--({0.5},0.4)--({0.5},-0.4)--cycle;}
		   	{\draw ({-0.25*\n},-0.4)--({-0.25*\n},0.4)--({0.25*\n},0.4)--({0.25*\n},-0.4)--cycle;}
		   \foreach \i in {1,2}{
		    \ifthenelse{\not \isodd{\i}}
		    	{\draw[red] ({0.25-0.25*\n+(0.5*(\i-1))},0.4) -- ++(0,0.5);
		    	\draw[red] ({0.25-0.25*\n+(0.5*(\i-1))},-0.4) -- ++(0,-0.5);
				\fill[red] ({0.25-0.25*\n+(0.5*(\i-1))},0.9) circle (2pt);
				\fill[red] ({0.25-0.25*\n+(0.5*(\i-1))},-0.9) circle (2pt);
				}
		    	{\draw[blue] ({0.25-0.25*\n+(0.5*(\i-1))},0.4) -- ++(0,0.5);
		    	\draw[blue] ({0.25-0.25*\n+(0.5*(\i-1))},-0.4) -- ++(0,-0.5);
				\fill[blue] ({0.25-0.25*\n+(0.5*(\i-1))},0.9) circle (2pt);
				\fill[blue] ({0.25-0.25*\n+(0.5*(\i-1))},-0.9) circle (2pt);
				}
		   };
		   \foreach \i in {\n}{
		    \draw[violet] ({0.25-0.25*\n+(0.5*(\i-1))},0.4) -- ++(0,0.5);
		    \draw[violet] ({0.25-0.25*\n+(0.5*(\i-1))},-0.4) -- ++(0,-0.5);
				\fill[violet] ({0.25-0.25*\n+(0.5*(\i-1))},0.9) circle (2pt);
				\fill[violet] ({0.25-0.25*\n+(0.5*(\i-1))},-0.9) circle (2pt);
		   };
		   \foreach \i in {1,2,3}{
		    \fill ({0.25-0.25*\n+(0.5*(1.5+0.5*\i))},0.65) circle (1pt);
		    \fill ({0.25-0.25*\n+(0.5*(1.5+0.5*\i))},-0.65) circle (1pt);
		   };
		  };
	 	\end{tikzpicture}=\prod\limits_{\substack{x\in T \\ x\le \mathbf{t}_n}}\alpha_x
	 	\]
	 	where $\lceil \cdot \rceil:\mathbb{R}\rightarrow\ZZ$ is the ceiling function.
	 	\end{cor}
	 	\begin{rmk}
	 		Dhillon and Makam \cite{DhiMak} compute extensions between Verma modules for the dihedral groups
	 		in a graded version of category $\mathcal{O}$. 
	 		They consider the category of \emph{Soergel modules} which is obtained from 
	 		the Hecke category by killing on the right the maximal ideal generated by positive degree monomials.
	 		Then one can see the dihedral groups as quotients of the infinite dihedral group $\tilde{A}_1$. 
	 		This defines a functor under which the Wakimoto sheaves are sent to the Verma modules and
	 		if one repeats our computations in that category one can actually deduce the same result.
	 	\end{rmk}
		\section{The two color garden}
		We now drop assumption \eqref{eq_asschr0} and we consider an arbitrary realization $\hh$.

		In this generality we do not have minimal complexes, but in \cite{Mal_red}, one can find a 
		general reduction of any $F_{\uom}^{\bullet}$ to a 
		homotopy equivalent simpler complex $F_{\uom}$, which is then just a different representative 
		for the same Rouquier complex. 
		Let $\Wakr{n}$ denote this summand inside 
		$\Wak{n}^\bullet$, for $n\in \ZZ$.

		It is very hard to reduce further the complex
		$\Wakr{n}$ in this generality. In characteristic zero we could take advantage of the simplicity of 
		decomposing Bott-Samelson objects (essentially encoded in the fact that Kazhdan-Lusztig polynomials
		are all $1$ in this case).

		Nevertheless, from last chapter we have a description of the dg-module of morphisms
		\[
			C_n:=\Homb(\Wakr{-n},\un)
		\]
		in the category $\Cdg(\D)$. We will reduce this dg-module with some homological algebra, 
		strongly using the basis of light leaves maps.
		\subsection{DG diagrams for Wakimoto sheaves}
		A consequence of the reduction from \cite{Mal_rou} is that the space $C_n$ can be written, as a graded module, as
		\[
			C_n=\bigoplus_{\ux\preceq \us_{2n}} \Hom_{\D}(B_{\ux},\un)/\Lambda_{\ux}\Tate{q_{\ux}}
		\]
		where the sum is over all subwords of $\us_{2n}$ (which is the Coxeter projection of $(\sigma_t\sigma_s)^{-n}$),
		and $q_{\ux}=2n-\ell(\ux)$. 
		The symbol $\Lambda$ denotes the bottom-ideal generated by boundary trivalent vertices, of the following form. 
		
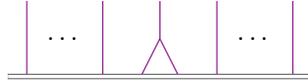
\[
			\begin{tikzpicture}
				\draw[violet] (\d,0)--(\d,\h);
				\node at (2*\d,.5*\h) {\dots};
				\draw[violet] (3*\d,0)--(3*\d,\h);
				\draw[violet] (4*\d,0)--(4.5*\d,.5*\h)--(5*\d,0);
				\draw[violet] (4.5*\d,.5*\h)--(4.5*\d,\h);
				\draw[violet] (6*\d,0)--(6*\d,\h);
				\node at (7*\d,.5*\h) {\dots};
				\draw[violet] (8*\d,0)--(8*\d,\h);
				\draw[negative] (0.5*\d,0)--(8.5*\d,0);
			\end{tikzpicture}
		\]	 
		Recall from \cite{Mal_red} that the differential map of $C_n$ acts by successively \textit{uprooting} strands, as in the following
		picture 
		
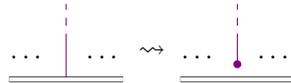
\[
			\begin{tikzpicture}[baseline=.3*\h]
				\node at (0,.3*\h) {\dots};
				\draw[violet] (\d,0)--(\d,.5*\h);\draw[violet,dashed] (\d,.5*\h)--(\d,\h);
				\node at (2*\d,.3*\h) {\dots};
				\draw[negative] (-.5*\d,0)--(2.5*\d,0);
			\end{tikzpicture}
			\rightsquigarrow
			\begin{tikzpicture}[baseline=.3*\h]
				\node at (0,.3*\h) {\dots};
				\fill[violet] (\d,.2*\h) circle (1.5pt); \draw[violet] (\d,.2*\h)--(\d,.5*\h);\draw[violet,dashed] (\d,.5*\h)--(\d,\h);
				\node at (2*\d,.3*\h) {\dots};
				\draw[negative] (-.5*\d,0)--(2.5*\d,0);
			\end{tikzpicture}
		\]
		The image of a diagram via the differential map is then a linear combination, with some signs,
		of all the diagrams obtained
		by uprooting a single strand connected to the bottom boundary.
		The sign is
		determined by the subword $\uu$ and the position of the strand being uprooted, as follows.
		Let $r$ be the number of missing colors on the right of the uprooted strand, with respect to
		the sequence $\dots ststst$, then
		the sign is given by $(-1)^{r+1}$.	
		\begin{exa}
			When applying the differential map to
			\begin{center}
				\begin{tikzpicture}[baseline=.5cm]
					\pic[red] at (\d,0) {dot={\hdot}};
					\pic[blue] at (2*\d,0) {dot={\hdot}};
					\pic[blue] at (3*\d,0) {arch={3*\d}{1*\h}};
					\pic[red] at (4*\d,0) {dot={\hdot}};
					\pic[red] at (5*\d,0) {dot={\hdot}};
					\draw[negative](0.5*\d,0)--(6.5*\d,0);
				\end{tikzpicture}	
			\end{center}
			we obtain
			\begin{multline*}
				\def\s{0.7}\def\b{.15cm}
				-\,
				\begin{tikzpicture}[baseline=\b,scale=\s, transform shape]
					\pic[red] at (\d,.2*\h) {dot={\hdot}}; \fill[red] (\d,.2*\h) circle (1.5pt);
					\pic[blue] at (2*\d,0) {dot={\hdot}};
					\pic[blue] at (3*\d,0) {arch={3*\d}{1*\h}};
					\pic[red] at (4*\d,0) {dot={\hdot}};
					\pic[red] at (5*\d,0) {dot={\hdot}};
					\draw[negative](0.5*\d,0)--(6.5*\d,0);
				\end{tikzpicture}	
				\,-\,
				\begin{tikzpicture}[baseline=\b,scale=\s, transform shape]
					\pic[red] at (\d,0) {dot={\hdot}};
					\pic[blue] at (2*\d,.2*\h) {dot={\hdot}}; \fill[blue] (2*\d,.2*\h) circle (1.5pt);
					\pic[blue] at (3*\d,0) {arch={3*\d}{1*\h}};
					\pic[red] at (4*\d,0) {dot={\hdot}};
					\pic[red] at (5*\d,0) {dot={\hdot}};
					\draw[negative](0.5*\d,0)--(6.5*\d,0);
				\end{tikzpicture}	
				\,+\,
				\begin{tikzpicture}[baseline=\b,scale=\s, transform shape]
					\pic[red] at (\d,0) {dot={\hdot}};
					\pic[blue] at (2*\d,0) {dot={\hdot}};
					\draw[blue] (3*\d,.2*\h) ..controls (3.3*\d,\h) and (6*\d,\h).. (6*\d,0); \fill[blue] (3*\d,.2*\h) circle (1.5pt);
					\pic[red] at (4*\d,0) {dot={\hdot}};
					\pic[red] at (5*\d,0) {dot={\hdot}};
					\draw[negative](0.5*\d,0)--(6.5*\d,0);
				\end{tikzpicture}	
				\,+ \\ 
				\def\s{0.7}\def\b{.15cm}
				+\,
				\begin{tikzpicture}[baseline=\b,scale=\s, transform shape]
					\def\hdott{.8*\hdot}
					\pic[red] at (\d,0) {dot={\hdot}};
					\pic[blue] at (2*\d,0) {dot={\hdot}};
					\pic[blue] at (3*\d,0) {arch={3*\d}{1*\h}};
					\pic[red] at (4*\d,.2*\h) {dot={\hdott}}; \fill[red] (4*\d,.2*\h) circle (1.5pt);
					\pic[red] at (5*\d,0) {dot={\hdot}};
					\draw[negative](0.5*\d,0)--(6.5*\d,0);
				\end{tikzpicture}	
				\,-\,
				\begin{tikzpicture}[scale=\s,baseline=\b, transform shape]
					\def\hdott{.8*\hdot}
					\pic[red] at (\d,0) {dot={\hdot}};
					\pic[blue] at (2*\d,0) {dot={\hdot}};
					\pic[blue] at (3*\d,0) {arch={3*\d}{1*\h}};
					\pic[red] at (4*\d,0) {dot={\hdot}};
					\pic[red] at (5*\d,.2*\h) {dot={\hdott}}; \fill[red] (5*\d,.2*\h) circle (1.5pt);
					\draw[negative](0.5*\d,0)--(6.5*\d,0);
				\end{tikzpicture}	
				\,+\,
				\begin{tikzpicture}[baseline=\b,scale=\s, transform shape]
					\pic[red] at (\d,0) {dot={\hdot}};
					\pic[blue] at (2*\d,0) {dot={\hdot}};
					\draw[blue] (3*\d,0) ..controls (3*\d,\h) and (5.7*\d,\h).. (6*\d,.2*\h); \fill[blue] (6*\d,.2*\h) circle (1.5pt);
					\pic[red] at (4*\d,0) {dot={\hdot}};
					\pic[red] at (5*\d,0) {dot={\hdot}};
					\draw[negative](0.5*\d,0)--(6.5*\d,0);
				\end{tikzpicture}.	
			\end{multline*}		
			Here, for instance, the sign of the second term, obtained by
			uprooting the blue dot, is $-1$: in fact
			there are two colors missing on its right (a red on the right of the blue dot, and a 
			blue between the two red dots under the arch), so $r=2$ in this case.
		\end{exa}	
		It is convenient to consider the (unbounded) dg-module 
		\[
			\Gamma=\bigoplus_{\ux}\Hom_{\D}(B_{\ux},\un)/{\Lambda_{\ux}}\Tate{-\ell(\ux)}
		\]
		where $\ux$ runs through all Coxeter words, and the differential map is the same as before.
		Then notice that the dg tensor product of morphisms
		(with signs according to the Koszul rule) makes $\Gamma$ a dg-algebra.
		We see the space $C_n$ as a dg submodule of $\Gamma[-2n]$.
		\begin{rmk}
			Notice that if $k \le n$ all Coxeter words of length $k$ appear as
			subwords of $\us_{2n}$. This means that, if $\tilde{\tau}^{\ge i}$ is the naive truncation in $\Kb(\D)$ (or even in $\Cdg(\D)$)
			we have
			\[
				\tilde{\tau}^{\ge n}(C_n)=\tilde{\tau}^{\ge n}(\Gamma[-2n])
			\]
			In other words, for high cohomological degree the ``tails'' of $C_n$ and $\Gamma$ are the same (up to a shift).
			So the properties of $\Gamma$ will correspond to \textit{asymptotic} properties of $C_n$.
		\end{rmk}			
		A crucial property of $C_n$ and $\Gamma$ for our purpose is that
		they are free as left $R$-modules. Bases are obtained from the light leaves bases of the 
		spaces $\Hom(B_{\ux},\un)$. Let us now give a 
		precise description of these basis for $C_n$.
		\subsection{Shrubberies}
		First let us look more closely to the light leaves maps in type $\tilde{A_1}$.
		
		We are interested in morphism spaces of the form $\Hom(B_{\uw},\un)$. In the sequel the
		\emph{length} of an element $\gamma$ of such a
		space, denoted by $\ell(\gamma)$, will just mean the length of $\uw$ as a Coxeter word.
		
		We introduce an operation on morphisms of this form.	
		For a given positive integer $k$, consider the morphism
		\[
			\psi_{k}^{\re{s}}:=
			\underbrace{%
			\begin{tikzpicture}[baseline=.2cm,x=.3cm]
				\draw[gray] (-.5,0)--(7.5,0);
				\pic[red] at (0,0) {aqueduct={2}{5}{7}{1}};
				\node at (3.5,.3) {\dots};
				\foreach \i in {0,2,5,7}{%
					\node[anchor=north] at (\i,0) {$B_s$};
					}
				\node at (3.5,-.2) {\dots};
			\end{tikzpicture}
			}_{k+1}
			\in \Hom_{\D}(B_{\underline{ss\dots ss}},\un)
		\]	
		For $i=1,\dots, k$, let $\uw_i$ be a Coxeter word and
		$\gamma_i\in \Hom_{\D}(B_{\uw_i},\un)$. Then
		we define
		\[
			\ropen \gamma_1 \rcomma \dots \rcomma \gamma_k \rclosed :=
			\psi_{k}^s
			\circ
			(\id_{B_s}\otimes\gamma_1\otimes\id_{B_s}\otimes\dots\otimes\id_{B_s}\otimes\gamma_k\otimes\id_{B_s})
		\]
		This is a morphism from $B_{\uw'}$ to $\un$ where 
		\[
			\uw'= s\uw_1 s \dots s \uw_k s.
		\]
		In other words, this is the morphism obtained by covering the 
		original ones with an $s$-arch and separating them by vertical strands. 
		
		\begin{exa}
			Let
			\begin{align*}
				&\gamma_1=
				\begin{tikzpicture}[baseline=.2cm,x=.5cm,y=.5cm]
					\pic[blue] at (1,0) {arch={2}{2}};
					\pic[red] at (2,0) {dot={1}}; 
					\draw[negative] (0.5,0)--(3.5,0);
				\end{tikzpicture}
				\in \Hom(B_{\bl{t}\re{s}\bl{t}},\un) \\
				&\gamma_2=
				\begin{tikzpicture}[baseline=.2cm,x=.5cm,y=.4cm]
					\pic[red] at (1,0) {bridge={2}{4}{2.5}};
					\pic[blue] at (2,0) {dot={1}}; 
					\pic[blue] at (4,0) {dot={1}}; 
					\pic[red] at (6,0) {dot={1}}; 
					\draw[negative] (0.5,0)--(6.5,0);
				\end{tikzpicture}
				\in \Hom(B_{\re{s}\bl{t}\re{s}\bl{t}\re{ss}},\un) \\
				&\gamma_3=
				\begin{tikzpicture}[baseline=.2cm,x=.5cm,y=.5cm]
					\pic[blue] at (1,0) {dot={1}}; 
					\pic[red] at (2,0) {dot={1}}; 
					\draw[negative] (0.5,0)--(2.5,0);
				\end{tikzpicture}
				\in \Hom(B_{\bl{t}\re{s}},\un)
			\end{align*}
			Then 
			\[
				\ropen \gamma_1 \rcomma \gamma_2 \rcomma \gamma_3 \rclosed=
				\begin{tikzpicture}[baseline=.2cm,x=.35cm,y=.4cm]
					\pic[red] at (1,0) {aqueduct={4}{11}{14}{4}};
					\pic[blue] at (2,0) {arch={2}{2}};
					\pic[red] at (3,0) {dot={1}}; 
					\pic[red] at (6,0) {bridge={2}{4}{2.5}};
					\pic[blue] at (7,0) {dot={1}}; 
					\pic[blue] at (9,0) {dot={1}}; 
					\pic[red] at (11,0) {dot={1}}; 
					\pic[blue] at (13,0) {dot={1}}; 
					\pic[red] at (14,0) {dot={1}}; 
					\draw[negative] (0.5,0)--(15.5,0);
				\end{tikzpicture}
				\in \Hom(B_{\re{s}\bl{t}\re{s}\bl{t}\re{ss}\bl{t}\re{s}\bl{t}\re{sss}\bl{t}\re{ss}},\un)
			\]
		\end{exa}		
		In a similar way one defines
		$\psi_k^{\bl{t}}$ and $\bopen \cdot \bcomma \cdots \bcomma \cdot \bclosed$.
		
		Now, recall the diagrammatic construction of light leaves maps from \cite{EW}. Notice that, 
		as we only have one reduced expression for all elements of $W$ and there is no braid relation, 
		the light leaves basis does not depend on any choice. Furthermore we can easily write its 
		elements.
		\begin{exa}
			The following diagram is a light leaves map in $\tilde{A_1}$ towards the unit.
			\begin{equation}\label{eq_shrubberywithtriv}
				\begin{tikzpicture}[baseline=1cm]
					\def \d{.5cm}
					\def \h{1cm}
					\def \hdot{.4cm}
					\draw[gray](0,0)--(17*\d,0);
					\pic[red] at (\d,0) {arch={8*\d}{2*\h}};
					\pic[blue] at (2*\d,0) {bridge={2*\d}{6*\d}{1.5*\h}};
					\pic[red] at (3*\d,0) {dot={\hdot}};
					\pic[red] at (5*\d,0) {bridge={\d}{2*\d}{1*\h}};
					\pic[blue] at (10*\d,0) {aqueduct={\d}{2*\d}{4*\d}{1*\h}};
					\pic[red] at (13*\d,0) {dot={\hdot}};
					\pic[red] at (15*\d,0) {dot={\hdot}};
					\pic[blue] at (16*\d,0) {dot={\hdot}};
					\foreach \i in {0,1,4,9}{%
						\node[font=\footnotesize] at (\i*\d+\d,-.2*\h) {U1};
					}
					\foreach \i in {2,12,14,15}{%
						\node[font=\footnotesize] at (\i*\d+\d,-.2*\h) {U0};
					}
					\foreach \i in {6,7,8,13}{%
						\node[font=\footnotesize] at (\i*\d+\d,-.2*\h) {D1};
					}
					\foreach \i in {3,5,10,11}{%
						\node[font=\footnotesize] at (\i*\d+\d,-.2*\h) {D0};
					}
				\end{tikzpicture}			
			\end{equation}		
			On bottom we have indicated the decorated 01-sequence (expressing the identity element) 
			corresponding to it.
		\end{exa}
		\begin{rmk}
			Let $\uw=s_1\cdots s_k$ and $x=1$. If the Bruhat stroll defined by a subexpression $\boe$
			of $\uw$ passes through the identity at the step $i$, namely $\uw_{\le i}^{\boe_{\le i}}=1$, 
			then the corresponding light leaves map $L_{\uw,\boe}$ can be split as the tensor product
			$L_{\uw_{\le i},\boe_{\le i}}\otimes L_{\uw_{>i},\boe_{>i}}$. For example the above map can be 
			written as the tensor product of four pieces.
		\end{rmk}		
		Here is some some terminology to describe these morphisms.
		\begin{enumerate}
			\item The empty basis element, whose starting word is the empty word, is called \textit{trivial};
			\item A non-trivial basis element that cannot be written as the tensor product of smaller non-trivial basis elements
				is called \textit{shrub} (this corresponds to the fact that the Bruhat stroll determined by 
				the 01-sequence $\boe$ does not return to the identity element until the end). The trivial basis
				element will also be referred to as a \textit{trivial shrub} in the sequel;
			\item A (non-trivial) shrub is said to be \textit{red} (or \textit{blue}) if its first starting point is red (or blue, respectively).
				This means that its outer strand, the one adjacent to the topmost region of the strip, is red (or blue). We declare that 
				the trivial shrub is both red and blue;
			\item Any basis element is then written uniquely as the tensor product of shrubs,
				and we will refer to it as a \textit{shrubbery}. Clearly we have a separation between different shrubs whenever
				the Bruhat stroll gets to the identity element.
				For example, the shrubbery \eqref{eq_shrubberywithtriv}
				has four shrubs, two of them are red and the other two blue;
			\item A shrubbery is \textit{red} (or \textit{blue}) if all its shrubs are red (or blue respectively). 
		\end{enumerate}
		We can now give a recursive description of shrubberies:
		\begin{itemize}
			\item We denote by $\rdot$ and $\bdot$ the red and blue dots respectively;
			\item A blue shrub can be uniquely written in the form $\bopen L_1 \bcomma \dots \bcomma L_k \bclosed$ 
				for a sequence of (possibly trivial) red shrubberies $L_1,\dots, L_k$ (we see the blue dot 
				as the case $k=0$).
				Formally, if $(\uu_1,\boe_1),\dots,(\uu_k,\boe_k)$ are
				the starting words and subexpressions defining $L_1,\dots, L_k$, 
				this corresponds to take $(\uu,\boe)$ as follows
				\begin{gather*}
					\uu=t\, \uu_1\, t\, \uu_2\, \dots\, t\, \uu_k\,t\\
					\boe=1\, \boe_1\, 0\, \boe_2\, \dots\, 0\, \boe_k\, 1
				\end{gather*}
				(the decorations of the new symbols will be necessarily $U$ for the first and $D$ for all the others).
				Vice versa, one gets a red shrub starting by a sequence of blue shrubberies.
			\item Finally, a shrubbery is denoted by the concatenation of its shrubs.
				Hence, for example, the shrubbery \eqref{eq_shrubberywithtriv} would be
				\[
					\ropen \bopen \rdot \bcomma \ropen \rcomma \rclosed \bclosed \rclosed \bopen \bcomma \bcomma \rdot \bclosed \rdot \bdot
				\]
		\end{itemize}
		\begin{rmk}
			Notice that one can easily pass from the above notation to the pair $(\uw,\boe)$ and vice versa.
			The word $\uw$ corresponds to the sequence of colors and for $\boe$ we use
			the dictionary:
			\[
				(\,\leftrightarrow\, U1\qquad\qquad \bullet\, \leftrightarrow\, 
				U0 \qquad\qquad )\,\leftrightarrow\, D1 \qquad\qquad |\,\leftrightarrow\, D0
			\]
		\end{rmk}
		\begin{rmk}\label{rmk_breaking}
			Let $L$ be a shrubbery with starting word $\uw$, and consider the morphism 
			$\bopen L \bclosed$.
			If $L$ is red then this is a blue shrub.
			Otherwise, if $L$ contains at least one blue shrub, then it can be written as 
			\[
				K_1 \bopen L_1 \bcomma \dots \bcomma L_k \bclosed K_2 
			\]
			where $K_1$ and $K_2$ are (any) shrubberies and $L_1,\dots, L_k$ are red shrubberies
			(recall that $k=0$ correspond to a blue dot).
			Notice then that we can protract a blue dot from the blue shrub 
			and factor $\bopen L \bclosed$ as 
			\[
				\begin{tikzpicture}
					\draw[gray] (.5*\d,0)--(11.5*\d,0);
					\draw[blue] (\d,0)--(\d,2.2*\d);
					\node[anchor=south] at (2*\d,0) {$K_1$};\pic[dashed] at (1.25*\d,0) {arch={1.5*\d}{.8*\h}};
					\draw[blue] (3*\d,0) arc (180:90:2*\d);
					\draw[blue] (5*\d,0)--(5*\d,2*\d);
					\draw[blue] (5*\d,2*\d)--(7*\d,2*\d);
					\draw[blue] (6*\d,2*\d)--(6*\d,3*\d);\fill[blue] (6*\d,3*\d) circle (1.5pt);
					\draw[blue] (7*\d,0)--(7*\d,2*\d);
					\draw[blue] (9*\d,0) arc (0:90:2*\d);
					\node[anchor=south] at (4.1*\d,0) {$L_1$}; \pic[dashed,red] at (3.3*\d,0) {arch={1.5*\d}{.8*\h}};
					\node[anchor=south] at (6*\d,.5*\d) {$\dots$};
					\node[anchor=south] at (7.9*\d,0) {$L_k$}; \pic[dashed,red] at (7.2*\d,0) {arch={1.5*\d}{.8*\h}};
					\node[anchor=south] at (10*\d,0) {$K_2$}; \pic[dashed] at (9.25*\d,0) {arch={1.5*\d}{.8*\h}};
					\draw[blue] (11*\d,0)--(11*\d,2.2*\d);
					\draw[dashed,gray] (.5*\d,2.2*\d)--(11.5*\d,2.2*\d);
					\pic[blue] at (\d,2.2*\d) {arch={10*\d}{\h}};
				\end{tikzpicture}
			\]		
			Then notice that
			\[
				\begin{split}
					\begin{tikzpicture}[baseline=.2*\h,scale=.9,transform shape]
						\def\dd{.5cm}
						\draw(-.5*\d,0)--({2*\dd+.5*\d},0);
						\pic[blue] at (0,0) {arch={2*\dd}{\h}};
						\pic[blue] at (\dd,0) {dot={\hdot}};
					\end{tikzpicture}
					& =
						\begin{tikzpicture}[baseline=.2*\h,scale=.9,transform shape]
							\def\dd{.5cm}\def\hhdot{0.2cm}
							\draw(-.5*\d,0)--({2*\dd+.5*\d},0);
							\pic[blue] at (0,0) {arch={2*\dd}{\h}};
							\begin{scope}
								\clip (0,0) ..controls (0,\h) and (2*\dd,\h).. (2*\dd,0);
								\draw[blue] (\dd,\h)--(\dd,0.45cm);
								\fill[blue] (\dd,0.45cm) circle (1.5pt);
							\end{scope}
							\pic[blue] at (\dd,0) {dot={\hhdot}};
						\end{tikzpicture}
						=
						\begin{tikzpicture}[baseline=.2*\h,scale=.9,transform shape]
							\def\dd{1cm}\def\hhdot{0.2cm}
							\draw(-.5*\d,0)--({2*\dd+.5*\d},0);
							\pic[blue] at (0,0) {arch={2*\dd}{\h}};
							\begin{scope}
								\clip (0,0) ..controls (0,\h) and (2*\dd,\h).. (2*\dd,0);
								\draw[blue] (\dd,\h)--(\dd,0);
							\end{scope}
							\node 
								at (.6*\dd,0.32cm) {$\delta_t$};
						\end{tikzpicture}
						-
						\begin{tikzpicture}[baseline=.2*\h,scale=.9,transform shape]
							\def\dd{1cm}\def\hhdot{0.2cm}
							\draw(-.5*\d,0)--({2*\dd+.5*\d},0);
							\pic[blue] at (0,0) {arch={2*\dd}{\h}};
							\begin{scope}
								\clip (0,0) ..controls (0,\h) and (2*\dd,\h).. (2*\dd,0);
								\draw[blue] (\dd,\h)--(\dd,0);
							\end{scope}
							\node 
								at (1.4*\dd,0.32cm) {$t(\delta_t)$};
						\end{tikzpicture}
					\\
					& =
						\begin{tikzpicture}[baseline=.2*\h,scale=.9,transform shape]
							\def\dd{.7cm}\def\hhdot{0.2cm}
							\draw(-.5*\d,0)--({2*\dd+.5*\d},0);
							\pic[blue] at (0,0) {arch={2*\dd}{\h}};
							\begin{scope}
								\clip (0,0) ..controls (0,\h) and (2*\dd,\h).. (2*\dd,0);
								\draw[blue] (\dd,\h)--(\dd,0);
							\end{scope}
							\node 
								at (-.2cm,0.7cm) {$t(\delta_t)$};
						\end{tikzpicture}
						+
						\begin{tikzpicture}[baseline=.2*\h,scale=.9,transform shape]
							\def\dd{.7cm}\def\hhdot{0.2cm}
							\draw(-.5*\d,0)--({2*\dd+.5*\d},0);
							\draw[blue] (0,0) arc (180:150:\dd); \fill[blue,xshift=\dd] (150:\dd) circle (1.5pt);
							\draw[blue,xshift=\dd] (120:\dd) arc (120:0:\dd); \fill[blue,xshift=\dd] (120:\dd) circle (1.5pt);
							\draw[blue] (\dd,0)--(\dd,\dd);
							\node 
								at (-.2cm,0.7cm) {$\partial_t(\delta_t)$};
						\end{tikzpicture}
						-
						\begin{tikzpicture}[baseline=.2*\h,scale=.9,transform shape]
							\def\dd{.7cm}\def\hhdot{0.2cm}
							\draw(-.5*\d,0)--({2*\dd+.5*\d},0);
							\pic[blue] at (0,0) {arch={2*\dd}{\h}};
							\begin{scope}
								\clip (0,0) ..controls (0,\h) and (2*\dd,\h).. (2*\dd,0);
								\draw[blue] (\dd,\h)--(\dd,0);
							\end{scope}
							\node 
								at (2*\dd+.1cm,0.7cm) {$\delta_t$};
						\end{tikzpicture}
						-
						\begin{tikzpicture}[baseline=.2*\h,scale=.9,transform shape]
							\def\dd{.7cm}\def\hhdot{0.2cm}
							\draw(-.5*\d,0)--({2*\dd+.5*\d},0);
							\draw[blue] (2*\dd,0) arc (0:30:\dd); \fill[blue,xshift=\dd] (30:\dd) circle (1.5pt);
							\draw[blue,xshift=\dd] (60:\dd) arc (60:180:\dd); \fill[blue,xshift=\dd] (60:\dd) circle (1.5pt);
							\draw[blue] (\dd,0)--(\dd,\dd);
							\node 
								at (2*\dd+.4cm,0.7cm) {$\partial_t\big(t(\delta_t)\big)$};
						\end{tikzpicture}
					\\
					& =
						-\alpha_t
						\begin{tikzpicture}[baseline=.2*\h,scale=.9,transform shape]
							\def\dd{.7cm}\def\hhdot{0.2cm}
							\draw(-.5*\d,0)--({2*\dd+.5*\d},0);
							\pic[blue] at (0,0) {arch={2*\dd}{\h}};
							\begin{scope}
								\clip (0,0) ..controls (0,\h) and (2*\dd,\h).. (2*\dd,0);
								\draw[blue] (\dd,\h)--(\dd,0);
							\end{scope}
						\end{tikzpicture}
						+
						\begin{tikzpicture}[baseline=.2*\h,scale=.9,transform shape]
							\draw (.5*\d,0)--(3.5*\d,0);
							\pic[blue] at (\d,0) {dot={\hdot}};
							\pic[blue] at (2*\d,0) {arch={\d}{.7*\h}};
						\end{tikzpicture}
						+
						\begin{tikzpicture}[baseline=.2*\h,scale=.9,transform shape]
							\draw (.5*\d,0)--(3.5*\d,0);
							\pic[blue] at (\d,0) {arch={\d}{.7*\h}};
							\pic[blue] at (3*\d,0) {dot={\hdot}};
						\end{tikzpicture}
				\end{split}
			\]
			(in fact $t(\delta_t)=\delta_t-\langle \delta_t,\alpha_t^\vee\rangle\alpha_t=\delta_t-\alpha_t$).
			
			Hence
			\[
				\bopen L \bclosed = -\alpha_t \bopen K_1 \bcomma L_1 \bcomma \dots \bcomma L_k \bcomma K_2 \bclosed +\bdot\, K_1 \bopen L_1 \bcomma \dots \bcomma L_k \bcomma K_2 \bclosed
				+\bopen K_1 \bcomma L_1 \bcomma \dots \bcomma L_k \bclosed K_2 \, \bdot
			\]
			and, iterating the process if necessary, one gets the linear combination of shrubberies expressing
			$\bopen L \bclosed$. 
			
			One can then see that for any $\gamma_1,\dots,\gamma_k$, the morphism 
			$\bopen \gamma_1 \bcomma \dots \bcomma \gamma_k \bclosed$ can be written as a linear combination
			of shrubberies which are either products of smaller shrubs, or shrubs in which the sequence 
			of the sizes of the arches is a refinement of 
			the starting one. This means shrubs of the form
			\[
				\bopen L^{(1)}_1 \bcomma \dots \bcomma L^{(1)}_{r_1}\bcomma L^{(2)}_1 \bcomma \dots \bcomma L^{(2)}_{r_2} \bcomma\dots \bcomma \dots \bcomma L^{(k)}_1 \bcomma \dots \bcomma L^{(k)}_{r_k}\bclosed
			\]
			where $\ell(L^{(j)}_1)+\dots +\ell(L^{(j)}_{r_j})=\ell(\gamma_j)$.
		\end{rmk}
		
		Finally, to find bases for $C_n$ we have to quotient by the morphisms in $\Lambda_{\ux}$.		
		More precisely,
		the spaces $\Hom_{\D}(B_{\ux},\un)/\Lambda_{\ux}$ are still free with bases 
		given by light leaves maps that do not belong to
		$\Lambda_{\ux}$. Then the collection of all such light leaves maps when $\ux$ runs through 
		the subwords of $\us_{2n}$ give a basis for $C_n$.
		Similarly, one gets a basis for $\Gamma$ by considering all Coxeter words.		
	
		The shrubberies that are not in the annihilator are precisely those that do not have \textit{empty arches}, namely
		subsequences of one of the forms 
		\begin{align*}
			&(|,&	&(),&	&||,&	&|).		
		\end{align*}
		In fact, these correspond precisely to the forbidden sequences.
		\begin{exa}	
			Here is a shrubbery with this property,
			
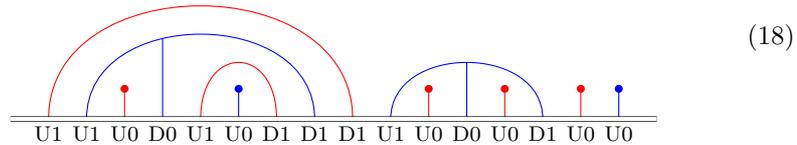
\begin{equation}\label{eq_shrubbery2}
				\begin{tikzpicture}[baseline=1cm]
					\def \d{.5cm}
					\def \h{1cm}
					\def \hdot{.4cm}
					\pic[red] at (\d,0) {arch={8*\d}{2*\h}};
					\pic[blue] at (2*\d,0) {bridge={2*\d}{6*\d}{1.5*\h}};
					\pic[red] at (3*\d,0) {dot={\hdot}};
					\pic[blue] at (6*\d,0) {dot={\hdot}};
					\pic[red] at (5*\d,0) {arch={2*\d}{1*\h}};
					\pic[blue] at (10*\d,0) {bridge={2*\d}{4*\d}{1*\h}};
					\pic[red] at (11*\d,0) {dot={\hdot}};
					\pic[red] at (13*\d,0) {dot={\hdot}};
					\pic[red] at (15*\d,0) {dot={\hdot}};
					\pic[blue] at (16*\d,0) {dot={\hdot}};
					\foreach \i in {0,1,4,9}{%
						\node[font=\footnotesize] at (\i*\d+\d,-.2*\h) {U1};
					}
					\foreach \i in {2,5,10,12,14,15}{%
						\node[font=\footnotesize] at (\i*\d+\d,-.2*\h) {U0};
					}
					\foreach \i in {3,11}{%
						\node[font=\footnotesize] at (\i*\d+\d,-.2*\h) {D0};
					}
					\foreach \i in {6,7,8,13}{%
						\node[font=\footnotesize] at (\i*\d+\d,-.2*\h) {D1};
					}
					\draw[negative](0,0)--(17*\d,0);
				\end{tikzpicture}			
			\end{equation}
			which corresponds to the expression
			\[
				\ropen \bopen \rdot \bcomma \ropen \bdot \rclosed \bclosed \rclosed \bopen \rdot \bcomma \rdot \bclosed \rdot \bdot
			\]
		\end{exa}
		Let $\mathcal{L}$ denote the set of all such shrubberies, so that $\mathcal{L}$ is a basis for $\Gamma$.
		Let also $\Rshr$ and $\Bshr$ denote the subsets of red and blue shrubberies respectively.
		\begin{rmk}
			As every shrubbery is the tensor product of shrubs, we see that
			$\Gamma$ is generated by shrubs, as a dg-$R$-algebra.
		\end{rmk}
		\subsection{Gardening principles}
		We will now investigate the properties of the differential map with respect to this basis.

		We introduce a little more terminology.
		Again, if $L$ is a shrubbery its \textit{length} $\ell(L)$ is that of its starting word.
		We say that a shrub is \textit{complete} when its starting word is of 
			the form $\ut_{2h+1}$ or $\us_{2h+1}$ (i.e.\ the colors of its starting points are alternating).
			A shrubbery is said \textit{complete} when all of its shrubs are (notice that the whole
			starting word need not be alternating).
		We call \textit{stem} a vertical strand connecting the boundary with a trivalent vertex (so stems
		correspond precisely to $D0$'s in $\boe$.
		A shrubbery is \textit{well-tended} if it is complete and without stems. 
		Hence notice that well-tended \textit{shrubs} have to be
		of the form
		\[
			\bopen \ropen \dots \ropen \bopen \rdot \bclosed \rclosed \dots \rclosed \bclosed =
			\begin{tikzpicture}[scale=.7, transform shape,baseline= .5cm]
				\draw (-6*\d,0)--(6*\d,0);
				\pic[red] at (0,0) {dot={\hdot}};
				\pic[blue] at (-\d,0) {arch={2*\d}{\h}};
				\pic[red] at (-2*\d,0) {arch={4*\d}{1.5*\h}};
				\node at (-3*\d,\hdot) {\dots};
				\node at (3*\d,\hdot) {\dots};
				\pic[red] at (-4*\d,0) {arch={8*\d}{2.5*\h}};
				\pic[blue] at (-5*\d,0) {arch={10*\d}{3*\h}};
			\end{tikzpicture}
		\]
		In particular, dots are well-tended shrubs. Also the empty diagram is declared 
		to be well-tended.
		Notice that a well-tended shrub is determined by the number of its strands and
		by its outer color: we will denote $\rho_k$ and $\beta_k$ respectively 
		the red and blue well-tended shrubs with $k$ strands (which has length $2k-1$).
	
		We want to define a partial order $\le$ on shrubberies of a given length $\ell$. 
		We proceed by induction on $\ell$.
		\begin{dfn}\label{def_order}
			\begin{enumerate}
				\item If $\ell=0,1$ (i.e.\ when the only shrubberies are, respectively, the empty one, or 
					the red dot and the blue dot)
					we declare that $L\le L'$ if and only if $L=L'$;
				\item If $\ell>1$, suppose that we have defined a partial order on all sets of
					shrubberies	of length smaller than $\ell$. Take first two \textit{shrubs} $L$ and $L'$ 
					of length $\ell$.
					We declare that $L\le L'$ if they are of the same color, say blue, and, when written as
					\begin{gather*}
						L=\bopen L_1\bcomma \dots \bcomma L_k \bclosed \\
						L'=\bopen L_1'\bcomma \dots \bcomma L'_{k'}\bclosed
					\end{gather*}
					we have either:
					\begin{itemize}
						\item $\ell(L_1)<\ell(L_1')$ (we also call these the \textit{sizes of the first arches}), or;
						\item $\ell(L_1)=\ell(L_1')$ and $L_1< L_1'$ in the order of shrubberies of length $\ell(L_1)$, or;
						\item $L_1= L_1'$, and $\bopen L_2\bcomma \dots\bcomma L_k\bclosed \le \bopen L_2\bcomma \dots\bcomma L'_{k'}\bclosed$ in the order of 
							shrubs of length $\ell-\ell(L_1)$ (notice that if $k=1$, we already have $L=L'$).
					\end{itemize}
				\item Finally take two shrubberies $L$ and $L'$, which are not just shrubs, and
					decompose them into shrubs
					\begin{gather*}
						L=K_1\otimes\dots\otimes K_m \\
						L'=K_1'\otimes\dots\otimes K'_{m'}.
					\end{gather*}
					Then we declare that $L\le L'$ if, either:
					\begin{itemize}
					\item $\ell(K_1)<\ell(K_1')$ or;
					\item $\ell(K_1)=\ell(K_1')$ and $K_1< K_1'$ in the order on shrubs of length $\ell(K_1)$, or;
					\item $K_1= K_1'$, and $K_2\otimes\dots\otimes K_m\le K_2'\otimes\dots\otimes K_{m'}'$
						in the order of shrubberies of length $\ell-\ell(K_1)$.
					\end{itemize}					
				\end{enumerate}
		\end{dfn}	
		\begin{rmk}\label{rmk_ordermon}
			By definition, the order is lexicographic with respect to the monoidal structure, which means that, given
			the shrubberies $L$, $L'$, $M$, $M'$ and $N$, we have	
			\[
				L < L'\,\Rightarrow\, N\otimes L\otimes M < N\otimes L'\otimes M'
			\]
		\end{rmk}
		We now want to prove that the differential map has an upper-triangularity property with respect to this order.
		\begin{rmk}
			Uprooting a stem from a shrubbery always gives another shrubbery.
			Given $L$ a shrubbery with at least one stem, let $u(L)$ be the shrubbery obtained by uprooting the 
			leftmost outer stem. More precisely, if $L$ is a single, say, red shrub $\ropen L_1\rcomma L_2\rcomma \dots\rcomma L_k \rclosed$ 
			with $k>1$, then
			$u(L):=\ropen L_1L_2\rcomma \dots\rcomma L_k \rclosed$ (and same for blue); if $k=1$ then set $u(L)=\ropen u(L_1)\rclosed$ 
			recursively. If $L$ is a product of shrubs
			$K_1\otimes \dots \otimes K_m$ and $K_i$ has a stem whereas $K_1,\dots,K_{i-1}$ have not, then
			$u(L):=K_1\otimes \dots \otimes u(K_i)\otimes \dots \otimes K_m$.
		\end{rmk}
		\begin{lem}\label{lem_order}
			Let $L$ be a shrubbery with at least one stem. Then 
			\[
				d(L)\in \pm u(L) +\bigoplus_{L'<u(L)}RL'
			\]
		\end{lem}
		\begin{proof}
			By Remark \ref{rmk_ordermon}, it is sufficient to deal with the case where $L$ is a single, say, red shrub.
			We can write $L$ in the form
			\[
				L=\ropen L_1\rcomma \dots \rcomma L_k \rclosed
			\]
			where the $L_j$'s are blue shrubberies. 
			
			We proceed by induction.
			The smallest such $L$ is the following of length $5$
			\[
				\ropen \bdot \rcomma \bdot \rclosed =
				\begin{tikzpicture}[baseline=.3cm]
					\pic[red] at (0,0) {bridge={2*\d}{4*\d}{\h}};
					\foreach \i in {1,3}{%
						\pic[blue] at (\i*\d,0) {dot={\hdot}};
						}
					\draw[negative] (-.5*\d,0)--(4.5*\d,0);
				\end{tikzpicture}
			\]
			and one can easily work out that 
			\begin{equation}\label{eq_size5triv}
				d\big( \ropen \bdot \rcomma \bdot \rclosed
					\big)=
				\bdot\ropen\bdot\rclosed-
				\qn{2}{t}\rdot\ropen\bdot\rclosed
				+\ropen\bdot\bdot\rclosed
				-\qn{2}{s}\ropen\bdot\rclosed\rdot
				+\ropen\bdot\rclosed\bdot
			\end{equation}
			and $\ropen\bdot\bdot\rclosed=u(L)$ is bigger than all the other shrubberies appearing.
	
			For a general shrub $L$, the image $d(L)$ is the following linear combination of diagrams 
			obtained by uprooting strands
			\begin{multline*}		
				d(L)= \sigma_0 L_1\ropen L_2\rcomma \dots \rcomma L_k\rclosed + \sum_j \sigma_j \ropen L_1\rcomma \dots\rcomma L_jL_{j+1}\rcomma \dots\rcomma L_k\rclosed +\\
					-\ropen L_1 \rcomma \dots \rcomma L_{k-1} \rclosed L_k+
					\sum_j \sigma'_j \ropen L_1\rcomma \dots\rcomma d(L_j)\rcomma \dots\rcomma L_k\rclosed 
			\end{multline*}		
			for some signs $\sigma_j$ and $\sigma_j'$. The first and third terms, obtained by uprooting the first and last stems, are product of 
			smaller shrubs, so they are clearly lower than $u(L)$.
			
	 		Now, if $k>1$, then $u(L)$ is the term with $j=1$ in the first sum and its first arch has size
	 		$\ell(L_1)+\ell(L_2)$, hence it
	 		is bigger than that of all the other terms of that sum.
	 		
	 		For the second sum one can use the remark \ref{rmk_breaking} to see that, each term can be developed
	 		as a linear combination of shrubberies that are either
	 		products of smaller shrubberies or that have first arch with size $\le \ell(L_1)$.
	
	 		If $k=1$ then the shrubbery $L_1$ has to have at least one stem. The above sum reduces to
			\[
				d(L)= \sigma_0 L_1 \rdot - \rdot L_1 + \sigma ' \ropen d(L_1)\rclosed 
			\]
			and the first two terms are again clearly smaller than $u(L)$. 
			Then we can apply the induction hypothesis on $L_1$ and we conclude.
		\end{proof}
		\subsection{Weeding the garden}\label{subs_weed}
		In this section we will describe a simpler dg-module $\Cn$ homotopy equivalent to $C_n$, 
		which is much more handy and will allow us to compute cohomology in the antispherical category.
		In fact the direct computation on $C_n$ would still be 
		prohibitive, so we use Gaussian elimination to reduce it.
		
		Let $\tilde{\Gamma}$ be the free associative dg-$R$-algebra generated by elements $\tilde{\rho}_k$ and 
		$\tilde{\beta}_k$ for $k\in \ZZ_{>0}$. Both $\tilde{\rho}_0$  and $\tilde{\beta}_0$ will denote the unit and
		$\tilde{\rho}_k$ and $\tilde{\beta}_k$ are in degree $-2k+1$. The differential map is defined as follows
		\begin{align*}
			d(\tilde{\rho}_k)=	\begin{cases}
							\alpha_s \tilde{\rho}_0 & \text{if $k=1$}\\
							-\sum_{i=1}^{k-1}\qbc{k}{i}{s}\tilde{\rho}_{k-i}\tilde{\rho}_{i}
							+\sum_{i=1}^{k-1}\qbc{k-1}{i-1}{t}(\tilde{\rho}_i\tilde{\beta}_{k-i}+\tilde{\beta}_{k-1}\tilde{\rho}_i)
							& \text{if $k>1$}
						\end{cases}\\
			d(\tilde{\beta}_k)=	\begin{cases}
							-\alpha_t \tilde{\beta}_0 & \text{if $k=1$}\\
							\sum_{i=1}^{k-1}\qbc{k}{i}{t}\tilde{\beta}_{k-i}\tilde{\beta}_{i}
							-\sum_{i=1}^{k-1}\qbc{k-1}{i-1}{s}(\tilde{\beta}_i\tilde{\rho}_{k-i}+\tilde{\rho}_{k-1}\tilde{\beta}_i)
							& \text{if $k>1$}
						\end{cases}		
		\end{align*}
		and then extended $R$-linearly and according to \emph{right} Leibniz rule, which means, for monomials $x$ and $y$:
		\[
			d(xy)=xd(y)+(-1)^{|y|}d(x)y,
		\]
		where $|y|$ is the number of generators appearing in $y$.
		
		We define the \emph{word associated} to a monomial in $\tilde{\Gamma}$ as 
		\begin{align*}
			&\tilde{\rho}_k\mapsto \us_{2k-1}& &\tilde{\beta}_k \mapsto \ut_{2k-1}
		\end{align*}
		and then extended multiplicatively to a map of monoids 
		$\langle \tilde{\rho}_k,\tilde{\beta}_k \rangle_{k\in \mathbb{N}} \rightarrow \cword$.
		Then the span of the monomials whose associated word is a subword of $\us_{2n}$ 
		is a dg submodule of $\tilde{\Gamma}$, that will be denoted by $\tilde{C}_{n}$. 
		
		We can now state the main result of this section
		\begin{thm}\label{thm_antisph}
			The dg-module $C_n$ is homotopy equivalent to $\tilde{C}_{n}[-2n]$
		\end{thm}
		\begin{rmk}\label{rmk_divis}
			Consider a complete shrubbery $L$, which is not well-tended. 
			There will be a nonempty set of stems. The shrubberies obtained by uprooting some of 
			these stems are still not well-tended, and we can divide them into two classes.
			The class $\mathcal{E}_L$ will contain those in which the leftmost outer stem has not been 
			uprooted, and $\mathcal{F}_L$ the others.
			Then the function $u$ gives a bijection from $\mathcal{E}_L$ to $\mathcal{F}_L$.
		\end{rmk}
		The following is an example where $L$ has 3 stems.
		\begin{exa}
			Take
			\[
				L=
				\begin{tikzpicture}[baseline=.3cm,x=.4cm]
					\pic[blue] at (0,0) {bridge={2}{8}{1.5}};
					\pic[red] at (3,0) {bridge={2}{4}{1}};
					\pic[red] at (9,0) {bridge={2}{4}{1}};
					\foreach \i in {1}{%
						\pic[red] at (\i,0) {dot={\hdot}};%
						}
					\foreach \i in {4,6,10,12}{%
						\pic[blue] at (\i,0) {dot={\hdot}};%
						}
					\draw[negative] (-.5,0)--(13.5,0);
				\end{tikzpicture}
			\]
			then the two classes and the bijection are described by the following picture
			
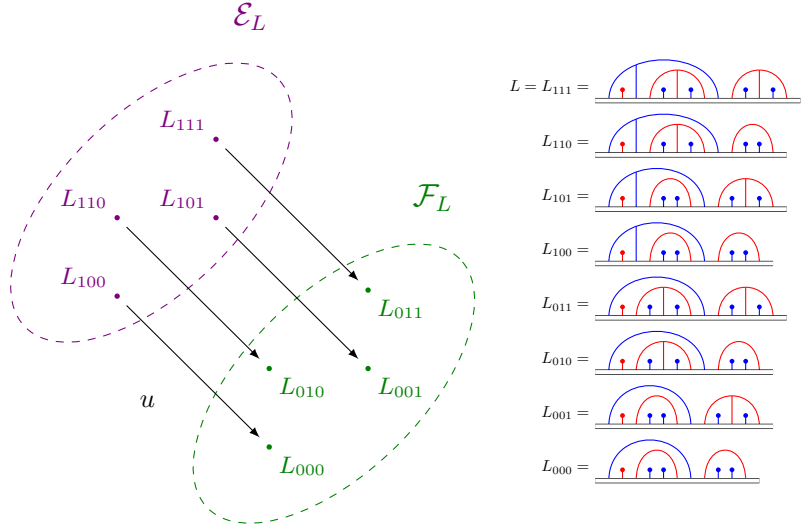
\begin{center}
				\begin{tikzpicture}
					\def\u{1.3cm}\def\v{2cm}\def\w{.8*\u}
					\draw[dashed,violet,xshift=.5*\u-1.1*\v,yshift=\w+1.1*\v,rotate=45] (0,0) ellipse ({1.8*\u} and {.9*\u});				
					\draw[dashed,green,xshift=.5*\u+.1*\v,yshift=\w-.1*\v,rotate=45] (0,0) ellipse ({1.8*\u} and {.9*\u});
					\foreach \j in {0,1}{%
						\foreach \k in {0,1}{%
							\node (a) at ({-\v+\k*\u},{\v+\j*\w+\k*\w}) {};
							\node[violet,anchor=south east,font=\small] at (a) {$L_{1\j\k}$};
							\node (b) at ({\k*\u},{\j*\w+\k*\w}) {};
							\node[green,anchor=north west,font=\small] at (b) {$L_{0\j\k}$};
							\fill[violet] (a) circle (1pt);\fill[green] (b) circle (1pt);
							\draw[-latex] (a)--(b);
						}
					}
					\node[green,xshift=.1*\v,yshift=-.1*\v,font=\large] at (1.5*\u,2.7*\u) {$\mathcal{F}_L$};
					\node[violet,xshift=-1.1*\v,yshift=1.1*\v,font=\large] at (1.5*\u,2.7*\u) {$\mathcal{E}_L$};
					\node[xshift=-.2*\v,yshift=-.2*\v] at (-.6*\v,.5*\v) {$u$};
				\end{tikzpicture}
				\begin{tikzpicture}[scale=.6,transform shape,baseline=-5.5cm]
					\def\d{.3cm}\def\h{.6cm}\def\hdot{.24}\def\y{1.2cm}\def\hw{.5*\h}
					\begin{scope}
						\node[anchor=east] at (-\d,\hw) {$L=L_{111}=$};
						\pic[blue] at (0,0) {bridge={2*\d}{8*\d}{2*\h}};
						\pic[red] at (3*\d,0) {bridge={2*\d}{4*\d}{1.5*\h}};
						\pic[red] at (9*\d,0) {bridge={2*\d}{4*\d}{1.5*\h}};
						\foreach \i in {1}{%
							\pic[red] at (\i*\d,0) {dot={\hdot}};%
							}
						\foreach \i in {4,6,10,12}{%
							\pic[blue] at (\i*\d,0) {dot={\hdot}};%
							}
						\draw[negative] (-\d,0)--(14*\d,0);
					\end{scope}
					\begin{scope}[yshift=-\y]
						\node[anchor=east] at (-\d,\hw) {$L_{110}=$};
						\pic[blue] at (0,0) {bridge={2*\d}{8*\d}{2*\h}};
						\pic[red] at (3*\d,0) {bridge={2*\d}{4*\d}{1.5*\h}};
						\pic[red] at (9*\d,0) {arch={3*\d}{1.5*\h}};
						\foreach \i in {1}{%
							\pic[red] at (\i*\d,0) {dot={\hdot}};%
							}
						\foreach \i in {4,6,10,11}{%
							\pic[blue] at (\i*\d,0) {dot={\hdot}};%
							}
						\draw[negative] (-\d,0)--(13*\d,0);
					\end{scope}
					\begin{scope}[yshift=-2*\y]
						\node[anchor=east] at (-\d,\hw) {$L_{101}=$};
						\pic[blue] at (0,0) {bridge={2*\d}{7*\d}{2*\h}};
						\pic[red] at (3*\d,0) {arch={3*\d}{1.5*\h}};
						\pic[red] at (8*\d,0) {bridge={2*\d}{4*\d}{1.5*\h}};
						\foreach \i in {1}{%
							\pic[red] at (\i*\d,0) {dot={\hdot}};%
							}
						\foreach \i in {4,5,9,11}{%
							\pic[blue] at (\i*\d,0) {dot={\hdot}};%
							}
						\draw[negative] (-\d,0)--(13*\d,0);
					\end{scope}
					\begin{scope}[yshift=-3*\y]
						\node[anchor=east] at (-\d,\hw) {$L_{100}=$};
						\pic[blue] at (0,0) {bridge={2*\d}{7*\d}{2*\h}};
						\pic[red] at (3*\d,0) {arch={3*\d}{1.5*\h}};
						\pic[red] at (8*\d,0) {arch={3*\d}{1.5*\h}};
						\foreach \i in {1}{%
							\pic[red] at (\i*\d,0) {dot={\hdot}};%
							}
						\foreach \i in {4,5,9,10}{%
							\pic[blue] at (\i*\d,0) {dot={\hdot}};%
							}
						\draw[negative] (-\d,0)--(12*\d,0);
					\end{scope}
					\begin{scope}[yshift=-4*\y]
						\node[anchor=east] at (-\d,\hw) {$L_{011}=$};
						\pic[blue] at (0,0) {arch={7*\d}{2*\h}};
						\pic[red] at (2*\d,0) {bridge={2*\d}{4*\d}{1.5*\h}};
						\pic[red] at (8*\d,0) {bridge={2*\d}{4*\d}{1.5*\h}};
						\foreach \i in {1}{%
							\pic[red] at (\i*\d,0) {dot={\hdot}};%
							}
						\foreach \i in {3,5,9,11}{%
							\pic[blue] at (\i*\d,0) {dot={\hdot}};%
							}
						\draw[negative] (-\d,0)--(13*\d,0);
					\end{scope}
					\begin{scope}[yshift=-5*\y]
						\node[anchor=east] at (-\d,\hw) {$L_{010}=$};
						\pic[blue] at (0,0) {arch={7*\d}{2*\h}};
						\pic[red] at (2*\d,0) {bridge={2*\d}{4*\d}{1.5*\h}};
						\pic[red] at (8*\d,0) {arch={3*\d}{1.5*\h}};
						\foreach \i in {1}{%
							\pic[red] at (\i*\d,0) {dot={\hdot}};%
							}
						\foreach \i in {3,5,9,10}{%
							\pic[blue] at (\i*\d,0) {dot={\hdot}};%
							}
						\draw[negative] (-\d,0)--(12*\d,0);
					\end{scope}
					\begin{scope}[yshift=-6*\y]
						\node[anchor=east] at (-\d,\hw) {$L_{001}=$};
						\pic[blue] at (0,0) {arch={6*\d}{2*\h}};
						\pic[red] at (2*\d,0) {arch={3*\d}{1.5*\h}};
						\pic[red] at (7*\d,0) {bridge={2*\d}{4*\d}{1.5*\h}};
						\foreach \i in {1}{%
							\pic[red] at (\i*\d,0) {dot={\hdot}};%
							}
						\foreach \i in {3,4,8,10}{%
							\pic[blue] at (\i*\d,0) {dot={\hdot}};%
							}
						\draw[negative] (-\d,0)--(12*\d,0);
					\end{scope}
					\begin{scope}[yshift=-7*\y]
						\node[anchor=east] at (-\d,\hw) {$L_{000}=$};
						\pic[blue] at (0,0) {arch={6*\d}{2*\h}};
						\pic[red] at (2*\d,0) {arch={3*\d}{1.5*\h}};
						\pic[red] at (7*\d,0) {arch={3*\d}{1.5*\h}};
						\foreach \i in {1}{%
							\pic[red] at (\i*\d,0) {dot={\hdot}};%
							}
						\foreach \i in {3,4,8,9}{%
							\pic[blue] at (\i*\d,0) {dot={\hdot}};%
							}
						\draw[negative] (-\d,0)--(11*\d,0);
					\end{scope}
					\def\d{.5cm}\def\h{1cm}\def\hdot{.4cm}
				\end{tikzpicture}
			\end{center}	
			Here the sequences of 0's and 1's that label the shrubberies indicate which stems have been uprooted.
			\end{exa}
			Let us then prove the theorem.
			\begin{proof}
				If one makes the subdivision of Remark \ref{rmk_divis} for all complete non well-tended shrubberies $L$,
				one gets two classes of shrubberies,
				$\mathcal{E}$ and $\mathcal{F}$, which span graded submodules
				\begin{align*}
					&E=\bigoplus_{L\in \mathcal{E}}R\cdot L&
					&\text{and}&
					&F=\bigoplus_{L\in \mathcal{F}}R\cdot L.				
				\end{align*}
				Let now $A=\bigoplus_p A^p$ be the span of the well-tended shrubberies. The dg-module $C_n$ 
				will then look as follows
				\[
					\begin{tikzcd}[row sep=tiny,column sep=huge]
						\dots \ar[r]\ar[rdd,lightgray]						& A^{p-1} \ar[r,"a^{p-1}"]\ar[rdd,lightgray]								& A^{p} \ar[r,"a^p"]\ar[rdd,lightgray]															& A^{p+1} \ar[r]\ar[rdd,lightgray]																			& \dots 		 \\
																		& \oplus															& \oplus																					& \oplus																					& 			   \\
						\dots \ar[r,lightgray]\ar[ruu]						& E^{p-1} \ar[r,lightgray]\ar[ruu,near end,"g^{p-1}"]						& E^p \ar[r,lightgray]\ar[ruu,near end,"g^p"]												& E^{p+1} \ar[ruu]\ar[r,lightgray]																& \dots				\\
																		& \oplus															& \oplus																					& \oplus																					& 				\\
						\dots \ar[r,lightgray]\ar[ruu,lightgray]					& F^{p-1} \ar[r,lightgray]\ar[ruu,lightgray]\ar[from=luu,green]\ar[from=luuuu]	& F^p \ar[r,lightgray]\ar[ruu,lightgray]\ar[from=luu,green,swap,"\phi^{p-1}"]\ar[from=luuuu,"f^{p-1}"]	& F^{p+1} \ar[r,lightgray]\ar[ruu,lightgray]\ar[from=luu,green,swap,"\phi^{p}"]\ar[from=luuuu,"f^p"]	& \dots \ar[from=luu,green]\ar[from=luuuu]		\\
					\end{tikzcd}
				\]
				We claim that all the morphisms $\phi^p:E^p\rightarrow F^{p+1}$ (green in the picture) 
				are isomorphisms.
				In fact, by the remark \ref{rmk_divis}, the morphism $\phi^p$ is expressed by a 
				square matrix with respect to the bases of $E^{p}$ and $F^{p+1}$, and by lemma 
				\ref{lem_order}, this matrix is upper triangular with $\pm 1$ on the diagonal.
			Hence, by Gaussian elimination, the dg-module $C_n$ is homotopy equivalent to
			\begin{equation}\label{eq_newcompl}
				\begin{tikzcd}[baseline=.3cm]
					\cdots \ar[r]	& A^{p-1} \ar[r,"\tilde{d}^{p-1}"]	& A^{p} \ar[r,"\tilde{d}^{p}"]	& A^{p+1} \ar[r]	& \cdots \\
				\end{tikzcd}
			\end{equation}
			where $\tilde{d}^p=a^{p}-g^{p}(\phi^{p})^{-1}f^{p}$.
			
			Recall now that $A$ is spanned by well-tended shrubberies which are
			tensor products of well-tended shrubs of the form $\rho_k$ and $\beta_k$, such that the 
			starting word is a subword of $\sigma_{2n}$. Hence the map sending 
			\begin{align*}
				\tilde{\rho}_k \mapsto \rho_k \\
				\tilde{\beta}_k \mapsto \beta_k
			\end{align*}
			and extended multiplicatively, defines an isomorphism of graded modules between 
			$\tilde{C}_{n}[-2n]$ and \eqref{eq_newcompl}.
			We have to show that also the differential maps agree.
			
			Observe that, as $d$ is compatible with the monoidal structure, and tensoring with elements in
			$A$ stabilizes $A$, $E$ and $F$, then also the new differential $d$ is compatible with the
			monoidal structure. Hence it is sufficient to deal with the case of a single (well-tended)
			shrub. The case $k=1$ is immediate. For $k>1$, we have to prove that 
			\[
				\tilde{d}(\rho_k)=	-\sum_{i=1}^{k-1}\qbc{k}{i}{s}\rho_{k-i} \rho_{i}
										+\sum_{i=1}^{k-1}\qbc{k-1}{i-1}{t}(\rho_i \beta_{k-i}+ \beta_{k-i} \rho_i),
			\]
			and the analogous equality (with signs reversed) for $\tilde{d}(\beta_k)$. 
			
			When $k=2$ we have directly
			\[
				\rho_{2}=
				\ropen \bdot \rclosed \overset{d}{\mapsto} \bdot\rdot - \qn{2}{s}\rdot\rdot
				+\rdot\bdot
				=\beta_1\rho_1 - \qn{2}{s}\rho_{1}\rho_{1}+\rho_1\beta_1.
			\]
			When $k=3$, we have
			\begin{multline*}
				\rho_{3}=
				\ropen \bopen \rdot \bclosed \rclosed
				\overset{d}{\mapsto}
				\bopen\rdot\bclosed\rdot+\rdot\ropen\bdot\rclosed-\qn{2}{t}\ropen\bdot\bdot\rclosed
				+\ropen\bdot\rclosed\rdot+\rdot\bopen\rdot\bclosed=	\\
				=\beta_2\rho_1+\rho_1\rho_2-\qn{2}{t}\ropen\bdot\bdot\rclosed+\rho_2\rho_1+\rho_1\beta_2.
			\end{multline*}
			This can be written, using the notation from the above diagram,
			\begin{align*}
				&a(\rho_3)=\beta_2\rho_1+\rho_1\rho_2+\rho_2\rho_1+\rho_1\beta_2,&
				&f(\rho_3)=-\qn{2}{t}\ropen\bdot\bdot\rclosed.
			\end{align*}
			On the other hand, equation \eqref{eq_size5triv} implies
			\begin{align*}
				&g\big(\ropen\bdot\rcomma\bdot\rclosed\big)=\beta_1\rho_2-\qn{2}{s}\rho_1\rho_2
				-\qn{2}{s}\rho_2\rho_1+\rho_2\beta_1,&
				&\phi\big(\ropen\bdot\rcomma\bdot\rclosed\big)=\ropen\bdot\bdot\rclosed.
			\end{align*}
			Hence
			\begin{align*}
				\tilde{d}(\rho_3)&=
						\beta_2\rho_1+\rho_1\rho_2+\rho_2\rho_1+\rho_1\beta_2
							+\qn{2}{t}\big(\beta_1\rho_2-\qn{2}{s}\rho_1\rho_2
							-\qn{2}{s}\rho_2\rho_1+\rho_2\beta_1\big)\\
				&=
					\begin{multlined}[t]
						(\rho_1\beta_2+\beta_2\rho_1)+\qn{2}{t}(\rho_2\beta_1+\beta_1\rho_2)
						-(\qn{2}{t}\qn{2}{s}-1)\rho_1\rho_2+\\-(\qn{2}{t}\qn{2}{s}-1)\rho_2\rho_1
					\end{multlined}\\
				&=(\rho_1\beta_2+\beta_2\rho_1)+\qn{2}{t}(\rho_2\beta_1+\beta_1\rho_2)
				-\qn{3}{s}\rho_1\rho_2-\qn{3}{s}\rho_2\rho_1\\
				&=(\rho_1\beta_2+\beta_2\rho_1)+\qn{2}{t}(\rho_2\beta_1+\beta_1\rho_2)
				-\qbc{3}{1}{s}\rho_1\rho_2-\qbc{3}{2}{s}\rho_2\rho_1
			\end{align*}
			Now we want to proceed by induction.
			One can see, along the lines of Remark \ref{rmk_breaking}, that the image via the 
			differential map of any shrub is a sum of shrubberies with at most two shrubs. In particular
			$d(\rho_k)$ is such a sum, and looking at the starting word one can see
			that the only terms contained in $F$ have one shrub.
			This implies that also $\tilde{d}(\rho_k)$ is a sum of (well-tended) shrubberies with 
			at most two shrubs. But from the form of well-tended shrubberies we deduce that 
			that all these terms have exactly two shrubs which are not both blue.
			Hence we can write
			\[
				\tilde{d}(\rho_k)=\sum_{i=1}^{k-1} \lambda^{ss}_{k,i} \rho_i \rho_{k-i}
				+\sum_{i=1}^{k-1} \lambda^{st}_{k,i} \rho_i \beta_{k-i}
				+\sum_{i=1}^{k-1} \lambda^{ts}_{k,i} \beta_i \rho_{k-i},
			\]
			In the same way, we argue that $\tilde{d}(\beta_k)$ is a sum of shrubberies with at most two shrubs
			that are not both red, hence:
			\[
				\tilde{d}(\beta_k)=\sum_{i=1}^{k-1} \mu^{ss}_{k,i} \beta_i \beta_{k-i}
				+\sum_{i=1}^{k-1} \mu^{st}_{k,i} \rho_i \beta_{k-i}
				+\sum_{i=1}^{k-1} \mu^{ts}_{k,i} \beta_i \rho_{k-i}.
			\]
			We want to show that 
			\begin{align*}
				&\lambda^{ss}_{k,i}=-\qbc{k}{i}{s},& &\lambda^{st}_{k,i}=\qbc{k-1}{i-1}{t},& &\lambda^{ts}_{k,i}=\qbc{k-1}{k-i-1}{t},& \\
				&\mu^{tt}_{k,i}=\qbc{k}{i}{t},& &\mu^{st}_{k,i}=-\qbc{k-1}{k-i-1}{s},& &\mu^{ts}_{k,i}=-\qbc{k-1}{i-1}{s}.& \\
			\end{align*}
			The differential preserves the horizontal symmetry of the well-tended shrubberies, hence 
			$\lambda^{st}_{k,i}=\lambda^{ts}_{k,k-i}$ and $\mu^{st}_{k,i}=\mu^{ts}_{k,k-i}$. So it is 
			sufficient to determine only one version for the mixed coefficients.
			Notice also that we can easily get, for all $k$
			\begin{equation}\label{eq_lambdaeq1}
				\lambda^{st}_{k,1}=1			
			\end{equation}
			because this is the coefficient 
			of $\rho_1\beta_k$ in the expansion of $a(\rho_k)$ and this shrubbery cannot appear in the expansion
			of any image via $g$ of a shrub in $E$.
			
			Now we use that $\tilde{d}^2=0$ to deduce all the other coefficients.
			The coefficient of $\rho_1 \beta_j \beta_{k-j-1}$ in $\tilde{d}^2(\rho_k)$ is
			\[
				0=\lambda^{st}_{k,1}\mu^{tt}_{k-1,j} - \lambda^{st}_{k,j+1}\lambda^{st}_{j+1,1}=
				\qbc{k-1}{j}{t} - \lambda^{st}_{k,j+1},
			\]
			where for the second equality we used induction and \eqref{eq_lambdaeq1}.
			We deduce that 
			\[
				\lambda^{st}_{k,j+1}=\qbc{k-1}{j}{t}. 
			\]
			Similarly we get $\mu^{ts}_{k,j+1}=-\qbc{k-1}{j}{s}$.
			
			Secondly, the coefficient of $\rho_1\rho_1\beta_{k-2}$ is 
			\[
				0=\lambda^{ss}_{k,1}\lambda^{st}_{k-1,1} - \lambda^{st}_{k,2}\lambda^{ss}_{2,1} +
				\lambda^{st}_{k,1}\mu^{st}_{k-1,k-2}=
				\lambda^{ss}_{k,1} + \qn{k-1}{t}\qn{2}{s} -\qn{k-2}{s},
			\]
			where we used again \eqref{eq_lambdaeq1} and the values of the coefficients known by induction.
			We deduce $\lambda^{ss}_{k,1}=-\qn{k-1}{t}\qn{2}{s} +\qn{k-2}{s}=-\qn{k}{s}$.
			
			Finally, the coefficient of $\rho_1 \rho_{i-1} \rho_{k-i}$ is
			\[
				0=\lambda^{ss}_{k,1}\lambda^{ss}_{k-1,i-1} - \lambda^{ss}_{k,i} \lambda^{ss}_{i,1}
			\]
			and we get
			\[
				\lambda^{ss}_{k,i}=-\frac{\qn{k}{s}\qbc{k-1}{i-1}{s}}{\qn{i}{s}}=-\qbc{k}{i}{s}.
			\]
			In a similar way we get $\mu^{tt}_{k,i}=\qbc{k}{i}{t}$. This concludes the proof.
			\end{proof}
			Passing to the limit, one obtains
			\begin{cor}
				The dg-algebra $\Gamma$ is homotopy equivalent to $\tilde{\Gamma}$.
			\end{cor}		
		\section{The antispherical category}
			The above reduction allows to determine morphism spaces between Wakimoto sheaves in the 
			\textit{anti-spherical category}. Let us first recall the definition of the latter.
			\subsection{Generalities}
			In \cite{RW}, Riche and Williamson describe a diagrammatic categorification of the anti-spherical 
			module of the affine Hecke algebra with respect to the finite Weyl group. 
			Their construction actually extends (see remark \cite[Rmk 4.10]{RW}) to any Coxeter system with respect to
			a parabolic subgroup $W_I$ determined by a subset $I\subset S$.
			Libedinsky and Williamson also define, in \cite{LibWil_anti}, a similar category using the geometric 
			representation of the group $W$.
	
			Let ${}^I\!W$ be the set of minimal length representatives of the left cosets of $W_I$ in $W$.
			Let us call \textit{$I$-words} those words starting with an element of $I$ 
			(called \textit{$I$-sequences} in \cite{LibWil_anti}).
			\begin{dfn}
				The \emph{Bott-Samelson anti-spherical category} $\Da_{\mathrm{BS}}$ 
				is the quotient of the Bott-Samelson diagrammatic Hecke 
				category by those morphisms which are of the form $\gamma\cdot \phi$ with
				$\gamma\in \hh^*$ and those which factor through $I$-words.			
			\end{dfn}
			This is naturally a graded $\Ra$-linear\footnote{Here we follow \cite{RW}. One could also 
			quotient only by the roots $\alpha_s$, with $s\in I$ as done in \cite{LibWil_anti}.}
			category.
	
			To clarify which category we are considering, the objects in $\Da_{\mathrm{BS}}$ will be 
			denoted by ${}^I\! B_{\uw}$. Then one takes $\Da$ to be the Karoubi envelope of
			$\Da_{\mathrm{BS}}$.
			Over a complete local ring, one can also define $\Da$ to be the quotient of $\D$ by the 
			ideal generated by the indecomposale objects $B_x(k)$ for all $x\notin {}^I\!W$. Then one 
			can see that the two definitions agree (see \cite[Proposition\ 3.2]{LibWil_anti}\footnote{The proof 
			there is made over $\mathbb{R}$ but the only property of the Kazhdan-Lusztig basis which is 
			used is also shared by the $p$-canonical basis.}).
			
			One has a natural full, essentially surjective functor
			\begin{equation}\label{eq_functanti}
				{}^I\!(-):\D_{(\mathrm{BS})}\rightarrow \Da_{(\mathrm{BS})}	
			\end{equation}
			sending the object $B_{\uw}$ to ${}^I\!B_{\uw}$.
			The category $\Da$ (and already $\Da_{\mathrm{BS}}$) is endowed with a natural right action of the 
			Hecke category given by ${}^I\! B_{\uw}\cdot B_s={}^I\! B_{\uw \underline{s}}$.
			
			Notice that in the $\tilde{A_1}$ case the only interesting choices for $I$ are $\{s\}$ or $\{t\}$.
			\subsection{Double leaves bases}
			Both \cite[Lemma\ 4.7, Remark\ 4.10]{RW} and \cite[Theorem\ 3.7]{LibWil_anti} prove that 
			a certain subset of the light leaves maps (or rather, the double leaves maps obtained from them)
			are spanning sets for the morphism spaces in the categorified anti-spherical module.
	
			Namely, given a Coxeter word $\uw=s_1\cdots s_k$, we say that a 01-sequence 
			$\boe$ \textit{avoids} $K\subset W$, when,
			for all $i=0,1,\dots, k-1$, we have $\uw_{\le i}^{\boe_{\le i}}s_{i+1}\notin K$.
			This means that in the corresponding Bruhat stroll we never even consider to pass through $K$. 
			Then the considered light leaves maps are the $L_{\uw,\boe}$ such that $\boe$ avoids $W\setminus {}^I\! W$.
			\begin{rmk}
				Consider the spaces $\Hom(B_{\uw},\un)$ in $\tilde{A_1}$. One can rather easily see that 
				01-sequences avoiding $W\setminus {}^I\! W$ 
				correspond precisely to blue (red) shrubberies if $I=\{s\}$ (respectively $I=\{t\}$). So
				they span, over $\Ra$, the dg-modules $\Casph$ and $\Gammaa$.
			\end{rmk}	
			Furthermore, both \cite{RW} and \cite{LibWil_anti} show that these sets form bases for the 
			realizations considered in those works.
			It is actually not difficult to show directly that they form a basis of $\Hom(B_{\uw},\un)$, 
			for all realizations of $\tilde{A_1}$.
			\begin{lem}\label{lem_basis}
				Blue (respectively red) shrubberies form an $\Ra$-basis for $\Hom(B_{\uw},\un)$,
				when $I=S_{\mathrm{f}}=\{s\}$ (or $I=\{t\}$ respectively).
			\end{lem}
			\begin{proof}
			 	Take $I=\{s\}$. It is sufficient to prove that the morphisms factoring through 
			 	$I$-words are precisely those which can be written in the form
			 	\[
			 		\alpha_s \sum_{K\in \Bshr} f_K K + \sum_{K\notin \Bshr} f_K K.
			 	\]
				In fact then it is easy to see that blue shrubberies span (just consider surjectivity of the map
				${}^I\!a$) and if a $\Ra$-linear combination
				\[
					\sum_{K\in \Bshr} f'_K K
				\]
			 	is zero in $\Gammaa$, then, in $\Gamma$, it can be written in the above form. But 
			 	this implies $f_K'=\alpha_s f_K$, so it is zero in $\Ra$. 
		
				Let us then prove the claim. Take a morphism $\phi$ in $\Hom(B_{\uw},\un)$ factoring 
				through $B_{\ux}$ with $\ux$ a word starting with $s$.
				Then, using relation \eqref{dotline}, we can factor it as
				\[
					\phi =
					\begin{tikzpicture}[baseline=1cm]
						\def\hh{.5cm}
						\draw (.5*\d,0)--(4.5*\d,0);
						\foreach \i in {1,2,4}{%
							\draw[violet] (\i*\d,0)--(\i*\d,\hh);
						}
						\node at (3*\d,.5*\hh) {\dots};
						\draw (.75*\d,\hh)--(4.25*\d,\hh)--(4.25*\d,2*\hh)--(.75*\d,2*\hh)--cycle;
						\node[font=\large] at (2.5*\d,1.3*\hh) {*};
						\draw[red] (\d,2*\hh)--(\d,3*\hh);
						\foreach \i in {2,4}{%
							\draw[violet,yshift=2*\hh] (\i*\d,0)--(\i*\d,\hh);
						}
						\node at (3*\d,2.5*\hh) {\dots};
						\draw[yshift=2*\hh] (.75*\d,\hh)--(4.25*\d,\hh)--(4.25*\d,2*\hh)--(.75*\d,2*\hh)--cycle;
						\node[font=\large] at (2.5*\d,3.3*\hh) {*};
						\draw (.5*\d,5*\hh)--(4.5*\d,5*\hh);
					\end{tikzpicture}
					=
					\begin{tikzpicture}[baseline=1cm]
						\def\hh{.5cm}
						\draw (0,0)--(4.5*\d,0);
						\foreach \i in {1,2,4}{%
							\draw[violet] (\i*\d,0)--(\i*\d,\hh);
						}
						\node at (3*\d,.5*\hh) {\dots};
						\draw (.75*\d,\hh)--(4.25*\d,\hh)--(4.25*\d,2*\hh)--(.75*\d,2*\hh)--cycle;
						\node[font=\large] at (2.5*\d,1.3*\hh) {*};
						\draw[red] (\d,2*\hh)--(\d,3*\hh);
						\foreach \i in {2,4}{%
							\draw[violet,yshift=2*\hh] (\i*\d,0)--(\i*\d,\hh);
						}
						\node at (3*\d,2.5*\hh) {\dots};
						\draw[yshift=2*\hh] (.75*\d,\hh)--(4.25*\d,\hh)--(4.25*\d,2*\hh)--(.75*\d,2*\hh)--cycle;
						\node[font=\large] at (2.5*\d,3.3*\hh) {*};
						\draw (0,5*\hh)--(4.5*\d,5*\hh);
						\draw[red] (\d,2.5*\hh) ..controls (.45*\d,2.5*\hh).. (.5*\d,.4*\hh); \fill[red] (.5*\d,.4*\hh) circle (1.5pt);
					\end{tikzpicture}
					=  \phi_1 \circ \Big(\begin{tikzpicture}[baseline=.4*\d]
										\draw (0,0)--(\d,0); \draw(0,\d)--(\d,\d); \draw[red] (.5*\d,.5*\d)--(.5*\d,\d);\fill[red] (.5*\d,.5*\d) circle (1.5pt);
									\end{tikzpicture}\otimes \id_{B_{\uw}}\Big)
				\]
				Where $\phi_1\in \Hom(B_{s\uw},\un)$. Write $\phi_1$ as a linear combination of shrubberies
				\[
					\phi_1=\sum_{K} f_K K
				\] 
				and notice that by the form of the starting word, all $K$'s must start with a red shrub. 
				Consider the compositions
				\[
					K\circ  \Big(\begin{tikzpicture}[baseline=.4*\d]
										\draw (0,0)--(\d,0); \draw(0,\d)--(\d,\d); \draw[red] (.5*\d,.5*\d)--(.5*\d,\d);\fill[red] (.5*\d,.5*\d) circle (1.5pt);
									\end{tikzpicture}\otimes \id_{B_{\uw}}\Big).
				\]
				If the first shrub of $K$ is a red dot then we get a factor $\alpha_s$. Otherwise the above composition will give a shrubbery wich still has at least one red shrub, which concludes the proof of the claim.
			\end{proof}
			\subsection{Categories of complexes}
			The functor \eqref{eq_functanti} induces a functor on the level of homotopy categories, 
			as well as on the level of dg categories of complexes, that will still be denoted by ${}^I\!(-)$.
			The above right action induces a right (dg) action of $\Kb(\D)$, and $\Cdg(\D)$, on 
			$\Kb(\Da)$, and $\Cdg(\Da)$, respectively.
			For a braid word $\uom$ and a letter $\sigma\in \Sigma$, we have 
			${}^I\!F_{\uom}\cdot F_{\sigma}={}^I\!F_{\uom\sigma}$. Hence 
			$-\cdot F_{\sigma}$ and $-\cdot F_{\sigma^{-1}}$ define mutually inverse self-equivalences
			of $\Kb(\Da)$.
			\subsection{Antispherical Wakimoto sheaves}\label{subs_antisphwaki}
			We want to describe morphism spaces between the images ${}^I\!\Wakr{k}$
			of the Wakimoto sheaves (the homotopy equivalence $\Wak{k}^\bullet\simeq \Wakr{k}$ carries over to
			the anti-spherical category), and we see them as the cohomology groups of the dg morphism spaces
			\begin{equation}\label{eq_morph_antiW}
				\Homb({}^I\!\Wakr{k_1},{}^I\!\Wakr{k_2}).
			\end{equation}
			Using the right action, one can reduce \eqref{eq_morph_antiW} to the case $k_2=0$ as before.
	
			Furthermore, by (the anti-spherical version of) the Rouquier formula (see Proposition 
			\ref{pro_rouf}) 
			this gives zero when $k_1$ is positive. 
			So we are again reduced to study
			\[
				\Casph:=\Homb({}^I\!\Wakr{-n},\un)
			\]
			for $n>0$. We will denote ${}^I\!d$ the differential map of $\Casph$.
			As a graded module we can still write it as
			\[
				\Casph=\bigoplus_{\ux\preceq \us_{2n}}\Hom_{\Da}({}^I\!B_{\ux},\un)/\Lambda_{\ux}.
			\]
			Notice that the terms for $\ux$ an $I$-word are zero. 
			As before, we can consider 
			\[
				\Gammaa=\bigoplus_{\ux}\Hom_{\Da}({}^I\!B_{\ux},\un)/\Lambda_{\ux}.
			\]
			The functor \eqref{eq_functanti} induces a surjective morphism of dg-modules
			${}^I\!a:\Gamma\rightarrow\Gammaa$, which 
			endows $\Gammaa$ with a structure of dg-algebra: the product is 
			${}^I\!a(\gamma_1){}^I\!a(\gamma_2):={}^I\!a(\gamma_1\gamma_2)$ and it is well defined 
			because if $\delta_1$ and $\delta_2$ are morphisms in $\Gamma$ factoring through $I$-words, 
			then
			\[
				{}^I\!a\big((\gamma_1+\delta_1)(\gamma_2+\delta_2)\big)={}^I\!a(\gamma_1\gamma_2+\gamma_1\delta_2+\delta_1\gamma_2+\delta_1\delta_2)={}^I\!a(\gamma_1\gamma_2).
			\]
			In fact the three morphisms $\gamma_1\delta_2$, $\delta_1\gamma_2$ and $\delta_1\delta_2$ all 
			factor through $I$-words, their target being the unit.
	
			By Lemma \ref{lem_basis}, a basis of $\Casph$ (or $\Gammaa$) is given by the 
			collection of blue (or red) shrubberies without empty arches, whose starting word is a 
			subword of $\us_{2n}$ (respectively, any Coxeter word).
	
			The reduction of last section is compatible with the passage to the antispherical category. 
			In particular, consider the quotient ${}^I\!\tilde{\Gamma}$ of $\tilde{\Gamma}$ 
			obtained by imposing $\rho_k=0$ for all $k>0$. One can define ${}^I\!\tilde{C}_n$ in a
			similar way.
			\begin{pro}
				The complex $\Gammaa$ is homotopy equivalent to ${}^I\!\tilde{\Gamma}$ and 
				$\Casph$ is homotopy equivalent to ${}^I\!\tilde{C}_n$, in such a way that 
				the following diagrams commute:
				\begin{align*}
					&\begin{tikzcd}[ampersand replacement=\&]
						\Gamma \ar[d,two heads] \ar[r,"\sim"]	\& \tilde{\Gamma} \ar[d,two heads] \\
						{}^I\!\Gamma \ar[r,"\sim"]				\& {}^I\!\tilde{\Gamma}
					\end{tikzcd}&
					&\begin{tikzcd}[ampersand replacement=\&]
						C_n \ar[d,two heads] \ar[r,"\sim"]	\& \tilde{C}_n[-2n] \ar[d,two heads] \\
						{}^I\!C_n \ar[r,"\sim"]				\& {}^I\!\tilde{C}_n[-2n]
					\end{tikzcd}
				\end{align*}			
			\end{pro}
			Notice that, by the form of the differential, the complex 
			${}^I\!\tilde{C}_n$ splits into smaller complexes:
			\begin{equation}\label{eq_decompositiontildeC}
				{}^I\!\tilde{C}_n=\bigoplus_{m\le n} B_m,			
			\end{equation}
			where $B_m$ is the span of the monomials $\tibe_{k_1}\tibe_{k_2}\dots\tibe_{k_r}$ with
			$k_1+k_2+\dots+k_r=m$.
		\begin{exa}\label{exa_splitC4}
			The complex ${}^I\! \tilde{C}_4$ decomposes into four pieces. The first is:
			\[
				\begin{tikzcd}[row sep=tiny,column sep=large]
																					& \Ra\cdot \tibe_1\tibe_3 \ar[r,"-c^3_1"]\ar[ddr,"-c^3_2", near start]			& \Ra\cdot \tibe_1\tibe_1\tibe_2 \ar[ddr,"c^2_1"]										&		\\
																					& \oplus											& \oplus														&		\\
					\Ra\cdot \tibe_{4} \ar[uur,"c^4_1"] \ar[r,"c^4_2"] \ar[ddr,swap,"c^4_3"]	& \Ra\cdot \tibe_2\tibe_2 \ar[uur,crossing over,"c^2_1", near start]				& \Ra\cdot \tibe_1\tibe_2\tibe_1 \ar[r,"-c^2_1"]											& \Ra\cdot \tibe_1\tibe_1\tibe_1\tibe_1		\\
																					& \oplus											& \oplus														&		\\
																					& \Ra\cdot \tibe_3\tibe_1 \ar[uur,"c^3_1", near start] \ar[r,swap,"c^3_2"]			& \Ra\cdot \tibe_2\tibe_1\tibe_1 \ar[from=uul,crossing over,"-c^2_1", near start]\ar[uur,swap,"c^2_1"]	&		
				\end{tikzcd}
			\]
			where $c^i_j$ is a shortcut for $\qbc{i}{j}{t}$. The other three pieces are:
			\[
				\begin{tikzcd}[row sep=tiny,column sep=small]
																		& \Ra\cdot \tibe_1\tibe_2 \ar[rd,"c^2_1"]		&	\\
					\Ra\cdot \tibe_3 \ar[ru,"c^3_1"]\ar[rd,swap,"c^3_2"]& \oplus										& \Ra \cdot \tibe_1\tibe_1\tibe_1 \\
																		& \Ra\cdot \tibe_2\tibe_1 \ar[ru,swap,"-c^2_1"]	&
				\end{tikzcd},
				\,
				\begin{tikzcd}[column sep=small]
					\Ra \cdot \tibe_2 \ar[r,"c^2_1"]	& \Ra\cdot \tibe_1\tibe_1
				\end{tikzcd},\,
				\Ra\cdot \tibe_1,\,
				\Ra\cdot \tibe_0
			\]
		\end{exa}		
		We want to compute the cohomology of this dg module. First we need some more properties of 
		two-color quantum binomial coefficients.
		\subsection{Even more on two-color quantum numbers}
		The two-color quantum numbers, as the standard ones, can be factorized in $\ZZ[x,y]$ into
		\emph{two color cyclotomic polynomials}.
		\begin{pro}
			Consider the decomposition of the polynomials $\qn{n}{x}$ and $\qn{n}{y}$ into irreducible
			factors. Then we have:
			\begin{enumerate}
				\item each $\qn{n}{x}$ has a factor $\phi_{n,x}\in \ZZ[x,y]$ not appearing in the decomposition of 
					$\phi_{k,x}$ for $k<n$, and similarly for $y$;
				\item one has
					\[
						\qn{n}{x}=\prod_{d\mid n} \phi_{d,x}
					\]
					and similarly for $y$;
				\item\label{item_great2} for $n>2$, we have $\phi_{n,x}=\phi_{n,y}$.
			\end{enumerate}
		\end{pro}
		\begin{proof}
			If $n$ is odd, then $\qn{n}{x}=\qn{n}{y}$ is a polynomial in $xy$. If $n$ is even, then
			$\qn{n}{x}/x$ is a polynomial in $xy$, and similarly for $y$. For $n\ge 3$ let $p_n\in \ZZ[t]$
			be the minimal polynomial of $4\cos^2(\pi/n)$. Then let $\phi_n:=p_n(xy)$. Now notice that 
			the $n$-th symmetric quantum numbers in $\Laur$ vanishes at $e^{\frac{\pi i}{n}}$.
			Using the specialization from Remark \ref{rmk_specquantumnumb}, we deduce that $\qn{n}{x}$ and 
			$\qn{n}{y}$ are divisible by $\phi_n$. One then proves that under this specialization, the $\phi_d$ give
			(the symmetric version of) the usual cyclotomic polynomials. The other properties are 
			proved as those of the standard cyclotomic polynomials.
		\end{proof}
		\begin{dfn}
			Let $n\ge 2$. The polynomials $\phi_{n,x}$ and $\phi_{n,y}$ from the proposition are called
			\emph{two-color cyclotomic polynomials}. For $n>2$ we will simply write $\phi_n$.
		\end{dfn}
		The following two results will be useful in the computation of the cohomology.
		\begin{lem}\label{lem_factorsbinomial}
			Let $d,n\ge 2$ with $d\mid n$. Then $\phi_d\mid \qbc{n}{k}{y}$ if and only if 
			$d\nmid k$.
		\end{lem}
		\begin{proof}
			One can argue just as in the standard case.
		\end{proof}
		In fact, as in the standard case, two-color cyclotomic polynomials show the nice
		factorization pattern in the \emph{two-color quantum Pascal triangle} (i.e.\ the triangle of two-color quantum binomial coefficients), shown in Figure \ref{fig_pascal}.
		\begin{figure}
			\begin{tikzpicture}[x=.75cm,y=.7cm]
				\foreach \i in {0,1,2,3,4,5,6,7}{
					\node at (\i,-\i) {1};
					}
				\foreach \i in {1,2,3,4,5,6,7}{
					\node at (-\i,-\i) {1};
					}
				\foreach \i in {2,4,6,8}{
					\node[xscale=-1] at (-\i,-8) {$\ddots$};
					}
				\foreach \i in {2,4,6,8}{
					\node at (\i,-8) {$\ddots$};
					}
				\node at (0,-8) {$\vdots$};
				\node at (0,-2) {$\textcolor{green}{\phi_2}$};
				\node at (-1,-3) {$\textcolor{red}{\phi_3}$};\node at (1,-3) {$\textcolor{red}{\phi_3}$};
				\node at (-2,-4) {$\textcolor{green}{\phi_2}\textcolor{blue}{\phi_4}$};\node at (0,-4) {$\textcolor{red}{\phi_3}\textcolor{blue}{\phi_4}$};\node at (2,-4) {$\textcolor{green}{\phi_2}\textcolor{blue}{\phi_4}$};
				\node at (-3,-5) {$\textcolor{orange}{\phi_5}$};\node at (-1,-5) {$\textcolor{blue}{\phi_4}\textcolor{orange}{\phi_5}$};\node at (1,-5) {$\textcolor{blue}{\phi_4}\textcolor{orange}{\phi_5}$};\node at (3,-5) {$\textcolor{orange}{\phi_5}$};
				\node at (-4,-6) {$\textcolor{green}{\phi_2}\textcolor{red}{\phi_3}\textcolor{cyan}{\phi_6}$};\node at (-2,-6) {$\textcolor{red}{\phi_3}\textcolor{orange}{\phi_5}\textcolor{cyan}{\phi_6}$};\node at (0,-6) {$\textcolor{green}{\phi_2}\textcolor{blue}{\phi_4}\textcolor{orange}{\phi_5}\textcolor{cyan}{\phi_6}$};\node at (2,-6) {$\textcolor{red}{\phi_3}\textcolor{orange}{\phi_5}\textcolor{cyan}{\phi_6}$};\node at (4,-6) {$\textcolor{red}{\phi_3}\textcolor{green}{\phi_2}\textcolor{cyan}{\phi_6}$};
				\node at (-5,-7) {$\textcolor{violet}{\phi_7}$};\node at (-3,-7) {$\textcolor{red}{\phi_3}\textcolor{cyan}{\phi_6}\textcolor{violet}{\phi_7}$};\node at (-1,-7) {$\textcolor{orange}{\phi_5}\textcolor{cyan}{\phi_6}\textcolor{violet}{\phi_7}$};\node at (1,-7) {$\textcolor{orange}{\phi_5}\textcolor{cyan}{\phi_6}\textcolor{violet}{\phi_7}$};\node at (3,-7) {$\textcolor{red}{\phi_3}\textcolor{cyan}{\phi_6}\textcolor{violet}{\phi_7}$};\node at (5,-7) {$\textcolor{violet}{\phi_7}$};
			\end{tikzpicture}
			\caption{The two-color quantum Pascal triangle}\label{fig_pascal}
		\end{figure}
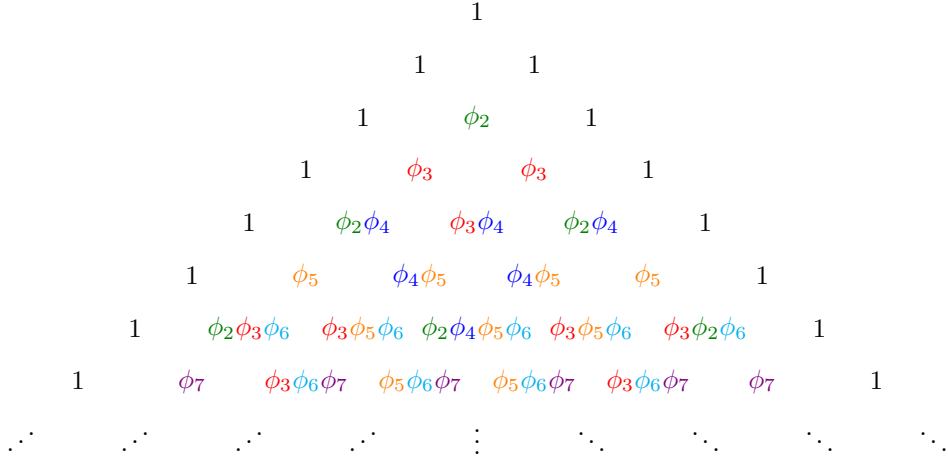
		\begin{lem}\label{lem_gcdcyclo}
			If $k<n$ are natural numbers with $k\nmid n$, then the ideal generated by $\phi_{k,x}$ and $\phi_{n,x}$ is the whole 
			ring $\ZZ[x,y]$:
			\[
				(\phi_{k,x},\phi_{n,x})=(1)
			\]
		\end{lem}
		\begin{proof}
			Again one can deduce the result from the analogous property of usual cyclotomic polynomials,
			via the specialization map.
		\end{proof}
		\begin{rmk}
			With the standard Cartan matrix the cyclotomic polynomials give:
			\[
				\phi_{n,x}(2,2)=\phi_{n,y}(2,2)=e^{\Lambda(n)},
			\]
			where $\Lambda$ denotes the Von Mangoldt function
			\[
				\Lambda(n)=
					\begin{cases}
						\log(p) & \text{if $n=p^r$ with $p$ prime,}\\
						0		& \text{otherwise.}
					\end{cases}
			\]
		\end{rmk}			
		\subsection{Extension groups}\label{subs_extgrpsZ}
		We now use our reduction to compute the cohomology of the complex ${}^I\!\tilde{C}_n$.
		We will work over $\kk=\ZZ[x,y]$ so that one can specialize to any realization.
		For convenience, let $\qn{n}{}:=\qn{n}{y}$ and $\qbc{n}{k}{}=\qbc{n}{k}{y}$. Let also $\phi_2:=\phi_{2,y}$.
		
		We first describe the result. We call a partition 
		$\lambda=(\lambda_1\ge \lambda_2\ge \dots\ge \lambda_{k-1}\ge \lambda_k)$ of 
		$n$ \emph{distinguished} if the parts divide each other:
		\[
			\lambda_k\mid \lambda_{k-1} \mid \dots \mid \lambda_2\mid\lambda_1.
		\]
		Let $\hat{P}(n)$ denote the set of distinguished partitions.
		For such a partition, let $I_{\lambda}$ be the ideal of $\kk$ generated by the corresponding
		cyclotomic polynomials (of color $y$):
		\[
			I_{\lambda}=(\phi_{\lambda_1},\phi_{\lambda_2},\dots,\phi_{\lambda_{k-1}},\phi_{\lambda_k}).
		\]
		For a given partition $\lambda$ with $k$ parts and with $d$ distinct numbers appearing,
		the \emph{weight} $|\lambda|$ is defined as $2k-d$. In other words, it is the sum between the number
		of parts and the number of repetitions. For example the partition $(5,4,4,2,1)$ of $16$ has weight
		$6$.
		
		Now, define the graded $\kk$-modules $H_k$ as follows. We set $H_1=\kk[0]$, and for $k\ge 2$,
		\[
			H_k:=\bigoplus_{\lambda\in\hat{P}(k)} \kk/I_\lambda[1-|\lambda|].
		\]
		Then we have
		\begin{thm}\label{thm_cohomologyZ}
			The cohomology of ${}^I\!\tilde{C}_n$ is
			\[
				H^\bullet({}^I\!\tilde{C}_n)=\bigoplus_{i=2}^{2n}H_{\lfloor i/2 \rfloor} [i-2],
			\]
			where $\lfloor\cdot\rfloor$ denotes the floor function.
		\end{thm}
		\begin{exa}
			Let $n=6$. We first compute the $H_k$ for $k$ up to $6$. Of course, for $k$ prime 
			$H_k$ is simply $\kk/(\phi_k)$. Hence it remains to compute $H_4$ and $H_6$. We have
			\begin{align*}
				&\hat{P}(4)=\{(4),(2,2)\},\\
				&\hat{P}(6)=\{(6),(4,2),(3,3),(2,2,2)\}.
			\end{align*}
			Then
			\begin{align*}
				&H_4=\kk/(\phi_4)\oplus \kk/(\phi_2)[-2],\\
				&H_6=\kk/(\phi_6)\oplus\kk/(\phi_4,\phi_2)[-1]\oplus\kk/(\phi_3)[-2]\oplus\kk/(\phi_2)[-4].
			\end{align*}
			So we can patch them together according to Theorem \ref{thm_cohomologyZ} and get Table \ref{tab_coh8}
			where the entry $i$ represents $\kk/(\phi_i)$ and $i,j$ represents $\kk/(\phi_i,\phi_j)$. 
			In the different rows one can see the contributions of the different $H_i$.
			\begin{table}[h]
			\[
				\begin{array}{|r|c|c|c|c|c|c|c|c|c|c|c|} 
					\hline
															& -10	& -9	& -8	& -7	& -6	& -5	& -4	& -3	& -2	& -1	& 0		\\
					\hline
					H_1		
															& 		& 		& 		& 		& 		& 		& 		& 		& 		& 		& \kk 	\\
					\hline
					H_1[1]	
															& 		& 	 	& 		& 		& 		& 		& 		&		& 		& \kk	& 		\\
					\hline
					H_2[2]	
															&		& 		& 		&		&		&		&		& 		& 2		&		& 		\\
					\hline
					H_2[3]	
															&		& 		& 		&		&		&		&		& 2		&		&		& 		\\
					\hline
					H_3[4]	
															& 		& 		& 		& 		& 		& 		& 3		& 		& 		& 		& 	 	\\
					\hline
					H_3[5] 
															& 		& 		& 		& 		& 		& 3		& 		& 		& 		& 		& 	 	\\
					\hline
					H_4[6]	
															& 		& 		& 		& 		& 4		& 		& 2		& 		& 		& 		& 	 	\\
					\hline
					H_4[7]	
															& 		& 		& 		& 4		& 		& 2		& 		& 		& 		& 		& 	 	\\
					\hline
					H_5[8]	
															& 		& 		& 5		& 		& 		& 		& 		& 		& 		& 		& 	 	\\
					\hline
					H_5[9]	
															& 		& 5		& 		& 		& 		& 		& 		& 		& 		& 		& 	 	\\
					\hline
					H_6[10]	
															& 6		& 2,4	& 3		& 		& 2		& 		& 		& 		& 		& 		& 	 	\\
					\hline
				\end{array}
			\]
			\caption{The cohomology of ${}^I\!\tilde{C}_8$.}\label{tab_coh8}
			\end{table}
			In each column one can read off the cohomology at the corresponding degree.
			For example, the cohomology at degree $-9$ is 
			\[
				\kk/(\phi_5)\oplus \kk/(\phi_2,\phi_4)\oplus \kk/(\phi_3).
			\]
		\end{exa}
		Before we pass to the proof of the theorem, let us show how this pattern emerges with an example.
		\begin{exa}
			Let again $n=6$. By \eqref{eq_decompositiontildeC}, we can restrict to the $B_m$'s, so let us consider
			the complex $B_6$, illustrated in Figure \ref{fig_B6}.
			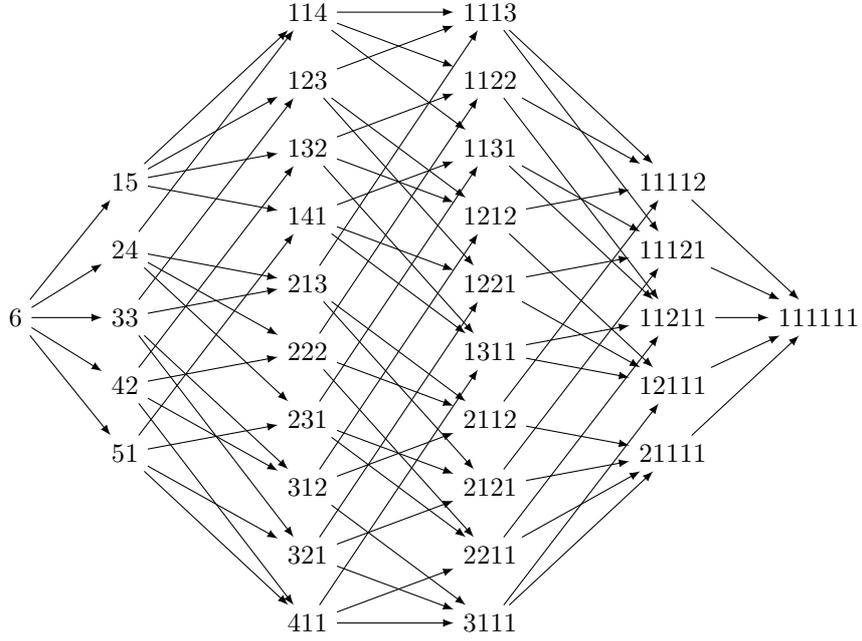
\begin{figure}
				\centering
				\begin{tikzpicture}[x=2.4cm,y=.9cm]
					\node (a6) at (0.4,0) {6};
					\node (a15) at (1,2) {15};\node (a24) at (1,1) {24};\node (a33) at (1,0) {33};\node (a42) at (1,-1) {42};\node (a51) at (1,-2) {51};
					\node (a114) at (2,4.5) {114};\node (a123) at (2,3.5) {123};\node (a132) at (2,2.5) {132};\node (a141) at (2,1.5) {141};\node (a213) at (2,0.5) {213};\node (a222) at (2,-0.5) {222};\node (a231) at (2,-1.5) {231};\node (a312) at (2,-2.5) {312};\node (a321) at (2,-3.5) {321};\node (a411) at (2,-4.5) {411};
					\node (a1113) at (3,4.5) {1113};\node (a1122) at (3,3.5) {1122};\node (a1131) at (3,2.5) {1131};\node (a1212) at (3,1.5) {1212};\node (a1221) at (3,0.5) {1221};\node (a1311) at (3,-0.5) {1311};\node (a2112) at (3,-1.5) {2112};\node (a2121) at (3,-2.5) {2121};\node (a2211) at (3,-3.5) {2211};\node (a3111) at (3,-4.5) {3111};
					\node (a11112) at (4,2) {11112};\node (a11121) at (4,1) {11121};\node (a11211) at (4,0) {11211};\node (a12111) at (4,-1) {12111};\node (a21111) at (4,-2) {21111};
					\node (a111111) at (4.8,0) {111111};
					\draw[-latex] (a6) to (a15);\draw[-latex] (a6) to (a24);\draw[-latex] (a6) to (a33);\draw[-latex] (a6) to (a42);\draw[-latex] (a6) to (a51);
					\draw[-latex] (a15) to (a114);\draw[-latex] (a15) to (a123);\draw[-latex] (a15) to (a132);\draw[-latex] (a15) to (a141);
					\draw[-latex] (a24) to (a114);\draw[-latex] (a24) to (a213);\draw[-latex] (a24) to (a222);\draw[-latex] (a24) to (a231);
					\draw[-latex] (a33) to (a123);\draw[-latex] (a33) to (a213);\draw[-latex] (a33) to (a312);\draw[-latex] (a33) to (a321);
					\draw[-latex] (a42) to (a132);\draw[-latex] (a42) to (a222);\draw[-latex] (a42) to (a312);\draw[-latex] (a42) to (a411);
					\draw[-latex] (a51) to (a141);\draw[-latex] (a51) to (a231);\draw[-latex] (a51) to (a321);\draw[-latex] (a51) to (a411);
					\draw[-latex] (a114) to (a1113);\draw[-latex] (a114) to (a1122);\draw[-latex] (a114) to (a1131);
					\draw[-latex] (a123) to (a1113);\draw[-latex] (a123) to (a1212);\draw[-latex] (a123) to (a1221);
					\draw[-latex] (a132) to (a1122);\draw[-latex] (a132) to (a1212);\draw[-latex] (a132) to (a1311);
					\draw[-latex] (a141) to (a1131);\draw[-latex] (a141) to (a1221);\draw[-latex] (a141) to (a1311);
					\draw[-latex] (a213) to (a1113);\draw[-latex] (a213) to (a2112);\draw[-latex] (a213) to (a2121);
					\draw[-latex] (a222) to (a1122);\draw[-latex] (a222) to (a2112);\draw[-latex] (a222) to (a2211);
					\draw[-latex] (a231) to (a1131);\draw[-latex] (a231) to (a2121);\draw[-latex] (a231) to (a2211);
					\draw[-latex] (a312) to (a1212);\draw[-latex] (a312) to (a2112);\draw[-latex] (a312) to (a3111);
					\draw[-latex] (a321) to (a1221);\draw[-latex] (a321) to (a2121);\draw[-latex] (a321) to (a3111);
					\draw[-latex] (a411) to (a1311);\draw[-latex] (a411) to (a2211);\draw[-latex] (a411) to (a3111);
					\draw[-latex] (a1113) to (a11112);\draw[-latex] (a1113) to (a11121);
					\draw[-latex] (a1122) to (a11112);\draw[-latex] (a1122) to (a11211);
					\draw[-latex] (a1131) to (a11121);\draw[-latex] (a1131) to (a11211);
					\draw[-latex] (a1212) to (a11112);\draw[-latex] (a1212) to (a12111);
					\draw[-latex] (a1221) to (a11121);\draw[-latex] (a1221) to (a12111);
					\draw[-latex] (a1311) to (a11211);\draw[-latex] (a1311) to (a12111);
					\draw[-latex] (a2112) to (a11112);\draw[-latex] (a2112) to (a21111);
					\draw[-latex] (a2121) to (a11121);\draw[-latex] (a2121) to (a21111);
					\draw[-latex] (a2211) to (a11211);\draw[-latex] (a2211) to (a21111);
					\draw[-latex] (a3111) to (a12111);\draw[-latex] (a3111) to (a21111);
					\draw[-latex] (a11112) to (a111111);\draw[-latex] (a11121) to (a111111);\draw[-latex] (a11211) to (a111111);\draw[-latex] (a12111) to (a111111);\draw[-latex] (a21111) to (a111111);
				\end{tikzpicture}
				\caption{The complex $B_6$.}\label{fig_B6}
			\end{figure}
			Each node is labeled by the composition of 6 representing a basis element. For instance, 
			$213$ represents $\tibe_2\tibe_1\tibe_3$.
			
			We want to find a convenient change of basis in order to split the complex into smaller pieces.
			The basis of $B_6$ is given in terms of the variables $\tibe_k$ and the differential is 
			defined multiplicatively from its values over the $\tibe_k$. Hence the idea is to define 
			new variables such that the formula for the differential gets simpler.
			In fact one can find variables $\gamma_1$, and $\gamma_k^{\pm}$, for $k=2,\dots,6$
			such that
			\begin{align}\label{eq_diffongamma}
				&d(\gamma_1)=0,&	&d(\gamma_k^+)=\qn{k}{}\gamma_k^-,&	&d(\gamma_k^-)=0.
			\end{align}
			This can be done as follows. Let $\gamma_1:=\tibe_1$ and 
			$\gamma_k^+:=\tilde{\beta}_k$ for all $k=2,3,4,5,6$. Then let
			\begin{align*}
				&\gamma_2^-=\tibe_1\tibe_1,\\
				&\gamma_3^-=\tibe_1\tibe_2+\tibe_2\tibe_1,\\
				&\gamma_4^-=\tibe_1\tibe_3+\frac{\qbc{4}{2}{}}{\qn{6}{}}\tibe_2\tibe_2+\tibe_3\tibe_1,\\
				&\gamma_5^-=\tibe_1\tibe_4+\frac{\qbc{5}{2}{}}{\qn{5}{}}\tibe_2\tibe_3+\frac{\qbc{5}{3}{}}{\qn{5}{}}\tibe_3\tibe_2 +\tibe_4\tibe_1,\\
				&\gamma_6^-=\tibe_1\tibe_5+\frac{\qbc{6}{2}{}}{\qn{6}{}}\tibe_2\tibe_4+\frac{\qbc{6}{3}{}}{\qn{6}{}}\tibe_3\tibe_3+\frac{\qbc{6}{4}{}}{\qn{6}{}}\tibe_4\tibe_2+\tibe_1\tibe_5. 
			\end{align*}
			By the values of the differential on the $\tibe_k$'s, one sees that we actually get \eqref{eq_diffongamma}.
			Now we want to find a basis of $B_6$ in terms of the new variables. Consider the set:
			\begin{align*}
				&\gamma_6^{\pm},							&\text{(2 elements)}\\
				&\gamma_4^{\pm}\gamma_2^{\pm},				&\text{(4 elements)}\\
				&\gamma_3^\pm \gamma_3^\pm,					&\text{(4 elements)}\\
				&\gamma_2^\pm \gamma_4^\pm,					&\text{(4 elements)}\\
				&\gamma_2^\pm \gamma_2^\pm \gamma_2^\pm,	&\text{(8 elements)}\\
				&\gamma_1\gamma_5^\pm,						&\text{(2 elements)}\\
				&\gamma_1\gamma_3^\pm\gamma_2^\pm,			&\text{(4 elements)}\\
				&\gamma_1\gamma_2^\pm \gamma_3^\pm.			&\text{(4 elements)}\\
			\end{align*}
			By the formulas for the new variables one sees that, choosing an appropriate order, 
			the new basis is uni-triangular with respect to the old one. 
			
			With respect to this basis, the complex becomes the one illustrated in Figure \ref{fig_B6seconda}.
			\begin{figure}
				\centering
				\begin{tikzpicture}[x=2.5cm,scale=0.9]
					\node (a6) at (0,0) {$\gamma_6^+$};\node (b6) at (1,0) {$\gamma_6^-$};
					\node (a42) at (1,-2) {$\gamma_4^+\gamma_2^+$};\node (a4b2) at (2,-1) {$\gamma_4^+\gamma_2^-$};\node (b4a2) at (2,-3) {$\gamma_4^-\gamma_2^+$};\node (b42) at (3,-2) {$\gamma_4^-\gamma_2^-$};
					\node (a33) at (1,-5) {$\gamma_3^+\gamma_3^+$};\node (a3b3) at (2,-4) {$\gamma_3^+\gamma_3^-$};\node (b3a3) at (2,-6) {$\gamma_3^-\gamma_3^+$};\node (b33) at (3,-5) {$\gamma_3^-\gamma_3^-$};
					\node (a24) at (1,-8) {$\gamma_2^+\gamma_4^+$};\node (a2b4) at (2,-7) {$\gamma_2^+\gamma_4^-$};\node (b2a4) at (2,-9) {$\gamma_2^-\gamma_4^+$};\node (b24) at (3,-8) {$\gamma_2^-\gamma_4^-$};
					\node (a222) at (2,-12) {$\gamma_2^+\gamma_2^+\gamma_2^+$};\node (a22b2) at (3,-10) {$\gamma_2^+\gamma_2^+\gamma_2^-$};\node (a2b2a2) at (3,-12) {$\gamma_2^+\gamma_2^-\gamma_2^+$};\node (b2a22) at (3,-14) {$\gamma_2^-\gamma_2^+\gamma_2^+$};\node (a2b22) at (4,-10) {$\gamma_2^+\gamma_2^-\gamma_2^-$};\node (b2a2b2) at (4,-12) {$\gamma_2^-\gamma_2^+\gamma_2^-$};\node (b22a2) at (4,-14) {$\gamma_2^-\gamma_2^-\gamma_2^+$};\node (b222) at (5,-12) {$\gamma_2^-\gamma_2^-\gamma_2^-$};
					\node (a5) at (1,-15) {$\gamma_5^+\tibe_1$};\node (b5) at (2,-15) {$\gamma_5^-\tibe_1$};
					\node (a23) at (2,-17) {$\gamma_2^+\gamma_3^+\tibe_1$};\node (a2b3) at (3,-16) {$\gamma_2^+\gamma_3^-\tibe_1$};\node (b2a3) at (3,-18) {$\gamma_2^-\gamma_3^+\tibe_1$};\node (b23) at (4,-17) {$\gamma_2^-\gamma_3^-\tibe_1$};
					\node (a32) at (2,-20) {$\gamma_3^+\gamma_2^+\tibe_1$};\node (a3b2) at (3,-19) {$\gamma_3^+\gamma_2^-\tibe_1$};\node (b3a2) at (3,-21) {$\gamma_3^-\gamma_2^+\tibe_1$};\node (b32) at (4,-20) {$\gamma_3^-\gamma_2^-\tibe_1$};
					\draw[violet] (0.5,0) ellipse (2cm and 1cm);
					\draw[violet] (2,-2) ellipse (3.1cm and 1.4cm);
					\draw[violet] (2,-5) ellipse (3.1cm and 1.4cm);
					\draw[violet] (3.5,-12) ellipse (5cm and 2.7cm);
					\draw[orange] (1.5,-15) ellipse (2cm and 1cm);
					\draw[-latex] (a6) to node[above] {$[6]$} node[below,red] {$\phi_6$} (b6);
					\draw[-latex] (a42) to node[above] {$[2]$} node[below,near end,red] {$\phi_2$} (a4b2);\draw[-latex] (a42) to node[below] {$[4]$} node[above,red,near end] {$\phi_4$} (b4a2);\draw[-latex] (a4b2) to node[above] {$[2]$} node[below,near start,red] {$\phi_2$} (b42);\draw[-latex] (b4a2) to node[below] {$[4]$} node[above,red,near start] {$\phi_4$} (b42);
					\draw[-latex] (a33) to node[above] {$[3]$} node[below,near end,red] {$1$} (a3b3);\draw[-latex] (a33) to node[below] {$[3]$} node[above,red,near end] {$1$} (b3a3);\draw[-latex] (a3b3) to node[above] {$[3]$} node[below,near start,red] {$\phi_3$} (b33);\draw[-latex] (b3a3) to node[below] {$[3]$} node[above,red,near start] {$\phi_3$} (b33);
					\draw[-latex] (a24) to node[above] {$[4]$} node[below,near end,red] {$\phi_4$} (a2b4);\draw[-latex] (a24) to node[below] {$[2]$} node[above,red,near end] {$1$} (b2a4);\draw[-latex] (a2b4) to node[above] {$[2]$} node[below,near start,red] {$1$} (b24);\draw[-latex] (b2a4) to node[below] {$[4]$} node[above,red,near start] {$\phi_4$} (b24);
					\draw[-latex] (a222) to node[above] {$[2]$} node[below,red,near end] {$\phi_2$} (a22b2);\draw[-latex] (a222) to node[above] {$[2]$} node[below,red] {$1$} (a2b2a2);\draw[-latex] (a222) to node[below] {$[2]$} node[above,red,near end] {$\phi_2$} (b2a22);
						\draw[-latex] (a22b2) to node[above] {$[2]$} node[below,red] {$1$} (a2b22);\draw[-latex] (a22b2) to node[above,pos=0.9] {$[2]$} node[below,near end,red] {$1$} (b2a2b2);\draw[-latex] (a2b2a2) to node[above,pos=0.1] {$[2]$} node[below,near start,red] {$\phi_2$} (a2b22);\draw[-latex] (a2b2a2) to node[below,pos=0.1] {$[2]$} node[above,near start,red] {$\phi_2$} (b22a2);\draw[-latex] (b2a22) to node[below,pos=0.9] {$[2]$} node[above,near end,red] {$1$} (b2a2b2);\draw[-latex] (b2a22) to node[below] {$[2]$} node[above,red] {$1$} (b22a2);
						\draw[-latex] (a2b22) to node[above] {$[2]$} node[below,red,near start] {$\phi_2$} (b222);\draw[-latex] (b2a2b2) to node[above] {$[2]$} node[below,red] {$\phi_2$} (b222);\draw[-latex] (b22a2) to node[below] {$[2]$} node[above,red,near start] {$\phi_2$} (b222);
					\draw[-latex] (a5) to node[above] {$[5]$} node[below,red] {$\phi_5$} (b5);
					\draw[-latex] (a32) to node[above] {$[2]$} node[below,near end,red] {$\phi_2$} (a3b2);\draw[-latex] (a32) to node[below] {$[3]$} node[above,near end,red] {$\phi_3$} (b3a2);\draw[-latex] (a3b2) to node[above] {$[2]$} node[below,near start,red] {$\phi_2$} (b32);\draw[-latex] (b3a2) to node[below] {$[3]$} node[above,near start,red] {$\phi_3$} (b32);
					\draw[-latex] (a23) to node[above] {$[3]$} node[below,near end,red] {$\phi_3$} (a2b3);\draw[-latex] (a23) to node[below] {$[2]$} node[above,near end,red] {$\phi_2$} (b2a3);\draw[-latex] (a2b3) to node[above] {$[2]$} node[below,near start,red] {$\phi_2$} (b23);\draw[-latex] (b2a3) to node[below] {$[3]$} node[above,near start,red] {$\phi_3$} (b23);
				\end{tikzpicture}
				\caption{The complex $B_6$ after change(s) of basis.}\label{fig_B6seconda}
			\end{figure}
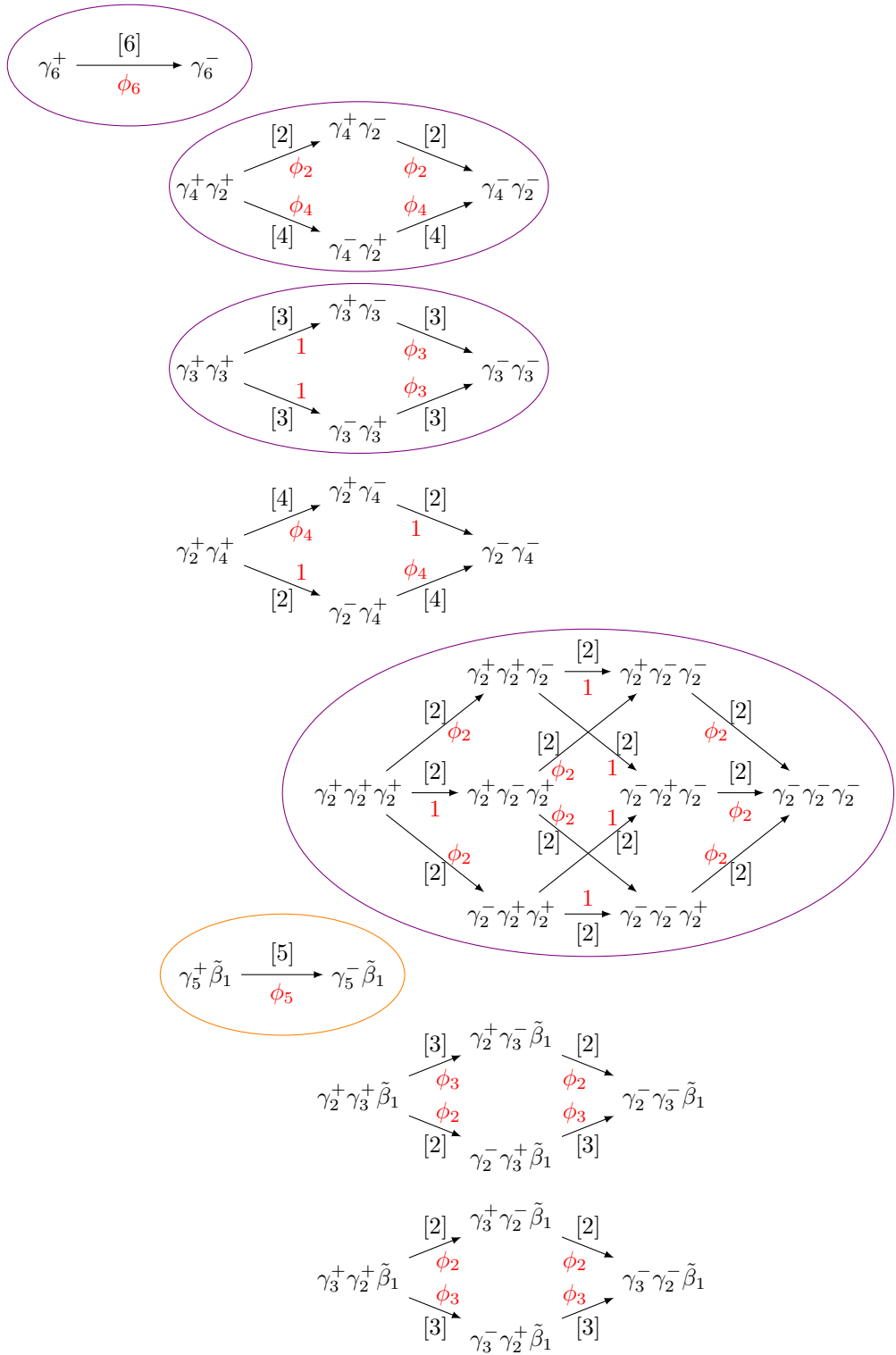
			However the change of basis that we defined is defined over $\mathbb{Q}(x,y)$ 
			but not over $\ZZ[x,y]$. This can be restored by rescaling the new basis elements with
			appropriate factors. More precisely, one could define a change of basis over $\ZZ[x,y]$
			which has the same overall effect on the complex as our change of basis, after rescaling.
			In this case, we will make the following replacements.
			\begin{align*}
				&\gamma_6^-\rightsquigarrow \gamma_6^-\cdot \phi_2\phi_3,&					&\gamma_4^-\gamma_2^+\rightsquigarrow \gamma_4^-\gamma_2^+\cdot\phi_2,&							&\gamma_4^-\gamma_2^-\rightsquigarrow \gamma_4^-\gamma_2^-\cdot\phi_2, \\
				&\gamma_3^+\gamma_3^+\rightsquigarrow \frac{\gamma_3^+\gamma_3^+}{\phi_3},&	&\gamma_2^+\gamma_2^+\gamma_2^+\rightsquigarrow \frac{\gamma_2^+\gamma_2^+\gamma_2^+}{\phi_2},&	&\gamma_2^-\gamma_4^-\rightsquigarrow \gamma_2^-\gamma_4^-\cdot\phi_2, \\
				&\gamma_2^+\gamma_4^+\rightsquigarrow \frac{\gamma_2^+\gamma_4^+}{\phi_2},&	&&					 																			&\gamma_2^+\gamma_2^+\gamma_2^-\rightsquigarrow \frac{\gamma_2^+\gamma_2^+\gamma_2^-}{\phi_2}, \\
				&&																			&&																								&\gamma_2^-\gamma_2^+\gamma_2^+\rightsquigarrow \frac{\gamma_2^-\gamma_2^+\gamma_2^+}{\phi_2}. 	
			\end{align*}
			This has the effect of changing the morphisms in the complex (the new ones are written in red
			in Figure \ref{fig_B6seconda}).
			One can now see that, by Lemma \ref{lem_gcdcyclo}, the only pieces giving non trivial cohomology are the circled ones.
			Those circled in violet give $H_6[10]$ and the one in orange give $H_5[9]$. 
			These are precisely	the last two rows of Table \ref{tab_coh8}.
			
			By applying the same argument to the other pieces $B_i$ for $i=0,\dots, 5$, one finds that
			each $B_i$ with $i\ge 3$ contributes with $H_i[2i-2]\oplus H_{i-1}[2i-3]$, whereas 
			$B_2$ gives $H_2[2]$, $B_1$ gives $H_1[1]$ and $B_0$ gives $H_1$.
		\end{exa}
		Let us now prove the theorem in general.
		\begin{proof}[Proof of Theorem \ref{thm_cohomologyZ}]
			By \eqref{eq_decompositiontildeC}, we can restrict to $B_n$.
			Let us consider the following variables
			\begin{align*}
				&\gamma_1:=\tibe_1,&
				&\gamma_k^+:=\tilde{\beta}_k, &
				&\gamma_k^-:=\sum_{i=1}^{k-1} \frac{\qbc{k}{i}{y}}{\qn{k}{y}}\tilde{\beta}_i\tilde{\beta}_{k-i}.
			\end{align*}
			Then the differential map satisfies \eqref{eq_diffongamma}. Now consider the set
			\[
				\{\gamma_{k_1}^\pm\dots\gamma_{k_r}^{\pm} \mid k_1+\dots+k_r=n\}\cup
				\{\gamma_{k_1}^\pm\dots\gamma_{k_r}^{\pm}\gamma_1 \mid k_1+\dots+k_r=n-1\}.
			\]
			As in the example, with an appropriate ordering, this set is uni-triangular with respect to the 
			original basis, then it defines a new basis.
						
			By \eqref{eq_diffongamma}, fixed a composition $(k_1,\dots, k_r)$ of $n$ with each $k_i\ge 2$,
			the span $B'_{(k_1,\dots,k_r)}$ of all the basis elements 
			\begin{equation}\label{eq_monomialgamma}
				\gamma_{k_1}^{\pm}\dots\gamma_{k_r}^{\pm}
			\end{equation}
			is a summand of $B_n$. In the same way, if $(k_1,\dots, k_r)$ is a composition of $n-1$ 
			the span $B''_{(k_1,\dots,k_r)}$ of all the basis elements
			\[
				\gamma_{k_1}^{\pm}\dots\gamma_{k_r}^{\pm}\gamma_1			
			\]
			is another summand of $B_n$. The complex $B_n$ is the direct sum of all the summand obtained in this way.

			We now have to rescale our basis. Consider a monomial of the form \eqref{eq_monomialgamma}.
			The \emph{$d$-weight} is the maximal number of disjoint submonomials of the form
			$\gamma_d^+\gamma_{kd}^{\pm}$.
			For example $\gamma_2^+\gamma_2^+\gamma_2^+\gamma_4^-$ has $2$-weight equal 2 (it contains $\gamma_2^+\gamma_2^+$ and $\gamma_2^+\gamma_4^-$ 
			and they do not intersect),
			as well as $\gamma_2^+\gamma_2^+\gamma_2^+\gamma_2^+\gamma_4^-$. Instead $\gamma_2^-\gamma_4^+\gamma_2^+$ has 2-weight zero. 
			Given a basis element of the form \eqref{eq_monomialgamma}. For each $d$, 
			let $w_d$ be its $d$-weight and let $r_d$ be the number of occurrences of $\gamma_d^-$. 
			Then we rescale it by the factor
			\[
				\frac{\prod_d \qn{d}{}^{r_d}}{\prod_d \phi_d^{w_d+r_d}}.
			\]
			At this point, consider $B'_{(k_1,\dots,k_r)}$ with the new basis. 
			The following facts are a bit technical but not difficult to prove and are left to the reader:
			\begin{itemize}
				\item if $(k_1,\dots,k_r)$ is not decreasing, then $B'_{(k_1,\dots,k_r)}$ is contractible.
				\item if it is decreasing, but not strictly, say $k_i=k_{i+1}$, then $B'_{(k_1,\dots,k_r)}$ is 
					homotopy equivalent to $B'_({k_1,\dots,\hat{k_i},\dots,k_r)}$.
				\item finally, if it is strictly decreasing, then $B'_{(k_1,\dots,k_r)}$ is isomorphic to
					\[
						(\kk\overset{\phi_{k_1}}{\rightarrow}\kk)\otimes\dots\otimes(\kk\overset{\phi_{k_r}}{\rightarrow}\kk)
					\]
					which gives cohomology $\kk/(\phi_{k_1},\dots,\phi_{k_r})$.
			\end{itemize}
			The same holds for the summands $B''$.

			It remains to prove that it is possible to obtain the same complex with a change of basis defined
			over $\ZZ[x,y]$.
			
			For simplicity we treat the lowest cohomological degree. The argument for the higher ones is essentially the same.
			Here the complex $B_n$ looks as follows:
			\[
				\begin{tikzcd}
					\tibe_n \ar[r,"\alpha"]		& \langle \tibe_1\tibe_{n-1},\tibe_2\tibe_{n-2},\dots,\tibe_{n-1}\tibe_1\rangle \ar[r]	& \dots
				\end{tikzcd}
			\]
			where
			\begin{align*}
				&\alpha=	\begin{pmatrix}
							\qbc{n}{1}{} \\ \qbc{n}{2}{} \\ \vdots \\ \qbc{n}{n-1}{}
						\end{pmatrix}.
			\end{align*}
			Let $d_1<d_2<\dots< d_k$ be the non trivial divisors of $n$ (i.e.\ different from $1$ and $n$).
			By Lemma \ref{lem_factorsbinomial}, the binomial coefficient $\qbc{n}{d_i}{}$ is not divisible by $\phi_{d_i}$.
			Consider
			\begin{align*}
				& a_1=\qn{n}{},&					 							& b_1=\qbc{n}{d_1}{},& 	& z_1=\frac{\qn{n}{}}{\phi_{d_1}}, \\
				& a_2=\frac{\qn{n}{}}{\phi_{d_1}},&  							& b_2=\qbc{n}{d_2}{},&	& z_2=\frac{\qn{n}{}}{\phi_{d_1}\phi_{d_2}}, \\
				& \dots,&														& \dots,&				& \dots,\\
				& a_k=\frac{\qn{n}{}}{\phi_{d_1}\phi_{d_2}\dots\phi_{d_{k-1}}},&& b_k=\qbc{n}{d_k}{},&	& z_k=\frac{\qn{n}{}}{\phi_{d_1}\phi_{d_2}\dots \phi_{d_k}}=\phi_n.
			\end{align*}
			By Lemma \ref{lem_gcdcyclo} we can find, for $i=1,\dots, k$, elements $x_i,y_i\in \kk$
			such that
			\[
				a_ix_i+b_iy_i=z_i,\quad \forall i=1,\dots,k.
			\]
			Now consider the matrices, with coefficients in $\kk$,
			\[
				\begin{pmatrix}
					x_i 				&	&  		& 	&y_i				&	& 		& \\
										& 1	&  		&	&					&	& 		&		\\
										&	&\ddots	& 	&					&	&		&	\\
										&	&		& 1 &					&	& 		&\\
					-\frac{b_i}{z_i}	&	&		&	& \frac{a_i}{z_i}	&	& 		&\\
										&	&		&	&					& 1 & 		&\\
										&	&		&	&					&	&\ddots &\\
										&	&		&	&					&	&		& 1
				\end{pmatrix}.
			\]
			Each of them decomposes, in $\mathbb{Q}(x,y)$, as
			\[	
				\underbrace{%
				\begin{pmatrix}
					1	 				&  		& \frac{y_iz_i}{a_i}&		 \\
										&\ddots	&					&			\\
										&		& 1					& 		\\
										&		&					&\ddots \\
				\end{pmatrix}}_{A_i}
				\underbrace{%
				\begin{pmatrix}
					\frac{z_i}{a_i}		&  		&					& 		 \\
										&\ddots	&					&			\\
										&		& \frac{a_i}{z_i}	& 		\\
										&		&					&\ddots \\
				\end{pmatrix}}_{B_i}
				\underbrace{%
				\begin{pmatrix}
					1	 				&  		&					& 		 \\
										&\ddots	&					&			\\
					-\frac{b_i}{a_i}	&		& 	1				& 		\\
										&		&					&\ddots \\
				\end{pmatrix}}_{C_i}.
			\]
			Now, composing all these matrices we obtain
			\[
				M=A_kB_kC_k\dots A_2B_2C_2A_1B_1C_1.
			\]
			One can check that 
			\[
				M\alpha=\begin{pmatrix}
					\phi_n \\ * \\ \vdots \\ *
				\end{pmatrix}
			\]
			with zeros in the $d_i$-th rows. Now, by Lemma \ref{lem_factorsbinomial}, we can eliminate all the remaining terms 
			via a matrix $C$ with coefficients in $\kk$.
			Now, the matrices $A_i$ can be ``shifted'' to the left, up to some changes that do not involve the 
			first column. 
			More precisely, the product $C_{k+1}M$ becomes
			\[
				A'C_{k+1}B_kC_k\dots B_1C_1,
			\]
			where $A'$ is an invertible matrix of the form
			\[
				A'=\begin{pmatrix}
					1	& *		&\dots 	&*\\
						&\vdots &\ddots &  \vdots\\
						& *		&\dots 	&*
				\end{pmatrix}.
			\]
			On the other hand, if we shift the $C_i$'s to the right we get
			\[
				B_k\dots B_1C_{k+1}'\dots C_1'.
			\]
			Now, the product $B_k\dots B_1$ gives precisely our rescaling, whereas the product 
			$C_{k+1}'\dots C_1'$ gives the first change of basis.
			In particular,
			\[
				C_{k+1}B_kC_k\dots B_1C_k\alpha=
				\begin{pmatrix}
					\phi_n \\ 0 \\ \vdots \\ 0
				\end{pmatrix}.
			\]
			Hence the additional matrix $A'$ has no effect on the above vector.
		\end{proof}		
	\printbibliography
\end{document}